\documentclass[a4paper,11pt]{amsart}

\usepackage{amssymb} 
\usepackage{xcolor} 
\usepackage{fancybox}
\usepackage{dashbox}
\newboolean{showcomments}
\setboolean{showcomments}{true}
\ifthenelse{\boolean{showcomments}}
{ \newcommand{\mynote}[3]{
  \fbox{\bfseries\sffamily\scriptsize#1}
  {\small$\blacktriangleright$\textsf{\emph{\color{#3}{#2}}}$\blacktriangleleft$}}}
{ \newcommand{\mynote}[3]{}}
\definecolor{asparagus}{rgb}{0.53, 0.66, 0.42}

\usepackage{mathrsfs}

  \usepackage{float}
  \usepackage{pgffor}
  \usepackage{tikz}
  \usetikzlibrary{calc,3d,arrows,positioning,shapes.misc, plotmarks, shapes}

  \usepackage[english]{babel}
  \usepackage{amsthm}
  \usepackage{amsmath}
  \usepackage{amssymb}
  \usepackage{mathtools}
  \usepackage[latin1]{inputenc}
  \usepackage{setspace}
  \usepackage{enumerate}
  \usepackage{stmaryrd}
  \usepackage{algpseudocode}
  \usepackage{algorithm}
  \usepackage{multirow}

\definecolor{phase1}{HTML}{377EB8}
\definecolor{phase2}{HTML}{FF7F00}
\definecolor{phase3}{HTML}{4DAF4A}

  \usepackage[colorlinks=true,linkcolor=black,citecolor=black,filecolor=black,urlcolor=black,menucolor=black]{hyperref}

  \usepackage{array}
  \usepackage{graphicx}
  \usepackage{xcolor}
  \usepackage{wrapfig}
  \usepackage[babel]{csquotes}
  \usepackage{geometry}
  \usepackage{bbm}
  \usepackage{stmaryrd}
  \usepackage[all]{xy}
  \usepackage{mathrsfs}
  \usepackage{subcaption}
  \theoremstyle{plain}
  \newtheorem{theorem}{Theorem}
  \newtheorem*{theorem*}{Theorem}
  \newtheorem{proposition}{Proposition}
  \newtheorem{corollary}{Corollary}
  
  \newtheorem{lemma}{Lemma}

  \theoremstyle{definition}
  \newtheorem{definition}{Definition}

  \newtheorem{rem}{Remark}

	\allowdisplaybreaks[3]

  \def \L{\mathcal{L}}
  \def \N{\mathbb{N}}
  \def \R{\mathbb{R}}

  \newcommand {\inj} {\mathop \textup{inj}(M)}

 \renewcommand {\div} {\textup{div}}
 \newcommand {\divrho} {\textup{div}_\mu}
 
 \newcommand {\dist} {\textup{dist}}
 \newcommand {\diam} {\textup{diam}}
 \newcommand {\tr} {\textup{tr}}
 
 \newcommand {\dmx} {\,d\mu(x)}
 \newcommand {\dmy} {\,d\mu(y)}
 
 \newcommand {\Per} {\textup{Per}}

  \newcommand {\eps} {\varepsilon}

\DeclareMathOperator*{\argmax}{\mathop \textup{arg\,max}}
\DeclareMathOperator*{\esssup}{\mathop \textup{ess\,sup}}

\DeclareMathOperator*{\argmin}{\mathop \textup{arg\,min}}

\DeclareMathOperator*{\Vol}{Vol}
 
\DeclareMathAlphabet{\mathbbold}{U}{bbold}{m}{n}
 
\makeatletter
 \def\namedlabel#1#2{\begingroup
    #2%
    \def\@currentlabel{#2}%
    \phantomsection\label{#1}\endgroup
}
\makeatother

\title[Convergence Volume-Preserving MBO]{Weighted, Multiphase, Volume-Preserving Mean Curvature Flow as Limit of the MBO Scheme on Manifolds}
\author{Fabius Kr{\"a}mer}

\address{Fakult{\"a}t f{\"ur} Mathematik, Universit{\"a}t Heidelberg, Im Neuenheimer Feld 205, 69120 Heidelberg, Germany (\texttt{f.kraemer@math.uni-heidelberg.de})}
\address{Fakult{\"a}t f{\"ur} Mathematik, Universit{\"a}t Regensburg, Universit{\"a}tsstra{\ss}e 31, 93053 Regensburg, Germany}
\date{\today}

\begin{document}
    
  \begin{abstract}
  The famous thresholding scheme by Merriman, Bence, and Osher (Motion of multiple junctions: A level set approach. Journal of Computational Physics 112.2 (1994): 334-363.) proved itself as a very efficient time discretization of mean curvature flow. The present paper studies a multiphase, volume constrained version on a weighted manifold that naturally arises in the context of data science. The main result of this work proves convergence to multiphase, weighted, volume-preserving mean curvature flow on a smooth, closed manifold. The proof is only conditional in the sense that convergence of the approximating energy to the weighted perimeter is assumed. These type of assumptions are natural and common in the literature. The proof shows convergence in the energy dissipation inequality, induced by the underlying gradient flow structure, similar to the work of Laux and Otto (The thresholding scheme for mean curvature flow and De Giorgi's ideas for minimizing movements. The Role of Metrics in the Theory of Partial Differential Equations 85 (2020): 63-94.). This leads to a new De Giorgi solution concept that describes weakly the limiting flow. Estimates regarding the derivatives of the heat kernel are developed in order to pass to the limit in the variations of the energy and metric.

	\medskip
    

  \medskip

\noindent \textbf{Mathematical Subject Classification (MSC2020)}:
 53E10, 53Z50, 49Q20 (Primary), 35D30, 58J35, 49Q05 (Secondary)
  \end{abstract}
\maketitle


\tableofcontents

\section{Introduction}
Mean curvature flow traditionally arose in material science, for example as flow describing the grain boundaries of metal in the annealing process \cite{MR1770893}. But nowadays, its natural property to (locally) minimize the perimeter as gradient flow, motivates its use in data science tasks \cite{MR4825221,  6714564,merkurjev2014diffuse, MR3115457}. One prime example for this is the MBO scheme on graphs by Bertozzi et al.~\cite{MR3115457}; it is a space and time discretization of mean curvature flow  based on the famous MBO scheme~\cite{MBO94}. The scheme is remarkable as, although computationally cheap, it inherits the gradient flow structure to iteratively minimize (locally), over all subsets $\Omega$ of the data, the energy \begin{equation*}
  \sum_x \sum_y p(h,x,y) \mathbbold{1}_\Omega (x) \mathbbold{1}_{\Omega^c}(y).
\end{equation*}
The heat kernel $p(h,x,y)$ depends on the weights of the underlying graph and is an over all paths from node $x$ to node $y$ averaged weight. Therefore, a minimizing set $\Omega$ of the energy minimizes the by the heat kernel given weights of edges that go from the set $\Omega$ to its complement $\Omega^c$. Remember, the classical partition problem aims for  minimizing the sum of weights from $\Omega$ to its complement. Hence, the MBO scheme solves, locally and up to reweighing, the partition problem.

Empirical studies \cite{jacobs2018auction} have shown that the scheme has even better accuracy in semi-supervised clustering tasks if one adds a volume constraint to the scheme. Often the number of data points in every set is (roughly) known in advance such that one can predescribe it in the algorithm. The aim of this paper is to show that this volume-constrained MBO scheme on graph is a space and time discretization of volume-preserving mean curvature flow. To that aim Laux and the author \cite{kramer2024efficient} have already proven that the scheme converges to its mean field limit if the number of data points goes to infinity. This result is obtained by using the $\Gamma$-convergence \cite{laux2021large} of a suitable rescaling of the previously mentioned energy together with convergence of the Lagrange multipliers \cite{kramer2024efficient} of the volume constraints to obtain convergence of the iterates as minimizers in the minimizing movement interpretation of the MBO scheme \cite{esedog2015threshold}. The setting for the big-data limit is given by the ``manifold assumption''; it states that the data points are sampled from a probability measure that is concentrated on a lower-dimensional manifold. The idea behind the assumption is that although the data may be embedded in a high dimensional space, points in the same class or cluster are similar and thus live in a lower dimensional structure. So far the mean field limit and its partition properties are not well understood. One is especially interested in its long-time behavior, as this describes approximately the properties of the result of the discrete clustering algorithm mentioned in the beginning.  

This paper is filling this gap by proving in the main Theorem \ref{the:main} that the mean field limit of the multiphase volume constrained MBO scheme on the weighted data-manifold (see \eqref{eq:volumeMBOscheme} below) conditionally convergences to multiphase, weighted, volume-preserving mean curvature flow on the manifold. The evolution law of the limiting flow on the interface of phase $i$ is (formally) given by
\begin{equation*}
  V_i = - H_i - \langle \mathbf{n}_i, \nabla \log(\rho) \rangle +  \Lambda_i
\end{equation*}
where $V_i$ is the (outer) normal velocity, $H_i$ is the mean curvature, $\mathbf{n}_i$ is the outer unit normal on the boundary of phase $i$, respectively, and $\rho$ is the probability density on the manifold while $\Lambda_i$ is the Lagrange multiplier ensuring the preservation of volume.
The parameter for the convergence is the diffusion time $h$ of the heat kernel (compare with the definition of the MBO scheme \eqref{eq:volumeMBOscheme}), which is actually the step-size for the time discretization of the gradient flow. The setting yields four major difficulties: The lack of the comparison principle as there are more than two phases, the curvature of the manifold, the inhomogeneity due to the weight, and the Lagrange multiplier arising from the volume constraint.

There have been already several works regarding the convergence of the MBO scheme to mean curvature flow \cite{MR1324298,  MR4803721,MR1237058,MR1674750,MR4358242,laux2023large,MR3556529,MR4056816,MR4385030,laux2017convergence,  MR3764921}. The only work on a \emph{weighted manifold} so far is by Laux and Lelmi \cite{laux2023large} where they use the powerful concept of viscosity solutions for mean curvature flow. Viscosity solutions are based on the comparison principle for mean curvature flow which is only valid for the case of two phases. The present paper is the first that proves the convergence of the scheme in the \emph{multiphase} case on manifolds. Therefore, it introduces a new distributional solution concept for (weighted, multiphase,) volume-preserving mean curvature flow on a manifold. This weak solution idea is a natural extension of the ones introduced in \cite{MR4358242, MR4385030} and based on De Giorgi's energy dissipation inequality for gradient flows. That is, one encodes all information of the gradient flow equation $\partial_t u = - \nabla E(u)$ in the inequality $E(u(0)) + \frac{1}{2} \int_0^T |\partial_t u(t)|^2 + |\nabla E(u(t))|^2 \, dt \leq E(u(T))$ stating sharp energy dissipation. In our case, the energy reads as the weighted perimeter, i.e., $E(u) = \int_M \rho |\nabla u|$ where $M$ is the manifold, $\rho$ is the weight and $|\nabla u|$ is the total variation of the Gauss-Green measure of $u$. The $L^2$-gradient (on the boundary and weighted by $\rho$) of the weighted perimeter used in Definition \ref{def:de_giorgi_sol} is, at least formally (one can make this rigorous by looking at inner variations, see \cite[Chapter 17]{MR2976521}), $\nabla E(\mathbbold{1}_\Omega) =  H + \langle \mathbf{n}, \nabla \log(\rho) \rangle $ where $H$ and $\mathbf{n}$ are the mean curvature and normal of $\Omega$, respectively.

Similar to the weight on the manifold, the work \cite{chiesa2025convergence} regards an MBO scheme for an \emph{inhomogeneous} (anisotropic) surface energy. Furthermore, they consider an obstacle mean curvature flow which analogously to the volume constraint yields a Lagrange multiplier. They prove $\Gamma$-convergence of the corresponding energies. Also in \cite{MR4803721} anisotropic, inhomogeneous mean curvature flow is studied but this time with a minimizing movement scheme in the spirit of Almgren-Taylor-Wang \cite{MR1205983}. The authors show that their discretization converges to a flat flow by combing a H{\"o}lder estimate with suitable bounds to apply the Arcel{\`a}-Ascoli theorem. In a second step they show under a similar assumption as \eqref{ass:energy_conv} that their flat flows are distributional solutions. 

The only proof of dynamics for a \emph{volume-constrained} MBO scheme before this paper is given by~\cite{laux2017convergence}. The key in the proof is an $L^2$-bound of the Lagrange multiplier which yields compactness in the weak $L^2$ topology. As their proof only applies to the case of two phases and in the Euclidean space, the bound had to be extended to the multiphase case on a weighted manifold in \cite{kramer2024efficient}. In \cite{kramer2024efficient} the bound was not used for a convergence proof but for estimating the running time of an algorithm to compute the Lagrange multiplier for the scheme. 

There are also other advances in approximating mean curvature flow. The Allen-Cahn equation, the gradient flow of the Modica-Mortala functional, is a phase field approach to the problem of dealing with the singularities arising in mean curvature flow. There is a wide literature on convergence of the Allen-Cahn equation to mean curvature flow. To just name the two closest to this paper, the fundamental work of Laux and Simon \cite{MR3847750} first proved the convergence in the multiphase, volume-preserving case to distributional solutions while the recent paper of Ganedi, Marveggio, and Stinson \cite{ganedi2025convergence} showed the inhomogeneous case. The later work similarly as in the present proof of Proposition \ref{prop:metricslope} shows the convergence of the first variations of the corresponding energies. 

To tackle the mathematical challenges, the discrete energy dissipation inequality of De Giorgi (see Lemma \ref{lem:dedi}) is used which is a consequence of the minimizing movement interpretation of Esedo\=glu and Otto \cite{esedog2015threshold}. Though, one has to be careful as the Lagrange multiplier makes the energy dependent on the previous iteration. The energy dissipation inequality is a stable and  powerful tool as it allows passing to the limit by just proving lower semicontinuity of the velocity and the curvature term (see Proposition \ref{prop:distance} and \ref{prop:metricslope}). For the proofs, Lemma \ref{lem:conv_densities} is crucial stating the convergence of the energy densities. The statement is framed as convergence of measures on the tangent bundle. This is a natural setup to work in; the heat kernel localizes when the diffusion time decreases. Hence, a tangential variable is necessary to describe the direction and velocity of the blow up. The lemma is an adaptation from the two phase Euclidean version in~\cite{MR4385030} and boils down to proving a localized lower-semicontinuity of the thresholding energy. So far there was only one generalization of \cite{MR4385030} to the vectorial setting; the authors in \cite{MR4358242} treated the multiphase, Euclidean case with arbitrary surface tensions $\sigma_{ij}$ and mobilities $\mu_{ij}$, i.e.\ it holds $V_{ij} = - \sigma_{ij}\mu_{ij}H_{ij}$ on the interface between phase $i$ and $j$. To do that an on De Giorgi's structure theorem based partition of unity is needed to localize on every single interface. The present paper provides a simpler approach for the proof of Lemma \ref{lem:conv_densities} in the multiphase case with constant mobilities and surface tensions that manages to avoid the geometric measure theory machinery. The simpler approach here could be easily adapted to the case of additive surface tensions $\sigma_{ij} = \sigma_i + \sigma_j$ which are induced by energies of the form $\sum_i \sigma_i \Per(\Omega_i)$.

One of the main technical contributions of this paper, that could be also useful for future works, are several estimates around the derivatives of the heat kernel $p(h,x,y)$ and its approximation $G_h(x,y) = \frac{1}{(4\pi h)^{d/2}} e^{-\frac{\dist^2(x,y)}{4h}}$. For example Corollary \ref{cor:mixed_gaussian_bound} states a Gaussian bound for the mixed second derivative of the heat kernel. Also, Lemma \ref{lem:hessian_dist_equal_identity} claims that if the curvature of the manifold is bounded then the Hessian of half the squared distance is up to second order equal to the identity, i.e.
\begin{equation*}
\left|\big\langle \nabla_y^2 \frac{\dist^2(x,y)}{2}(w), v \big\rangle_y  - \langle w,v \rangle_y \right| \lesssim \dist^2(x,y)|v||w|
\end{equation*}
for any $v,w \in T_y M$.
One is interested in the Hessian of the squared distance in the proof of Proposition \ref{prop:metricslope} as there the second derivatives of the approximate heat kernel $G_h$ appear in the second variation of the metric.

The rest of the paper is structured as follows: Chapter \ref{cha:main_results} states the main result and the necessary intermediate results for its proof. Then Chapter \ref{sec:notation} introduces basic notation and preliminaries for calculus on weighted manifold while Chapter \ref{cha:estimates_kernel} proves the above-mentioned estimates regarding the heat kernel. Finally, Chapter \ref{cha:proofs} gives the proof of the compactness of Lemma \ref{lem:Compactness}, the energy convergence of Lemma \ref{lem:conv_densities}, the lower semicontinuity of the metric quotient of Proposition \ref{prop:distance} and the lower semicontinuity of the metric slope of Proposition \ref{prop:metricslope}.
\section{Main Result and Structure of Proof}\label{cha:main_results}
Let a closed, smooth Riemannian manifold $M$ together with a probability measure $\mu = \rho \Vol$ be given, where $\Vol$ denotes the volume measure on $M$.
We assume the bound $\rho > 0$ and denote by $p$ the heat kernel on $M$ (see Chapter \ref{sec:notation} for more information).
Assume an initial partition of $M$ into $P \in \N$ phases given by the characteristic 
functions $\chi^0_1, \dots, \chi^0_P: M \rightarrow \{0,1\}, \sum_{i = 1}^P \chi_i = 1$. To describe the dynamics of the evolving sets $\Omega_i = \{ \chi_i = 1\}$ one updates the characteristic functions 
\begin{equation}\label{eq:volumeMBOscheme}
  \chi^{\ell}_i(x) :=  
  \begin{cases}
    1 & \text{if } i \in \argmax_{j = 1, \dots, P} \left\{ \big(p(h) * \chi_j^{\ell -1}\big)(x) - m^\ell_j \right\}\\
    0 & \text{otherwise},
  \end{cases}
\end{equation}
where the Lagrange multiplier $m^{\ell} \in \R^P$ is chosen such that $\sum_{i = 1}^P m^\ell_i = 0$ and the volume of every phase is preserved, i.e.,
\begin{align*}
  \int_M \chi^\ell \, d\mu = \int_M \chi^{0} \, d\mu.
\end{align*}
The Lagrange multiplier $m^\ell$ exists and is unique if $h > 0$ and $\int_M \chi^{0} \, d\mu > 0$, 
see \cite{jacobs2018auction,kramer2024efficient} for more information. Furthermore, one defines the piecewise 
constant interpolation 
\begin{align}\label{eq:piecewise_interpolation}
  \chi_h(t) := \chi^\ell \quad \text{if } t \in [\ell h, (\ell + 1)h),\quad \chi_h(t) = \chi^0 \text{ if } t \leq 0.
\end{align}
The main result of this paper states the convergence of this interpolation of the volume constrained MBO scheme to weighted, volume-preserving mean curvature flow. 
Here the following weak solution concept is used for the flow. 
\begin{definition}[De Giorgi Solution MCF]\label{def:de_giorgi_sol}
  A function $\chi: M \times (0,T) \rightarrow \{0,1\}^P$ with $\sum_i \chi_i = 1 \ \mu$-a.e.\ is called a De Giorgi solution to weighted, volume-preserving mean curvature flow if it satisfies the following three conditions a), b) and c):
  \begin{enumerate}[a)]
    \item For every $i = 1, \dots, P$ there exists $H_i \in L^2(|\nabla \chi_i| \,dt)$ which is the mean curvature in the 
    weak sense, i.e., for every $\xi \in C^\infty((0,T), \Gamma(TM))$ it holds
    \begin{align}\label{def:mean_curvature}
      \sum_{i = 1}^P \int_{M\times(0,T)} \big(\div\,\xi - \langle \mathbf{n}_i, \nabla_{\mathbf{n}_i} \xi \rangle\big) |\nabla
      \chi_i|(t) \,dt = - \sum_{i = 1}^P \int_{M \times (0,T)} H_i \langle\mathbf{n}_i, \xi \rangle |\nabla
      \chi_i|(t) \,dt.
    \end{align} 
    \item For every $i = 1, \dots, P$ there exists normal velocities $V_i \in L^2(|\nabla \chi_i| \,dt)$ such that for
    every $\eta \in C_c^\infty(M \times [0,T))$ one has
    \begin{align}\label{def:velocity}
      \int_M \eta(t = 0)\chi_i^0 \,d\Vol + \int_{M \times (0,T)}\partial_t \eta\, \chi_i \, 
      d\Vol \, dt + \int_{M \times (0,T)} \eta V_i \, |\nabla \chi_i| \, dt = 0.
    \end{align}
    \item De Giorgi's inequality is satisfied, i.e.,
    \begin{align}
      \limsup_{\tau \downarrow 0} &\frac{1}{\tau}\sum_{i = 1}^P \int_{(T-\tau,T)} \rho \,|\nabla \chi_i|(t) \,dt
      \nonumber\\+& \frac{1}{2}\sum_{i = 1}^P  \int_{M \times (0,T)} \rho\left(V_i^2 + \left(H - \Lambda_i + 
      \langle \mathbf{n}_i, \nabla \log \rho \rangle\right)^2 \right) \,|\nabla \chi_i| \, dt
      \leq \sum_{i = 1}^P \int_M \rho \, |\nabla \chi_i|. \label{eq:deGiorgi}
    \end{align}
    with 
    \begin{equation}\label{def:lagrange_multi}
          \Lambda_i =  \frac{\int_M ( H + \langle \mathbf{n}_i, \nabla \log \rho \rangle ) \rho |\nabla \chi_i|}{\int_M \rho |\nabla \chi_i|}.
    \end{equation}
  \end{enumerate}
  
\end{definition}

It follows a remark showing that Definition \ref{def:de_giorgi_sol} is meaningful.
\begin{rem} Definition \ref{def:de_giorgi_sol} is a natural extension of De Giorgi solutions of mean curvature flow introduced in \cite{MR4358242, MR4385030}. In a weak sense, equation \eqref{def:mean_curvature} enforces the Herring angle conditions which together with De Giorgi's inequality \eqref{eq:deGiorgi} imply the dynamics $ V_i = H_i - \Lambda_i + \langle \mathbf{n}_i, \log \rho\rangle$. Moreover, the dynamics together with definition of the velocities \eqref{def:velocity} and the Lagrange multiplier $\Lambda$ yield the conservation of mass $\int_M \chi(t_1) d\mu = \int_M \chi(t_2) d\mu$ for all $t_1, t_2 \in [0,T]$. To see this assume that \eqref{def:mean_curvature}, \eqref{def:velocity}  and \eqref{eq:deGiorgi} are satisfied for $\chi: M \times [0,T] \rightarrow \{0,1\}^P$ with $\sum_i \chi_i(t) = 1$ such that the sets $\Omega_i(t) = \{\chi_i(\cdot, t)=1\}$ have smoothly evolving interfaces $\Sigma_{ij} := \partial \Omega_i \cap \partial \Omega_j$.
\begin{enumerate}
  \item The Herring angle condition is a consequence of \eqref{def:mean_curvature}. To see this, apply the divergence theorem on surfaces (cf. \cite[Chapter 2.7]{MR756417}) to a test function $\xi \in C_c^\infty(M)$
  \begin{equation*}
    \int_{\Sigma_i} \div \,\xi - \langle \mathbf{n}_i, \nabla_{\mathbf{
      n}_i} \xi \rangle  \, d \mathcal{H}^{d-1} = - \int_{\Sigma_i}  H_i \langle \xi, \mathbf{n}_i \rangle \, d \mathcal{H}^{d-1} - \int_{\partial \Sigma_i} \langle \xi, \eta_i \rangle \,d \mathcal{H}^{d-2}
  \end{equation*} 
  where $\eta_i$ is the inward pointing unit conormal, i.e., tangent to $\Sigma_i$ and normal to $\partial \Sigma_i$. For example, when denoting with $J$ the rotation by ninety degrees in the normal plane, one has $\eta_i = J \mathbf{n}_i$ at any triple junction. So, equality \eqref{def:mean_curvature} and the above divergence theorem on surfaces imply 
  \begin{align*}
    0 = \int_{\partial \Sigma_1} \langle \xi , J \mathbf{n}_1 \rangle \,d \mathcal{H}^{d-2}+     \int_{\partial \Sigma_2} \langle \xi , J \mathbf{n}_2 \rangle \,d \mathcal{H}^{d-2}
+    \int_{\partial \Sigma_3} \langle \xi , J \mathbf{n}_3 \rangle \,d \mathcal{H}^{d-2},
  \end{align*}
  such that at triple junctions it has to hold the Herring angle condition $\mathbf{n}_1 +\mathbf{n}_2+\mathbf{n}_3 = 0$.

  \item The dynamics $ V_i = -  H_i + \Lambda_i - \langle \mathbf{n}_i, \nabla \log \rho \rangle$ are satisfied with $\Lambda_i$ as in \eqref{eq:deGiorgi}. This can be seen because the De Giorgi inequality \eqref{eq:deGiorgi} reduces in the smooth case to
  \begin{equation*}
    \sum_{i=1}^P \int_0^T \frac{d}{dt} \int_M \rho |\nabla \chi_i|  \,dt
    + \frac{1}{2} \sum_{i = 1}^P \int_{M \times (0,T)} \left(V_i^2 + (H_i - \Lambda_i + \langle\mathbf{n}_i, \nabla\log \rho\rangle )^2 \right) \rho |\nabla \chi_i| \, dt \leq 0 .
  \end{equation*}
  Together with the Herring angle condition one checks that 
  \begin{equation*}
    \sum_{i = 1}^P \frac{d}{dt} \int_M \rho |\nabla \chi_i(t) | = \sum_{i = 1}^P \int_M \rho V_i (H_i + \langle \mathbf{n}_i ,\nabla \log(\rho) \rangle) |\nabla \chi_i|
  \end{equation*}
    which together with $ \sum_{i=1}^P\int_M \rho V_i \Lambda_i |\nabla \chi_i| = 0$ allows us to complete the square 
  \begin{equation*}
    \sum_{i = 1}^P \int_{M \times (0,T)} \rho (V_i + (H_i - \Lambda_i + \langle \mathbf{n}_i, \nabla \log(\rho )\rangle))^2 |\nabla \chi_i| \, dt \leq 0.
  \end{equation*}
  Hence, the dynamics $V_i = - H_i + \Lambda_i - \langle \mathbf{n}_i, \nabla \log(\rho) \rangle$ hold on the boundary of $\Omega_i$.
  \item The volume is preserved, i.e., for any $T > t \geq 0$ it holds $\int_M \chi_i(t) \, d\mu = \int_M \chi_i(0) \,d\mu$. This follows immediately from the dynamics $V_i = - H_i + \Lambda_i - \langle \mathbf{n}_i, \nabla \rho \rangle$ together with the definitions \eqref{def:velocity} and \eqref{def:lagrange_multi} of $V_i$ and $\Lambda_i$, respectively, such that
  \begin{align*}
    \int_M \chi_i(t) \rho\, d\Vol - \int_M \chi_i(0) \rho \,d \Vol &= \int_{M \times (0,T)}  \rho V_i | \nabla \chi_i| \, dt \\
    &=  -\int_{M \times (0,T)} \rho (H_i - \Lambda_i + \langle \mathbf{n}_i, \nabla \rho \rangle ) |\nabla \chi_i| \, dt\\
    &=0.
  \end{align*}
\end{enumerate}
\end{rem}
The main result of this paper states that the volume constrained MBO scheme converges to a De Giorgi solution to volume-preserving mean curvature flow given by Definition \ref{def:de_giorgi_sol}.
\begin{theorem}[Main Theorem]\label{the:main}
  Given $\chi^0 = (\chi_1^0, \dots, \chi_P^0)$ with $\chi_i^0:M \rightarrow \{0,1\}$, $\sum_{i = 1}^P \chi_i^0 = 1$ $\mu$-a.e.\ and such that $|\nabla \chi_i^0|$ is a bounded measure for every $i = 1, \dots, P$. For any sequence $h \downarrow 0$ the interpolation $\chi_h$ is defined by \eqref{eq:volumeMBOscheme} and \eqref{eq:piecewise_interpolation}. Assume there exists a limit $\chi: M \times (0,T) \rightarrow [0,1]^P$ such that 
  \begin{equation}\label{eq:assumption_weakl1}
    \chi_h \rightharpoonup \chi \quad \text{in } L^1(\mu \times dt).
  \end{equation}
  Then it holds $\chi \in \{0,1\}^P, \sum_{i = 1}^P \chi_i = 1\ \mu$-a.e.\ and $\nabla \chi_i$ is a bounded measure which is equi-integrable in the time variable. If additionally one supposes that
  \begin{equation}\label{ass:energy_conv}
    \limsup_{h \downarrow 0} \int_0^T E_h(\chi_h(t)) \, dt\leq c_0 \sum_{i = 1}^P \int_{(0,T) \times M} 
    \rho(x) \, |\nabla \chi_i(x)| \,dt
  \end{equation}
  with $c_0 := \frac{1}{\sqrt{\pi}}$ holds, then $\chi$ is a De Giorgi solution to mean curvature flow in the sense of Definition~\ref{def:de_giorgi_sol}.
\end{theorem}
Similar to the works \cite{MR4358242,MR3556529,MR4056816,MR4385030}, which have been motivated by the work of Luckhaus and Sturzenhecker \cite{MR1386964}, Theorem \ref{the:main} is conditional and requires \eqref{ass:energy_conv} to prevent the loss of interface area. Only the contrary inequality 
  \begin{equation}\label{eq:lower_gamma_bound}
    \liminf_{h \downarrow 0} \int_0^T E_h(\chi_h(t)) \, dt \geq c_0 \sum_{i = 1}^P \int_{(0,T) \times M} 
    \rho(x) \, |\nabla \chi_i(x)| \,dt
  \end{equation}
  is true due to the lower semicontinuity of the energy proven in \cite{laux2021large}. In Lemma \ref{lem:conv_densities} an alternative, simpler proof will be given for a localized version of \eqref{eq:lower_gamma_bound}.

To access the proof of Theorem \ref{the:main} the following  interpretation of the MBO scheme as a minimizing movement scheme, due to \cite{esedog2015threshold}, opens a path to use variational techniques. For a metric space $(\mathcal{M}, d)$, an energy $E:\mathcal{M}\rightarrow \R$, a time step-size $h > 0$ and initial conditions $u^0 \in \mathcal{M}$, the minimizing movement scheme is given by iterating for $\ell = 1, 2, \dots$
\begin{equation}\label{eq:mms_classic}
  u^{\ell} \in \argmin_{u \in \mathcal{M}} E(u) + \frac{1}{2h} d^2(u, u^{\ell - 1}).
\end{equation}
This scheme is the standard discretization of a gradient flow. As mean curvature flow is the $L^2$-gradient flow of the perimeter, it shouldn't be of any surprise that the minimizing movement interpretation of the MBO scheme minimizes over an energy that approximates the perimeter (cf.\ Lemma \ref{lem:mmi}).

As here the volume constrained MBO scheme is under consideration the term incorporating the Lagrange multiplier has to be included in the minimizing movement interpretation. Therefore, the scheme is adapted by making the energy depend on the iteration which lets us hide the Lagrange multiplier in the energy, i.e.,
\begin{equation}\label{def:mms}
    u^{\ell} \in \argmin_{u \in \mathcal{M}} E^\ell(u) + \frac{1}{2h} d(u, u^{\ell - 1}).
\end{equation}
The minimizing movement interpretation is now the following lemma.
\begin{lemma}[Minimizing Movement Interpretation]\label{lem:mmi}
  The scheme \eqref{eq:volumeMBOscheme} satisfies \eqref{def:mms} provided one defines
  \begin{align}
    \mathcal{M} &:= \left\{u: M \rightarrow [0,1]^P \ measurable \Big| \sum_{i = 1}^P u_i = 1 \ \mu-a.e. \right\},\\
    \frac{1}{2h}d_h^2(u, \tilde{u}) &:= \frac{1}{\sqrt{h}} \sum_{i = 1}^P \int_{M^2} 
    p(h,x,y) (u_i -u_j)(x) (u_i -u_j)(y)  \dmy \dmx \label{def:dist_h},\\
    E_h^\ell(u) &:= E_h(u) + \Lambda^\ell \cdot \int_M u \, d\mu,
  \end{align}
  where $\Lambda^\ell = \frac{2}{\sqrt{h}} m^\ell$ is the rescaled Lagrange multiplier of \eqref{eq:volumeMBOscheme} and 
  \begin{equation}
  E_h(u) := \frac{1}{\sqrt{h}} \sum_{i = 1}^P \int_{M}\int_M p(h,x,y) u_i(x) (1- u_i(y))
    \dmy \dmx \label{def:energy_h}
  \end{equation}
  is the classical thresholding energy. Furthermore, $(\mathcal{M}, d_h)$ is a compact metric space and $E_h^\ell$ continuous.
\end{lemma}

Having a minimizing movement scheme \eqref{eq:mms_classic} at hand, the maybe naive approach would be to plug in $u^{\ell -1}$ as competitor. This leads first to 
\begin{equation}\label{eq:triv_inequality}
  E(u^\ell) + \frac{1}{2h}d^2(u^\ell, u^{\ell-1}) \leq E(u^{\ell-1})
\end{equation}
and after successively application then to $E(u^L) + \frac{h}{2} \sum_{\ell = 1}^L \left(\frac{d(u^\ell, u^{\ell -1})}{h}\right)^2 \leq E(u^0) $. This is the right type of equation in regard of \eqref{eq:deGiorgi} but a factor $2$ is missing as at least heuristically $\frac{d_h(\chi_h^\ell, \chi_h^{\ell -1})}{h} \rightarrow |\partial_t \chi|$. To solve this problem, another interpolation than the constant is introduced. The more intrinsic  
\begin{equation}\label{def:variation_interpolation}
  u(t) \in \argmin_{u \in \mathcal{M}} E^\ell(u) + \frac{1}{2(t- (\ell - 1)h)}d^2(u, u^{\ell-1}) \quad \text{for } t \in ((\ell -1 )h, \ell h]
\end{equation}
is called the ``variational interpolation''. After introducing \[e(t):= E^\ell(u(t)) + \frac{1}{2(t - (\ell-1)h)}d^2(u(t), u^{\ell -1 })\] and using that $u(t)$ is a critical point of $E^\ell + \frac{1}{2(t - (\ell-1)h)}d^2(\cdot, u^{\ell - 1})$ it holds on a heuristic level
\begin{equation*}
  \frac{d}{dt} e(t) = - \frac{1}{2(t- (\ell -1 )h)^2} d^2(u(t), u^{\ell -1 }).
\end{equation*}
Integrating in time from $(\ell -1)h$ to $\ell h$ yields
\begin{equation}\label{eq:heuristic_dedi}
  E^\ell(u^\ell) + \frac{1}{2h} d^2(u^\ell, u^{\ell -1}) + \frac{1}{2}\int_{h(\ell - 1)}^{h\ell} \frac{1}{2(t- (\ell- 1)h)^2} d^2(u(t), u^{\ell -1}) \, dt = E^\ell(u^{\ell - 1})
\end{equation}
which is exactly \eqref{eq:triv_inequality} but with the previously missing gradient term. To see this one uses heuristically the ``Euler--Lagrange equation'' of the definition of the variational interpolation~\eqref{def:variation_interpolation} to recognize the extra term is given by
\begin{equation*}
  \frac{1}{2}\int_{h(\ell - 1)}^{h\ell} \frac{1}{2(t- (\ell- 1)h)^2} d^2(u(t), u^{\ell -1}) \, dt = \frac{1}{2}\int_{h(\ell - 1)}^{h\ell} |\nabla E^\ell|^2 (u(t)) \, dt
\end{equation*}
which will later pick up the $H + \Lambda + \langle \mathbf{n}, \nabla \rho \rangle$ term of \eqref{eq:deGiorgi}.

The rigorous results for the previously introduced ideas are given in the next lemma where one needs a extra property of the scheme \eqref{eq:volumeMBOscheme} to be able to iterate \eqref{eq:heuristic_dedi}; namely, $E(\chi^\ell) - E(\chi^{\ell-1}) = E^\ell(\chi^\ell) - E^\ell(\chi^{\ell-1})$ holds as the volume of the iterates stays constant.
\begin{lemma}[Discrete Energy Dissipation Inequality]\label{lem:dedi}
  Let $(\mathcal{M}, d)$ be a metric space and $\{E^\ell:\mathcal{M}\rightarrow \R\}_{\ell \in \N}$ be continuous.
  Given $\chi^0 \in \mathcal{M}$ and $h > 0$ consider a sequence $\{\chi^\ell\}_{\ell\in\N}$ satisfying
  \begin{equation}\label{eq:min_move}
    \chi^\ell \text{ minimizes } \frac{1}{2h}d^2(u, \chi^{\ell-1}) + E^\ell(u) \quad \text{ among all } u \in \mathcal{M}.
  \end{equation}
  Then we have for all $t = \ell h$ with $ \ell \in \N$
  \begin{equation} \label{eq:fdedi}
    E^\ell(\chi(t)) + \frac{1}{2}\int_{t-h}^t \left(\frac{1}{h^2}d^2(\chi(s+h), \chi(s) ) +
    |\partial E^\ell(u(s))|^2 \right) ds \leq E^\ell(\chi(t-h)).
  \end{equation}
  Here $\chi(t)$ is the piecewise constant interpolation, cf.~\eqref{eq:piecewise_interpolation},
  $u(t)$ is the variational interpolation~\eqref{def:variation_interpolation} satisfying 
  \begin{align}
    d^2(u(t), \chi^{\ell-1}) \leq d^2(\chi^\ell, \chi^{\ell-1})  \quad \text{ for all } t \in ((\ell-1)h, \ell h], \label{eq:compare_dist}\\
    E^\ell(u(t)) \leq E^\ell(\chi(t)) \quad \text{ for all } t \in ((\ell-1)h, \ell h],\label{eq:vi_prop}
  \end{align}
  and $|\partial E|$ is the ``metric slope'' defined through
  \begin{equation}
    |\partial E|(u) := \limsup_{v:d(v,u) \rightarrow 0} \frac{(E(u) - E(v))_+}{d(u,v)} \in [0, \infty].
  \end{equation}

  In particular if for some energy $E$ it holds $E(\chi^\ell) - E(\chi^{\ell-1}) = E^\ell(\chi^\ell) - E^\ell(\chi^{\ell-1})$
  then it additionally holds for all $t \in \N h$
  \begin{align}
    E(\chi(t)) +  \sum_{\ell = 1}^{\frac{t}{h}} \frac{1}{2}\int_{(\ell-1)h}^{\ell h} \left(\frac{1}{h^2}d^2(\chi(s+h), \chi(s) ) +
    |\partial E^\ell(u(s))|^2 \right) ds &\leq E(\chi^0), \label{eq:dedi}\\
    \int_0^\infty \frac{1}{2h^2} d^2(u(t), \chi(t)) \,dt &\leq E(\chi^0). \label{eq:bound_dist}
  \end{align} 
\end{lemma}

Combing Lemma \ref{lem:mmi} and Lemma \ref{lem:dedi} the idea is to pass to the limit $h \downarrow 0$ in the discrete energy dissipation inequality \eqref{eq:dedi}. For that, convergence of the interpolations is needed in some topology. It turns out $L^1(M \times (0,T))$-convergence is sufficient for that and ensured by the following compactness result.

\begin{lemma}[Compactness]\label{lem:Compactness}
 Let $\{\chi_h\}_{h \downarrow 0}, \sum_{i = 1}^P \chi_i = 1$ a.e., be a sequence of functions on $(0,T) \times M$ with values
  in $\{0,1\}^P$, having the property of being piecewise constant as in \eqref{eq:piecewise_interpolation}
  and 
  \begin{align}\label{eq:ass_compactness}
     \esssup_{t \in (0,T)} E_h(\chi_h(t)) + \int_0^T \frac{1}{2h^2} d_h^2(\chi_h(t), \chi_h(t-h))
     \, dt \quad \text{stays bounded as } h \downarrow 0.
  \end{align}
  Then $\{\chi_h\}_{h \downarrow 0}$ is pre-compact in $L^1((0,T) \times M)$ and furthermore the 
  Gauss-Green measure $\nabla \chi$ of any (weak) limit $\chi$ is bounded, equi-integrable in $t$
  with 
  \begin{equation}\label{eq:lowersemicontinuity_energy}
    c_0 \int_{(0,T) \times M} \rho \, |\nabla \chi| \,dt \leq \liminf_{h \downarrow 0} \int_0^T 
    E_h(\chi_h(t)) \, dt.
  \end{equation}
\end{lemma}


\usetikzlibrary {fadings,patterns}
\begin{tikzfadingfrompicture}[name=fade right with circle]
  \shade[left color=transparent!0,
         right color=transparent!100] (0,0) rectangle (2,2);
  \fill[transparent!50] (1,1) circle (0.7);
\end{tikzfadingfrompicture}
 
\tikzset{
pattern size/.store in=\mcSize, 
pattern size = 5pt,
pattern thickness/.store in=\mcThickness, 
pattern thickness = 0.3pt,
pattern radius/.store in=\mcRadius, 
pattern radius = 1pt}
\makeatletter
\pgfutil@ifundefined{pgf@pattern@name@_nmqhmj3uw}{
\pgfdeclarepatternformonly[\mcThickness,\mcSize]{_nmqhmj3uw}
{\pgfqpoint{0pt}{0pt}}
{\pgfpoint{\mcSize+\mcThickness}{\mcSize+\mcThickness}}
{\pgfpoint{\mcSize}{\mcSize}}
{
\pgfsetcolor{\tikz@pattern@color}
\pgfsetlinewidth{\mcThickness}
\pgfpathmoveto{\pgfqpoint{0pt}{0pt}}
\pgfpathlineto{\pgfpoint{\mcSize+\mcThickness}{\mcSize+\mcThickness}}
\pgfusepath{stroke}
}}
\makeatother
\tikzset{every picture/.style={line width=0.75pt}} 

\begin{figure}
\begin{tikzpicture}[x=0.75pt,y=0.75pt,yscale=-1,xscale=1]

\draw  [fill={rgb, 255:red, 126; green, 211; blue, 33 }  ,fill opacity=1 ] (206.45,218.35) .. controls (176.45,212.35) and (126.45,91.35) .. (207.45,75.35) .. controls (288.45,59.35) and (351.45,50.35) .. (397.45,51.35) .. controls (443.45,52.35) and (532.45,141.35) .. (502.45,159.35) .. controls (472.45,177.35) and (506.45,223.35) .. (438.45,220.35) .. controls (370.45,217.35) and (300.45,234.35) .. (275.45,239.35) .. controls (250.45,244.35) and (236.45,224.35) .. (206.45,218.35) -- cycle ;
\draw  [fill={rgb, 255:red, 74; green, 144; blue, 226 }  ,fill opacity=1 ] (135.45,110.35) .. controls (142.45,96.35) and (165.45,86.35) .. (207.45,75.35) .. controls (249.45,64.35) and (228.45,82.35) .. (228.45,94.35) .. controls (209.45,108.35) and (236.45,225.35) .. (201.45,217.35) .. controls (182.45,215.35) and (170.58,202.73) .. (155.45,188.35) .. controls (140.33,173.98) and (152.45,155.35) .. (137.45,150.35) .. controls (122.45,145.35) and (128.45,124.35) .. (135.45,110.35) -- cycle ;
\draw  [fill={rgb, 255:red, 255; green, 255; blue, 255 }  ,fill opacity=1 ] (309.45,136.35) .. controls (331.45,117.35) and (409.45,121.35) .. (422.45,136.35) .. controls (402.45,156.35) and (389.45,157.35) .. (369.45,158.35) .. controls (348.45,158.35) and (328.45,155.35) .. (309.45,136.35) -- cycle ;
\draw    (298.45,125.35) -- (309.45,136.35) ;
\draw    (422.45,136.35) -- (431.45,126.35) ;

\draw  [fill={rgb, 255:red, 245; green, 166; blue, 35 }  ,fill opacity=1 ] (207.45,75.35) .. controls (197.45,76.35) and (359.45,48.35) .. (372.45,52.35) .. controls (385.45,56.35) and (388.45,64.35) .. (408.45,94.35) .. controls (428.45,124.35) and (252.45,94.35) .. (228.45,94.35) .. controls (221.9,83.7) and (217.45,74.35) .. (207.45,75.35) -- cycle ;
\draw  [color={rgb, 255:red, 155; green, 155; blue, 155 }  ,draw opacity=1 ][pattern=_nmqhmj3uw,pattern size=6pt,pattern thickness=0.75pt,pattern radius=0pt, pattern color=gray, opacity=0.35] (192.92,43.72) -- (65.09,120.98) -- (232.84,274.77) -- (360.67,197.51) -- cycle ;
\draw  [dash pattern={on 4.5pt off 4.5pt}] (162.45,159.35) .. controls (164.45,138.35) and (188.45,80.35) .. (226.45,100.35) .. controls (264.45,120.35) and (286.45,146.35) .. (277.45,168.35) .. controls (268.45,190.35) and (239.45,204.35) .. (216.45,206.35) .. controls (193.45,208.35) and (160.45,180.35) .. (162.45,159.35) -- cycle ;
\draw  [fill={rgb, 255:red, 0; green, 0; blue, 0 }  ,fill opacity=1 ] (218,154.73) .. controls (218,153.22) and (219.22,152) .. (220.73,152) .. controls (222.23,152) and (223.45,153.22) .. (223.45,154.73) .. controls (223.45,156.23) and (222.23,157.45) .. (220.73,157.45) .. controls (219.22,157.45) and (218,156.23) .. (218,154.73) -- cycle ;
\draw  [fill={rgb, 255:red, 0; green, 0; blue, 0 }  ,fill opacity=1, path fading=fade right with circle ] (276.75,157.73) .. controls (276.75,156.22) and (277.97,155) .. (279.48,155) .. controls (280.98,155) and (282.2,156.22) .. (282.2,157.73) .. controls (282.2,159.23) and (280.98,160.45) .. (279.48,160.45) .. controls (277.97,160.45) and (276.75,159.23) .. (276.75,157.73) -- cycle ;
\draw    (220.73,154.73) .. controls (248.44,174.58) and (266.79,167.03) .. (277.95,158.88) ;
\draw [shift={(279.48,157.73)}, rotate = 142.12] [color={rgb, 255:red, 0; green, 0; blue, 0 }  ][line width=0.75]    (10.93,-3.29) .. controls (6.95,-1.4) and (3.31,-0.3) .. (0,0) .. controls (3.31,0.3) and (6.95,1.4) .. (10.93,3.29)   ;

\draw [color={rgb, 255:red, 155; green, 155; blue, 155 }  ,draw opacity=1 ]   (220.73,154.73) -- (275.4,196.63) ;
\draw [shift={(276.99,197.85)}, rotate = 217.47] [color={rgb, 255:red, 155; green, 155; blue, 155 }  ,draw opacity=1 ][line width=0.75]    (10.93,-3.29) .. controls (6.95,-1.4) and (3.31,-0.3) .. (0,0) .. controls (3.31,0.3) and (6.95,1.4) .. (10.93,3.29)   ;

\draw (237.67,144.54) node [anchor=north west][inner sep=0.75pt]    {$\gamma _{x,y}$};
\draw (204.42,142.54) node [anchor=north west][inner sep=0.75pt]    {$x$};
\draw (289.92,143.54) node [anchor=north west][inner sep=0.75pt]    {$y$};
\draw (272.42,198.89) node [anchor=north west][inner sep=0.75pt]    {$\textcolor{gray}{z}$};
\draw (77.5,69.65) node [anchor=north west][inner sep=0.75pt]    {$\textcolor{gray}{T_{x} M}$};
\end{tikzpicture}
\caption{Illustration of coordinate change used to describe the blow up in Lemma \ref{lem:conv_densities}.}
\label{fig:coodinate_change}
\end{figure}
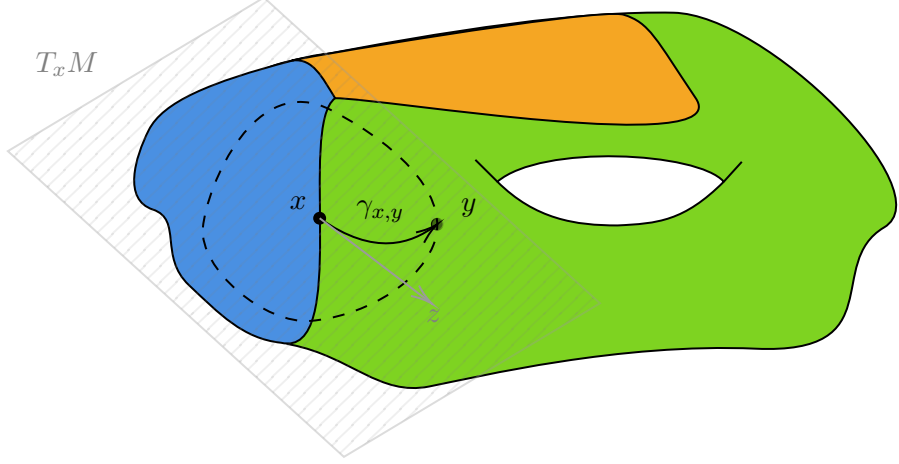

With the aim of passing later to the limit $h \downarrow 0$ in the thresholding energy \eqref{def:energy_h} and its variations, it will be necessary to pass to the limit in the energy densities 
\begin{equation}
      \eta_h(x,y) = \frac{1}{\sqrt{h}}\mathbbold{1}_{\dist(x,y)\leq \inj}p(h,x,y) u_{h,i}(x,t)(1-u_{h,i})(y,t).
\end{equation}
But as the heat kernel convergences (distributionally) to the delta distribution, i.e, $p(h,x,y) \rightarrow \delta_x(y)$ for $h \downarrow 0$, the variables $x$ and $y$ collapse in the limit. Hence, it is necessary to change the coordinates to $x$ and $z := \frac{\exp^{-1}_x(y)}{\sqrt{h}} \in T_x M$; see Figure \ref{fig:coodinate_change}. The tangential vector $z$ describes the direction and velocity of the blow up  of the heat kernel in its natural scaling $\sqrt{h}$. The map
\begin{align*}
  \Phi_h:\{(x,y) \in M^2: \dist(x,y) \leq \inj\} &\rightarrow TM^2,\\
   (x,y) &\mapsto \big((x,\frac{\exp_x^{-1}(y)}{\sqrt{h}}), (y,\frac{\exp_y^{-1}(x)}{\sqrt{h}})\big)
\end{align*}
describes exactly the coordinate change where the blow up in the $y$ variable is also tracked for technical convenience in later steps. The following lemma states now the convergence of the densities in the sense of weak convergence of the measure 
\begin{equation}\label{def:measure_energy}
  \mathfrak{m}_h = {\Phi_h}_{\#}\big(\eta_h(x,y) \dmy \dmx \big) \,dt.
\end{equation}
The limiting measure is similarly described by 
\begin{equation}\label{def:measure_limit}
    \mathfrak{m} = {\Phi}_{\#}\big(  \eta(z,x) \,dz |\nabla \chi_i|(x) \big) \,dt,
\end{equation}
where $dz$ denotes the Lebesgue measure on the tangent space $T_x M$, while $\Phi$ is just for technical reasons to track the blow up in the $x$ and $y$ components 
\begin{equation*}
    \Phi: TM \rightarrow TM^2, (x,z) \mapsto \big((x,z),(x, -z)\big)
\end{equation*}
and $\eta$ is the limiting density in the $x$ and $z$ variable 
\begin{align*}
  \eta(x,z) =  G(z) \left(\langle \mathbf{n}_i(x), z \rangle_x \right)_+  \rho(x).
\end{align*}

\begin{lemma}[Convergence of the Energy Densities]\label{lem:conv_densities}
  Let $\{u_h\}_{h \downarrow 0}$ be a sequence of $[0,1]$-valued functions on $(0,T) \times M$
  that convergences weakly in $L^1(M \times (0,T))$ to a $\{0,1\}$-valued function $\chi$, i.e.
  \begin{align}
    u_h \xrightharpoonup{h \rightarrow 0} \chi \quad \text{ in } L^1((0,T) \times M),
  \end{align}
  and \eqref{ass:energy_conv} holds with $\chi_h$ replaced by $u_h$.
  Furthermore, assume
  \begin{align}\label{ass:energy_bound}
    \esssup_{t \in (0,T)} E_h(u_h(t)) \quad \text{stays bounded as } h \downarrow 0.
  \end{align}
  Then the measures defined in \eqref{def:measure_energy} and \eqref{def:measure_limit} converge
  \begin{equation}\label{eq:conv_measure}
     \mathfrak{m}_h \xrightharpoonup{h \downarrow 0} \mathfrak{m} \quad \text{weakly-}* \text{ as measures on }TM^2,
  \end{equation}
  where the growth in the above described tangential variable $z$ can even be polynomial.
\end{lemma}

\begin{rem}
  Unfolding the definitions of $\mathfrak{m}_h$ and $\mathfrak{m}$ yields that the convergence \eqref{eq:conv_measure} is equivalent to 
  \begin{align*}
    &\lim_{h \downarrow 0} \int_0^T \frac{1}{\sqrt{h}}\int_M \int_{B_{\inj}(x)} f \Big(x,y, \frac{\exp^{-1}_x(y)}{\sqrt{h}},  \frac{\exp^{-1}_y(x)}{\sqrt{h}},t\Big) p(h,x,y)u_{h,i}(x,t) 
    (1-u_{h,i})(y,t) \\
    &\hspace{370pt} \times \dmy \dmx \,dt\nonumber\\
    &= \int_0^T \int_M \int_{T_x M}f(x,x,z,-z,t)  G_1(z) \left( \langle \mathbf{n}_i(x), z \rangle_x \right)_+
    \,dz\, \rho(x) |\nabla \chi_i|(x) \,dt
  \end{align*}
  for any test function $f \in C^0(TM^2 \times [0,T])$.
\end{rem}

With the convergence result of Lemma \ref{lem:conv_densities} one is able to pass to the limit in the energy dissipation inequality. The limiting inequality is then, at least morally, the De Giorgi's inequality~\eqref{eq:deGiorgi}. For that one needs Proposition \ref{prop:distance} with its sharp control of the velocity $V_i$ by the difference quotient of the metric $\frac{d_h^2(\chi_h(t+h), \chi_h(t))}{h^2}$. Similarly, Proposition \ref{prop:metricslope} is necessary to relate the metric slope to the gradient of the weighted perimeter, which is the limiting energy of $E_h$. For the compactness of the constant interpolation of the Lagrange multiplier of Lemma \ref{lem:mmi} the uniform bound 
\begin{equation}\label{eq:l2_lagrange_multiplier}
  h \sum_{\ell = 1}^{\lfloor \frac{T}{h} \rfloor} |\lambda^\ell|^2 \leq C
\end{equation}
is necessary. This was already established in \cite{kramer2024efficient} by bounding the first variation of the energy and distance in the Euler-Lagrange equation of the minimizing movement interpretation of Lemma~\ref{lem:mmi}.  

\begin{proposition}[Lower Semicontinuity of Distance Term]\label{prop:distance}
  Assume that \eqref{eq:assumption_weakl1} and the conclusion of Lemma \ref{lem:Compactness} and Lemma \ref{lem:conv_densities} holds. If the left-hand side of \eqref{eq:prop_distance} is finite then there exists normal velocities 
  \begin{equation} \label{eq:V_dist_vel}
    V_i \in L^2(|\nabla \chi_i|dt)
  \end{equation}
  in the sense of \eqref{def:velocity} satisfying $\partial_t \chi_i = V_i |\nabla \chi_i|$ distributionally and the relation 
  \begin{equation}\label{eq:prop_distance}
    \liminf_{h \rightarrow 0} \int_0^T \frac{1}{2h^2} d_h^2(\chi_h(t+h), \chi_h(t))\,dt \geq \frac{c_0}{2} \sum_{i = 1}^P\int_0^T \rho V_i^2 |\nabla \chi_i| \,dt.
  \end{equation}
\end{proposition}

\begin{proposition}[Lower Semicontinuity of Metric Slope Term]\label{prop:metricslope}
  Given the conclusion of Lemma~\ref{lem:conv_densities} and $u_h \rightarrow \chi$ strongly in $L^2(M \times (0,T), \R^d)$, then for every $i = 1, \dots, P$ there exists a distributional 
  mean curvature vector $H_i \in L^2(|\nabla \chi_i|dt)$ defined by \eqref{def:mean_curvature}  satisfying the relation
  \begin{align}\label{eq:prop_metric}
    \liminf_{h \downarrow 0} \frac{1}{2}\int_{0}^{T} \Big|\partial E_h^{\lceil\frac{t}{h}\rceil}\Big|^2(u_h(t)) \,dt \geq \frac{c_0}{2} \sum_{i = 1}^P \int_M \int_0^T  \left( H_i + \langle \mathbf{n}_i,  \nabla \log \rho\rangle - \Lambda_i\right)^2 \rho \,|\nabla \chi_i|\,dt. 
  \end{align}
  where $\Lambda_i$ as in \eqref{def:lagrange_multi}.
\end{proposition}
Given the results of Lemmas~\ref{lem:mmi}, \ref{lem:dedi}, \ref{lem:Compactness} and \ref{lem:conv_densities} as well as Propositions~\ref{prop:distance} and \ref{prop:metricslope} the proof of Theorem \ref{the:main} is just about piecing everything together.
\begin{proof}[Proof of Theorem \ref{the:main}]
   Lemma \ref{lem:mmi} allows applying Lemma \ref{lem:dedi} with the energy $E_h^\ell(u) = E_h(u) - \Lambda^\ell \cdot \int_M u \,d\mu$. Note that $E_h(\chi^\ell) - E_h(\chi^{\ell-1}) = E_h^\ell(\chi^\ell) - E_h^\ell(\chi^{\ell-1})$ as $\int_M \chi^\ell \,d\mu = \int_M \chi^{\ell-1} \,d\mu$ by definition of the volume constrained MBO scheme.
   Hence, the energy dissipation inequality~\eqref{eq:dedi} holds with $(E,d,\chi,u)$ replaced by $(E_h, d_h, \chi_h, u_h)$. The energy $E_h$ is consistent by the $\Gamma$-convergence result proven in \cite[Theorem 3]{laux2021large}, i.e.,
   \begin{equation}\label{eq:consistency}
    \lim_{h \downarrow 0} E_h(\chi^0) = E(\chi^0) = c_0 \sum_{i = 1}^P \rho |\nabla \chi_i^0|.
   \end{equation}
   Therefore, the assumption of Lemma \ref{lem:Compactness} is satisfied such that, up to a subsequence, $\chi_h \rightarrow \chi$ in $L^1(M \times (0,T), \R^d), \chi \in \{0,1\}^P, \sum_{i = 1}^P \chi_i = 1$ $\mu$-a.e. This implies $u_h \rightarrow \chi$ in $L^2(M \times (0,T), \R^d)$ as $\lim_{h \downarrow 0}\|u_h - \chi_h\|_{L^2(M \times (0,T))} = 0$ by \eqref{eq:bound_dist}, \eqref{eq:l2_u} and $\lim_{h \downarrow 0}\|\chi - \chi_h\|_{L^2(M \times (0,T))} \leq \lim_{h \downarrow 0}\|\chi - \chi_h\|_{L^1(M \times (0,T))} =  0$.

   Due to \eqref{ass:energy_conv} and \eqref{eq:vi_prop} the prerequisites for Lemma \ref{lem:conv_densities} are met for both $u_h$ and $\chi_h$ which allows applying Proposition \ref{prop:distance} to $\chi_h$ and Proposition \ref{prop:metricslope} to $u_h$. From here on the proof is following the same lines as in \cite{MR4385030}. The argument is sketched for the sake of completeness. 

   Define 
   \begin{align*}
       \rho(t) = \sum_{\ell = 1}^{\lfloor\frac{t}{h} \rfloor}  &\frac{1}{2}\int_{(\ell-1)h}^{nh}\left(\frac{1}{h^2}d^2(\chi_h(s), \chi_h(s-h) ) +
    |\partial E^\ell(u(s))|^2 \right) ds \\
    +  &\frac{1}{2}\int_{\lfloor\frac{t}{h} \rfloor h}^{t} \left(\frac{1}{h^2}d^2(\chi_h(s), \chi_h(s-h) ) +
    |\partial E^\ell(u(s))|^2 \right) ds
   \end{align*}
   such that $\rho(t) = \rho(\ell h) + \delta(t)$ for $t \in [\ell h, (\ell+1)h)$, with 
   \begin{equation*}
       \delta(t) :=\frac{1}{2}\int_{\ell h}^{t} \left(\frac{1}{h^2}d^2(\chi_h(s+h), \chi_h(s) ) +
    |\partial E^\ell(u(s))|^2 \right) ds.
   \end{equation*}
 Due to \eqref{eq:dedi} it holds $\int_0^T \delta(t) \, dt \leq h E_h(\chi_h^0)$. Hence, multiplying \eqref{eq:dedi} in form of $\rho(\ell h) \leq E_h(\chi_h^0)$ with $\eta(\ell h) - \eta((\ell -1)h)$ for a non-increasing $\eta \in C_0^\infty([0,T])$ yields $\int_0^\infty (- \frac{d \eta}{dt}) \rho \, dt \leq (\eta(0) + h \sup |\frac{d\eta}{dt}|) E_h(\chi^0)$. Use $\eta(t) = \max\{\min\{\frac{T-t}{\tau},1 \},0\}$ and integrate by parts once to get
    \begin{align*}
      \frac{1}{\tau} \int_{T-\tau}^T E_h(\chi_h(t))\, dt &+  \int_0^{T-\tau}
       \frac{1}{h^2}d^2(\chi_h(t), \chi_h(t-h) )\,dt \\
       &+\sum_{n = 1}^{\lfloor\frac{T- \tau}{h} \rfloor}\frac{1}{2}\int_{(n-1)h}^{nh} |\partial E^n_h(u(t))|^2  \,dt
    +  \frac{1}{2}\int_{\lfloor\frac{T-\tau}{h} \rfloor h}^{T-\tau} 
    |\partial E^n_h(u(t))|^2  \,dt\\
    \leq (1 + \frac{h}{\tau}) E_h(\chi^0).
    \end{align*}
    To pass to the limit $h \downarrow 0$ in the first left-hand side term, one takes \eqref{eq:lowersemicontinuity_energy} with $(0,T)$ replaced by $(T-\tau, T)$. The second left-hand side term is treated by \eqref{eq:prop_distance} (note that we may extend the integral down to 0 because of the second item in \eqref{eq:piecewise_interpolation}) with $(0,T)$ replaced by $(0, T - \tau)$. For the remainder of the left-hand side inequality \eqref{eq:prop_metric} is used with $(0,T)$ replaced $(T- \tau, T)$. For the right-hand side term, one applies \eqref{eq:consistency}. In conclusion one obtains with $ \Lambda_i = \frac{\int_M ( H + \langle \mathbf{n}_i, \nabla \log \rho \rangle ) \rho |\nabla \chi_i|}{\int_M \rho |\nabla \chi_i|}$ that 
    \begin{align*}
     &\frac{c_0}{\tau}\sum_i \int_{(T-\tau,T)} \rho \,|\nabla \chi_i|(t) \,dt
      \nonumber\\+& \frac{c_0}{2}\sum_i  \int_{M \times (0,T)} \rho\left(V_i^2 + \left(H - \Lambda_i + 
      \langle \mathbf{n}_i, \nabla \log \rho \rangle\right)^2 \right) \,|\nabla \chi_i| \, dt
      \leq \frac{c_0}{2}\sum_i \int_M \rho \, |\nabla \chi_i|.
   \end{align*}
   Dividing by $c_0$ and taking the limit $\tau \downarrow 0$ implies \eqref{eq:deGiorgi}.

  \end{proof}

\section{Notation and Prerequisites}\label{sec:notation}
The heat kernel on weighted Riemannian manifolds is first introduced. A brought introduction to the most important properties can be found in \cite{MR2218016}.
Let $(M, \mu, g)$ be a closed, smooth, $d$-dimensional, weighted Riemannian manifold. That is $M$ is compact and without boundary, the metric $g:TM \times TM \rightarrow \R$ is smooth, and $\mu = \rho \Vol$ is a measure on $M$ with smooth, positive density $\rho > 0$. The tangent bundle will be denoted by $TM$ which roughly is the union of all tangent spaces $T_x M$ for $x\in M$. For tangent vectors $v,w \in T_x M$ it will be used $\langle v,w \rangle_x := g_x(v,w)$ and $|v|_x^2 = g_x(v,v)$. The class of smooth functions $f:M \rightarrow \R$ is denoted by $f \in C^\infty(M)$ while the class of smooth vector fields $\xi: TM \rightarrow \R$ is denoted by $\xi \in \Gamma(TM)$. For $f \in C^\infty (M)$ the gradient $\nabla f(x) \in T_x M$ is defined by satisfying 
\begin{equation*}
  \langle \nabla f(x), v \rangle_x = D_x f(v) \quad \text{ for all } v \in T_x M
\end{equation*}
where $D_x f$ is the differential of $f$ at $x \in M$. The weighted divergence $\divrho: \Gamma(TM) \rightarrow C^\infty(M)$ is defined by the duality 
\begin{equation*}
  \int_M \divrho(\xi) f \, d\mu = - \int \langle \xi, \nabla f \rangle \, d\mu \quad \text{ for all } f \in C^\infty(M).
\end{equation*}
One can then check that $\divrho(\xi) = \frac{1}{\rho} \div(\rho \xi)$ where $\div$ denotes the divergence defined analogous to $\divrho$ but with respect to the volume measure $\Vol$. Consequently, the weighted Laplace--Beltrami operator is defined by $\Delta_\mu := - \divrho \nabla$ (the sign-convention is used to ensure the eigenvalues are positive). The solution $e^{-h \Delta_\mu} \chi$ of the heat equation 
\begin{equation*}
  \begin{cases}
    \frac{d}{dh} u = - \Delta_M u & \text{on } M \times (0, T]\\
    u(x,0) = \chi(x) &\text{in } M
  \end{cases}
\end{equation*}
can be expressed by convoluting with the heat kernel $p$, see \cite[Theorem 3.3]{MR2218016} and references therein. This means
\begin{equation}\label{def:heat_kernel}
  e^{-h \Delta_\mu} \chi = p(h) * \chi := \int_M p(h,x,y) \chi(y) \dmy
\end{equation}
and furthermore it holds the symmetry $p(h,x,y) = p(h,y,x)$, normalization $\int_M p(h,x,y) \dmy = 1$ and the semigroup property 
\begin{equation}\label{eq:semigroup_prop}
  p(t+s,x,y) = \int_M p(t,x,q) p(s,q,y) \, d\mu(q) \quad \text{ for any } t,s \in \R_+.
\end{equation}
Throughout the paper approximations for the heat kernel will be used that are close to the Euclidean heat kernel. Therefore, the following abbreviations are introduced
\begin{align*}
  G_h(x,y) &:= \frac{1}{(4\pi h)^{d/2}} e^{- \frac{\dist^2(x,y)}{4h}},\\
  G_h(z) &:= \frac{1}{(4\pi h)^{d/2}} e^{- \frac{|z|^2}{4h}}.
\end{align*}
In the following, for $x,y \in M, z \in T_x M$, it will be used
  $x_z := \exp_x(z)$ as shorthand notation for the exponential map $\exp_x:T_x M \rightarrow M$. The injective radius $\inj$, describing the radius in which the exponential map is diffeomorphic, is positive as $M$ is compact and smooth. Hence, the inverse $\exp_x^{-1}(y)$ is well-defined if $\dist(x,y) < \inj$. Furthermore, denote $\gamma_{x,y}$ for the unit speed geodesic connecting $x$ with $y$ (if it exists) and 
  $\gamma_{x,z}$ for the geodesic starting at $x$ with initial velocity $z$. Parallel transport from $x$ to $y$ along the geodesic $\gamma_{x,y}$ is denoted by $PT_{x \rightarrow y}$ where parallel transport will only be applied if $\dist(x,y) \leq \inj$ such that there exists a unique geodesic. The parallel transport is always with respect to the Levi-Civita connection which is in this paper depicted by $\nabla$ (it can be distinguished from the gradient as it is applied to vector fields and not to functions).

  A function $\chi:M \rightarrow \{0,1\}$ is of bounded variation, or short $\chi \in  BV(M, \{0,1\})$, if $\chi \in L^1(M)$ and
  \begin{equation*}
    \sup \left\{ \int_M \div(\xi) \chi \, d\Vol : \xi \in \Gamma(TM), \sup_{M} |\xi| \leq 1 \right\} < \infty .
  \end{equation*}
  As usual, the Riesz-representation theorem yields a Gauss--Green measure $\nabla \chi$ such that
  \begin{equation*}
    \int_M \div(\xi) \chi \, d\Vol = -  \int_M \langle \xi, \nabla \chi \rangle \quad \text{for all } \xi \in \Gamma(TM). 
  \end{equation*}
  Denoting by $\mathbf{n}$ the measure theoretic normal satisfying $\nabla \chi = \mathbf{n} |\nabla \chi|$, where $|\cdot|$ stands for the total variation, then the previous identity turns into
  \begin{equation}
        \int_M \div(\xi) \chi \, d\Vol = -  \int_M \langle \xi, \mathbf{n}\rangle  |\nabla \chi | \quad \text{for all } \xi \in \Gamma(TM). 
  \end{equation} 
  In the following the normal $\mathbf{n}_i$ will always be the one associated to the function $\chi_i$.
\section{Estimates Involving Heat Kernels}\label{cha:estimates_kernel}
We start with classic Gaussian bounds and the asymptotic expansion of the heat kernel:
\begin{rem}\label{rem:tools}
\begin{enumerate}[i)]
  \item\textbf{Scaling of volume and area.} Denote by $B_r(x)$ the ball around $x$ with radius $r$ with respect to the Riemannian distance. There exists $C > 0$ such that for any $x \in M$ and any $r > 0$ it holds 
  \begin{align}\label{eq:doubling_prop}
   \frac{1}{C} r^{d} \leq \mu(B_r(x)) \leq C r^d.
  \end{align}
  \item \label{item:ass_expansion}\textbf{Asymptotic expansion of the heat kernel.} There are smooth coefficient functions $v_i \in C^\infty(M \times M)$ such that for any $j,k \in \N$ and $K > j + k+ \frac{d}{2}$ there exists a constant $C_K$ such that for $x,y \in M$ with $\dist_M(x,y) \leq \frac{\inj}{2}$ it holds 
  \begin{align}\label{eq:asymptotic_exp}
  \left|\nabla^j_x \nabla^k_y \left(p(h,x,y) - \frac{e^{\frac{-\dist_M^2(x,y)}{4h}}}{(4\pi h)^{d/2}} \sum_{i = 0}^K v_i(x,y) h^i \right) \right| \leq C_K h^{K + 1 - \frac{d}{2}-j -k}.
  \end{align}
  Furthermore, $v_0(x,x) = \frac{1}{\rho(x)}$ and $v_0(x,y) = v_0(y,x)$ for all $x,y \in M$.
  \item \textbf{Gaussian bounds for heat kernel.} There are constants $C_1,C_2,C_3,C_4 > 0$ such that for every $h> 0$ and $x,y \in M$ it holds
  \begin{align}\label{eq:gaussian_one}
  \frac{C_{1}}{\mu(B_{\sqrt{h}}(x))}e^{-\frac{\dist_M^2(x,y)}{C_{2} h}} \leq p(h,x,y) \leq \frac{C_{3}}{\mu(B_{\sqrt{h}}(x))}e^{-\frac{\dist_M^2(x,y)}{C_{4} h}}.
  \end{align}
  \item \textbf{Upper Gaussian bound for gradient of heat kernel.} There are constants $\tilde{C}_1, \tilde{C}_2 > 0$ such that for every $h> 0$ and $x,y \in M$ it holds
  \begin{align}\label{eq:gaussian_two}
  |\nabla_x p(h,x,y)| \leq \frac{\tilde{C}_1}{\sqrt{h} \mu(B_{\sqrt{h}}(x))}e^{-\frac{\dist_M^2(x,y)}{\tilde{C}_2 h}}.
  \end{align}
  \end{enumerate}
  The scaling of volume is induced by the Bishop--Gromorov inequality \cite{MR169148} (see also \cite[Lemma 36]{MR2243772}) and the explicit representation of the volume element on the hyperbolic space. The necessary bound on the Ricci curvature follows from our regularity assumption and the compactness of the manifold.
  The asymptotic expansion can be proven by the parametrix method as done in \cite{MR1462892}.
  The Gaussian bounds \eqref{eq:gaussian_one} and \eqref{eq:gaussian_two} are a consequence of the Li--Yau inequality \cite{MR1186481}.
  \end{rem}

  To the knowledge of the author there is no literature regarding Gaussian bounds for mixed derivatives, but only for example for higher derivatives in a single variable \cite{cao2021hessian}. This is probably due the fact that the bounds on mixed derivatives can be easily deduced from the semigroup property \eqref{eq:semigroup_prop} together with the Gaussian bound for the single variable derivatives, as the following result shows.
\begin{corollary}\label{cor:mixed_gaussian_bound}
  There are constants $\hat{C_1}, \hat{C_2}> 0$ such that for every $h> 0$ and $x,y \in M$ is holds
  \begin{equation}\label{eq:gaussian_three}
    |\nabla_x \nabla_y p(h,x,y)| \leq \frac{\hat{C_1}}{h\mu(B_{\sqrt{h}}(x))}e^{-\frac{\dist^2(x,y)}{\hat{C_2}h}}.
  \end{equation}
\end{corollary}
\begin{proof}
  Spelling out the semigroup property \eqref{eq:semigroup_prop} gives $p(h,x,y) = \int_M p(h,x,q) p(h,q,y) \, d\mu(q)$. Hence, the mixed derivatives of the heat kernel can be estimated with the Gaussian bound \eqref{eq:gaussian_two} to obtain
  \begin{align*}
    |\nabla_x \nabla_y p(h,x,y)| &\leq  \int_M \Big|\nabla_x p\Big(\frac{h}{2}, x,q\Big)\Big| \Big|\nabla_y p\Big(\frac{h}{2}, q,y\Big)\Big| \, d\mu(q) \\
    &\lesssim \frac{1}{h} \int_M G_{\frac{\tilde{C_2}  h}{2}}(x, q)  G_{\frac{\tilde{C_2}  h}{2}} (q, y) \, d\mu(q).
  \end{align*}
  The claim follows now by using the Gaussian bound \eqref{eq:gaussian_one} and the semigroup property \eqref{eq:semigroup_prop} once more
  \begin{align*}
    \int_M G_{\frac{\tilde{C}_2 h}{2}}( x, q)  G_{\frac{\tilde{C}_2 h}{2}} ( q,y) \, d\mu(q) &\leq \int_M p\left(\frac{\tilde{C_2} h}{2 C_2}, x, q \right)  p\left(\frac{\tilde{C_2} h}{2 C_2}, q, y \right)\, d\mu(q)\\
    &= p\left(\frac{\tilde{C_2} h}{ C_2}, x, y \right)\\
    &\lesssim G_{\frac{\tilde{C_2} C_4 h}{ C_2}} (x, y). \qedhere
  \end{align*}
\end{proof}
The following ``Cut and Replace''-lemma will be used throughout the paper to reduce the domain of integration and replace the heat kernel by the zeroth-order approximation $G_h v_0$ in limits of terms similar to the energy \eqref{def:energy_h}. The cutting of the integration domain is due to the Gaussian bounds while the replacing of the heat kernel is a consequence of the asymptotic expansion.
\begin{lemma}\label{lem:cut_and_replace}
Let $f \in L^\infty(M \times M)$ and $\mathcal{I}_h$ be either the 
constant interpolation $\chi_h$ of \eqref{eq:piecewise_interpolation} or the variational interpolation \eqref{def:variation_interpolation} of Lemma \ref{lem:dedi}.
\begin{enumerate}[a)]
  \item \label{part:nonscaled_vanish} It holds for any $n \in \N$ and $C > 0$ that 
  \begin{equation}
    \lim_{h \downarrow 0} \int_M \int_M \left(\frac{\dist(x,y)}{\sqrt{h}}\right)^n G_{Ch}(x,y)  \mathcal{I}_h(x) (1-\mathcal{I}_h)(y) f(x,y) \dmx \dmy = 0.
  \end{equation}
  \item \label{part:cut_and_replace}  Let $k,\ell \in {0,1}$, $\xi_1 \in \Gamma(TM^\ell), \xi_2 \in \Gamma(TM^k)$ be smooth sections.
  Then it holds with $v_0 \in C^\infty(M \times M)$  defined by Remark \ref{rem:tools}.\ref{item:ass_expansion}) that one can \emph{cut} the domain of integration
  \begin{align}
    \lim_{h \downarrow 0} \frac{h^{(k + \ell)/2}}{h}\bigg|&\int_M \int_M \big\langle\nabla^k_y  \langle \nabla^\ell_x p(h,x,y), \xi_1(x)
    \rangle_x, \xi_2(y) \big\rangle_y \mathcal{I}_h(x)(1-\mathcal{I}_h)(y) f(x,y)\dmx \dmy \nonumber\\
    -& \int_M \int_{B_{\inj(x)}} \!\!\big\langle\nabla^k_y  \langle \nabla^\ell_x p(h,x,y), \xi_1(x)
    \rangle_x, \xi_2(y) \big\rangle_y \mathcal{I}_h(x)(1-\mathcal{I}_h)(y) f(x,y) \dmx \dmy \bigg| \nonumber\\
    = 0\label{eq:farfield}
  \end{align}
  and \emph{replace} the heat kernel by the zeroth-order term
  \begin{align}
     \lim_{h \downarrow 0} \frac{h^{(k + \ell)/2}}{h}\bigg| & \int_M \int_{B_{\inj(x)}} \!\!\big\langle\nabla^k_y  \langle \nabla^\ell_x p(h,x,y), \xi_1(x)
    \rangle_x, \xi_2(y) \big\rangle_y \mathcal{I}_h(x)(1-\mathcal{I}_h)(y) f(x,y) \dmx \dmy \nonumber \\
    - &\int_M \int_{B_{\inj(x)}} \!\!\big\langle\nabla^k_y  \langle 
      \nabla^\ell_x \big(G_h(x,y) v_0(x,y)\big), \xi_1(x)
    \rangle_x, \xi_2(y) \big\rangle_y \mathcal{I}_h(x)(1-\mathcal{I}_h)(y) f(x,y)\nonumber\\
    & \hspace{315pt} \times\dmx \dmy \bigg| \nonumber\\
    = 0.   \label{eq:reduced_kernel}
  \end{align}
\end{enumerate}

\end{lemma}
\begin{proof}
  a) Note that the function $ t \mapsto t^n e^{-t}$ is bounded over $\R$ such that 
  \begin{equation*}
  \sup_{h,x,y} \left(\frac{\dist(x,y)}{\sqrt{h}}\right)^n G_{Ch}(x,y) \leq C_n\;  G_{2Ch}(x,y).
  \end{equation*}
  The Gaussian bound \eqref{eq:gaussian_one} plus the scaling of volumes \eqref{eq:doubling_prop} implies $G_{Ch}(x,y) \lesssim p( C_2C h,x,y)$. This, combined with the monotonicity 
  \eqref{eq:mon_const} and  the interpolation properties \eqref{eq:piecewise_interpolation}, \eqref{eq:vi_prop} yields part~\ref{part:nonscaled_vanish}) 
  \begin{align*}
    &\lim_{h \downarrow 0} \int_M \int_M \left(\frac{\dist(x,y)}{\sqrt{h}}\right)^n G_{Ch}(x,y)   \mathcal{I}_h(x)(1-\mathcal{I}_h)(y) \dmx \dmy\\
    &\lesssim  \lim_{h \downarrow 0} \int_M \int_M p(2 C_2 Ch,x,y) \mathcal{I}_h(x)(1-\mathcal{I}_h)(y) f(x,y)\dmx \dmy \\
    &= \lim_{h \downarrow 0}  \sqrt{2C_2  h} E_{2C_2 C h}(\mathcal{I}_h)\\  
    &\lesssim \lim_{h \downarrow 0} \sqrt{h} E_h(\mathcal{I}_h)\\
    &\leq \lim_{h \downarrow 0} \sqrt{h} E_h(\chi^0)\\
    &=0.
  \end{align*}

  The first step for \ref{part:cut_and_replace}) is to show that the far field contribution vanishes which yields \eqref{eq:farfield}.
  Therefore, the Gaussian bounds~\eqref{eq:gaussian_one}, \eqref{eq:gaussian_two}, \eqref{eq:gaussian_three}
  and the scaling of balls~\eqref{eq:doubling_prop} is used to estimate
  \begin{align*}
    &\lim_{h \downarrow 0}\left| \frac{h^{(k + \ell)/2}}{h}\int_M \int_{M \setminus B_{\inj(x)}} \langle\nabla^k_y  \langle \nabla^\ell_x p(h,x,y), \xi_1(x)
    \rangle_x, \xi_2(y) \rangle_y  \mathcal{I}_h(x)(1-\mathcal{I}_h)(y) f(x,y)\, d\mu(x,y)\right| \\
    & \leq \lim_{h \downarrow 0} \frac{h^{(k + \ell)/2}}{h}\int_M \int_{M \setminus B_{\inj(x)}} |\nabla^k_y \nabla^\ell_x p(h,x,y)| | \xi_1(x)|
    |\xi_2(y)| \mathcal{I}_h(x)(1-\mathcal{I}_h)(y)|\dmx \dmy\\
    &\lesssim  \lim_{h \downarrow 0} \frac{\|\xi\|_\infty}{h^2} \int_M \int_{M \setminus B_{\inj(x)}} 
     G_{Ch}(x,y) \dmx \dmy  \\
    &\lesssim  \lim_{h \downarrow 0}\frac{\|\xi\|_\infty}{h^2} \frac{G_{Ch}(\inj)}{h} \\
    &= 0.
    \end{align*}

  Turning towards \eqref{eq:reduced_kernel}, the asymptotic expansion \eqref{eq:asymptotic_exp}
  yields a $K \in \N$ such that
  \begin{align*}
    \lim_{h \downarrow 0} \frac{h^{(k + \ell)/2}}{h}\int_M \int_{B_{\inj(x)}} \langle\nabla^k_y  \langle \nabla^\ell_x \left(p(h,x,y) - G_h(x,y)\sum_{i = 0}^K v_i(x,y) h^i \right) , \xi_1(x)
    \rangle_x, \xi_2(y) \rangle_y\hspace{42pt} \\  \cdot \mathcal{I}_h(x)(1-\mathcal{I}_h)(y)\dmx \dmy  = 0. 
  \end{align*}
  Hence, the limit is the same when replacing the heat kernel with the approximated heat kernel. It remains to prove 
  that the summands with $i > 0$ vanish in the limit. The derivatives of the Euclidean heat kernel are given by 
  \begin{align*}
    \nabla_x G_h(x,y) &= \frac{\exp^{-1}_x(y)}{h} G_h(x,y)\\
    \nabla_y \nabla_x G_h(x,y) &= \left(\nabla^k_y \frac{\exp^{-1}_x(y)}{h}\right) G_h(x,y) +  \frac{\exp^{-1}_y(x)}{h} \otimes \frac{\exp^{-1}_x(y)}{h} G_h(x,y).
  \end{align*}
  The functions $\nabla \exp^{-1}_x(y), v_0(x,y), \nabla v_0(x,y)$ and $\nabla_x\nabla_y v_0(x,y)$ are bounded uniformly over
  $M \times M$  due to compactness of $M$ and smoothness of the functions in their domain of definition. This yields together with 
  $|\exp^{-1}_x(y)| = \dist(x,y)$ the estimate
  \begin{align*}
    &\left|\langle\nabla^k_y  \langle \nabla^\ell_x \left(G_h(x,y) v_i(x,y) h^i \right) , \xi_1(x) \rangle_x, \xi_2(y) \rangle \right| \\
    &\leq \|\xi_1\|_\infty \|\xi_2\|_\infty (|\nabla^k_y \nabla^\ell_x G_h(x,y)| + |\nabla^k_y  G_h(x,y)||\nabla^\ell_x v_i(x,y)|\\
    &\hspace{152.2pt}+ |\nabla^\ell_x  G_h(x,y)||\nabla^k_y v_i(x,y)|+ |\nabla^k_y \nabla^\ell_x v_i(x,y)|)\\
    &\leq \|\xi_1\|_\infty \|\xi_2\|_\infty \left(\frac{C}{h^{(k+\ell)/2}} +  
    \frac{\dist^2(x,y)}{h^{k + \ell}}+\frac{\dist(x,y)}{h^k} + \frac{\dist(x,y)}{h^\ell} +1 \right)G_h(x,y).
  \end{align*}
  Now with part \ref{part:nonscaled_vanish}) one concludes for $i > 0$
  \begin{align*}
    &\lim_{h \downarrow 0} \left|\frac{h^{(k + \ell)/2}}{h}\int_M \int_{B_{\inj(x)}} \langle\nabla^k_y  \langle \nabla^\ell_x \left(G_h(x,y) v_i(x,y) h^i \right) , \xi_1(x)
    \rangle_x, \xi_2(y) \rangle  \mathcal{I}_h(x)(1-\mathcal{I}_h)(y)\, d\mu(x,y) \right|\\
    &\lesssim \lim_{h \downarrow 0} \int_M \int_M h^{(k + \ell)/2}\left(\frac{C}{h^{(k+\ell)/2}} +  
    \frac{\dist^2(x,y)}{h^{k + \ell}}+\frac{\dist(x,y)}{h^k} + \frac{\dist(x,y)}{h^\ell} +1 \right) \\
    &\hspace{100pt}\cdot G_h(x,y)   \mathcal{I}_h(x)(1-\mathcal{I}_h)(y)\dmx \dmy\\
     &\lesssim \lim_{h \downarrow 0} \int_M \int_M \left(1 +  
    \frac{\dist^2(x,y)}{h}+\frac{\dist(x,y)}{\sqrt{h}}  \right) G_h(x,y)   \mathcal{I}_h(x)(1-\mathcal{I}_h)(y)\dmx \dmy\\
    &=0. \qedhere
  \end{align*}
\end{proof}

The next Lemma will be crucial in the proof of Proposition \ref{prop:metricslope}. It will be necessary to estimate the Hessian of the squared distance function which appear in second derivatives of the heat kernel. In the Euclidean case $M = \R^d$ it holds $\nabla^2_x \frac{|x-y|^2}{2} =  Id_d$. On a manifold this holds only up to second order errors due to the spreading of geodesics by curvature. The first derivative of the distance on the manifold is given by $\nabla_x \dist^2(x,y) = 2 \exp^{-1}_x(y)$ such that after a suitable identification of tangent spaces the second derivative is given by $D_x \exp^{-1}_x(y)$. To see that this quantity is close to the identity one shows that its differential of the inverse $(D \exp_x)_{\exp^{-1}_x(y)}$ is almost the identity. Now, one can already see that this is related to the Jacobi fields $J(t) = (D \exp_x)_{tv} (tw)$ satisfying the Jacobi equation $\ddot{J}(t) + R(\dot{J}(t), \dot{\gamma}_{x,y}(t))\dot{\gamma}_{x,y}(t) = 0$ which will be the main tool for the proof.

    \begin{lemma}\label{lem:hessian_dist_equal_identity}
    For any $x,y\in M$ with $\dist(x,y)\leq \inj \wedge  \sqrt[4]{\frac{1}{2 d^3 C_R^2 }e^{-1 - d C_R \diam(M)^2}}$ and $v,w \in T_y M$  it holds the following error estimate of the Hessian of the (half) squared distance function 
    \begin{equation}\label{eq:taylor_log}
      \left|\Big\langle \nabla_y^2 \frac{\dist^2(x,y)}{2}(w), v \Big\rangle_y  - \langle w,v \rangle_y \right| \lesssim \dist^2(x,y)|v||w|.
    \end{equation}
    \end{lemma}
    \begin{proof}
    We follow the lines of Villani \cite[Chapter 14]{old_new}, \cite[Equation (4.2)]{distsquared}.
    Denote by $\gamma(t) := \exp_y(t \exp_y^{-1}(x))$ the geodesic from $y$ to $x$ with velocity $|\dot{\gamma}| = \dist(x,y)$. Further introduce the moving orthonormal frame $(e_1(t), \dots, e_d(t)), e_1(0) = \frac{\exp_y^{-1}(x)}{|\exp_y^{-1}(x)|}$,
    with $e_j(t)$ evolving with parallel transport. With the Riemannian curvature tensor $R$ (which is bounded as $M$ is compact and smooth) denote
    \begin{equation*}
    R_{ij}(t) = \langle R(\dot{\gamma}(t), e_i(t)) \dot{\gamma}(t), e_j(t) \rangle_{\gamma(t)}.
    \end{equation*}  
    One calls a matrix $J$ a Jacobi matrix if it satisfies the Jacobi matrix equation
    \begin{equation}
      \ddot{J}(t) + R(t) J(t) = 0.
    \end{equation}
    The dots represent the second covariant derivative. The Jacobi matrices are the fields describing how geodesics spread apart over time. They are closely related to Jacobi fields which describe the behavior of the differential of the exponential function.
    Furthermore, introduce the two Jacobi matrices $J_0$ and $J_1$ that solve the Jacobi matrix equation with initial conditions given by 
    \begin{equation}
      \begin{array}{cc}
        J_0(0) = 0_d & \dot{J}_0(0) = I_d\\
        J_1(0) = I_d &\dot{J}_1(0) = 0_d.
      \end{array}
    \end{equation} 
    Then by  \cite[Chapter 14]{old_new}, \cite[Equation (4.2)]{distsquared} it holds 
    \begin{equation}\label{eq:hess_dist_equal_jacobi}
      \langle \nabla_y^2 \frac{\dist^2(x,y)}{2}(w), v \rangle_y = \langle J_0(1)^{-1} J_1(1) \cdot w, v \rangle_y 
    \end{equation} 
    where $J_0(1)^{-1} J_1(1) \cdot w = (J_0(1)^{-1} J_1(1))_{ij} w^j e_i(1)$.
    The next step is to prove that $J_k(1)$ is close to the identity for $k = 0,1$. To that aim define for any Jacobi matrix $J$ the auxiliary function $\phi(t) := |J(t)|_F^2 + |\dot{J}(t)|_F^2$ with $|\cdot|_F$ denoting the Frobenius norm.
    Using $|R_{ij}| \leq C_R |\dot{\gamma}|^2|e_i||e_j| = C_R \diam(M)^2$, $\frac{d}{dt}|J_{ij}|^2 = \frac{d}{dt} |\langle J_i, e_j \rangle|^2 = 2 J_{ij} \dot{J}_{ij}$ and $\frac{d}{dt}|\dot J_{ij}|^2 =2 \dot J_{ij} \ddot{J}_{ij}$ one computes
    \begin{align*}
      \dot{g} &= 2\sum_{i,j=1}^d J_{ij} \dot{J}_{ij} + J_{ij} \ddot{J}_{ij}\\
      &= 2\sum_{i,j=1}^d J_{ij} \dot{J}_{ij} + J_{ij} \sum_{k = 1}^d R_{ik} J_{kj}\\
      &\leq \sum_{i,j=1}^d |J_{ij}|^2 +| \dot{J}_{ij}|^2 + \sum_{k = 1}^d C_R \diam(M)^2 (|J_{ij}|^2 +  |J_{kj}|^2)\\
      &= (1+ d C_R \diam(M)^2) \phi.
    \end{align*}
    Hence, the Gr{\"o}nwall inequality yields $|J(t)|_F^2 \leq e^{(1+ d C_R \diam(M)^2)t} (|J(0)|_F^2 + |\dot{J}(0)|^2)$ which implies for $J = J_k$, $k= 0,1$, $0 \leq t \leq 1$ that 
    \begin{equation}
      |J_k(t)|^2 \leq d  e^{1+ d C_R \diam(M)^2}.
    \end{equation}
    Furthermore, due to $J_k(1) - I_d = \int_0^1 \int_0^t \ddot {J_k}(s) \,ds \,dt$ and $|R|_F^2 \leq d^2 C_R^2 \dist^4(x,y)$ it holds
    \begin{align*}
      |J_k(1) - I_d|_F^2 \leq \int_0^1 \int_0^t |R(s)|_F^2 |{J_k}(s) |_F^2\,ds \,dt 
      \leq d^3 C_R^2 e^{1 + d C_R \diam(M)^2} \dist^4(x,y).
    \end{align*}
    Due to the Neumann series it holds for $\dist(x,y) \leq \sqrt[4]{\frac{1}{2 d^3 C_R^2 }e^{-1 - d C_R \diam(M)^2}}$ that $|J_0(1)^{-1}|_F \leq \frac{1}{1 - |I_d -J_0(1)|_F } \leq 2$ which lets us conclude
    \begin{align*}
      |J_0(1)^{-1}J_1(1) - I_d|_F &\leq |J_0(1)^{-1}|_F (|J_0(1) - I_d |_F + |J_1(1) - I_d |_F )\\
      & \leq 4 \sqrt{d^3 C_R^2 e^{1 + d C_R \diam(M)^2}}\dist^2(x,y).
    \end{align*} 
    The step is finished by putting the identity \eqref{eq:hess_dist_equal_jacobi} into \eqref{eq:taylor_log} and using the just established bound
    \begin{align*}
      \left|\langle \nabla_y^2 \frac{\dist^2(x,y)}{2}(w), v \rangle_y  - \langle v,w \rangle_y \right| = \left|\langle \left(J_0(1)^{-1} J_1(1) - I_d\right) \cdot w, v \rangle_y \right| \lesssim  \dist^2(x,y) |v||w|.
    \end{align*}
  \end{proof}
We finish this chapter with some useful estimates that are well known in the literature

\begin{lemma}[Useful estimates]
  \begin{enumerate}
    \item For any function $u:M \rightarrow [0,1]^P$ with $\sum_{i=1}^P u_i = 1$ $\mu$-a.e.\ it holds
    \begin{equation}\label{eq:kernel_diff}
      \int_M |u - p_h * u| \, d\mu \leq 2 \sqrt{h} E_h(u).
    \end{equation} 
    \item For any $\chi: M \rightarrow \{0, 1\}^P$ with $\sum_i^P \chi = 1$ $\mu$-a.e.\ and any 
    $0 < h \leq h_0$ it holds
    \begin{align}
      E_{Ch}(\chi) &\leq \max\{\frac{1}{\sqrt{C}}, \sqrt{C}\} E_h(\chi), \label{eq:mon_const}\\
      E_{h_0}(\chi) &\lesssim (E_{h}(\chi) + 1) \label{eq:approx_mon}.
    \end{align}
    \item For any $u, \tilde{u}: M \rightarrow [0, 1]^P$ with $\sum_i^P u_i = 
    \sum_i^P \tilde u_i= 1$ a.e. it holds
    \begin{equation}\label{eq:l2_u}
      \sum_{i = 1}^P \int_M |u_i - \tilde{u}_i|^2 \, d\mu \leq \frac{1}{2\sqrt{h}}d_h^2 (u, \tilde{u})
      + 2\sqrt{h} (E_h(u) + E_h (\tilde{u})).
    \end{equation}
    In particular for any $\chi, \tilde{\chi}: M \rightarrow \{0, 1\}^P$ with $\sum_{i = 1}^P \chi_i = 
    \sum_{i = 1}^P \chi_i= 1$ a.e. it holds
    \begin{equation}\label{eq:l1_estimate}
      \sum_{i = 1}^P \int_M |\chi_i - \tilde{\chi}_i| \, d\mu \leq \frac{1}{2\sqrt{h}}d_h^2 (\chi, \tilde{\chi})
      + 2\sqrt{h} (E_h(\chi) + E_h (\tilde{\chi})).
    \end{equation}
  \end{enumerate}
\end{lemma}
\begin{proof}
  
  Estimate \eqref{eq:kernel_diff} follows immediately from the inequality 
  $|u - u'| \leq u(1-u') + u'(1-u)$ which holds true for any $[0,1]$-valued function. Contrary, 
  the inequality \eqref{eq:approx_mon} is non-trivial and was proven in \cite[Lemma 13]{kramer2024efficient}
  with the help of the segment inequality of Cheeger and Colding~\cite{MR1320384}. Estimate
  \eqref{eq:mon_const} follows by monotonicity of $\sqrt{h}E_h(\chi)$ and $\frac{1}{\sqrt{h}}E_h(\chi)$, 
  see \cite[Lemma 11]{kramer2024efficient}. The last result
  \eqref{eq:l1_estimate} stems from the definition of the energy and distance term together with 
  the fact that for any functions $u, \tilde{u} \in [0,1]$ it holds
  \begin{equation*}
    |u - \tilde{u}|^2 \leq (u - \tilde{u})p(h)*(u - \tilde{u}) + |u - p(h)*u| +
     |\tilde{u} - p(h)*\tilde{u}|.
  \end{equation*}
  For more information see for example \cite[Lemma 3]{MR4385030} or 
  \cite[Lemma 15]{kramer2024efficient}.
\end{proof}
\section{Proofs}\label{cha:proofs}
\begin{proof}[Proof of Lemma \ref{lem:mmi}]
  The relation \eqref{def:mms} is proven in \cite[Lemma 8]{kramer2024efficient}. The definition \eqref{def:dist_h} and the fact that the heat kernel is positive definite (which itself is a consequence of the semigroup property \eqref{eq:semigroup_prop}) imply that $d_h$ is a distance on $\mathcal{M}$. The compactness of $(\mathcal{M}, d_h)$ and the continuity of $E_h^\ell$ follow by the fact that $d_h$ metrizes weak convergence. To proof this lets first assume that $\chi_n \xrightharpoonup{n \rightarrow \infty} \chi$: Then it holds that $p*(\chi_n - \chi) \rightarrow 0$ strongly in $L^2(M, \mu)$. To see this, note that $\|p(h)*\xi\|_{L^2} + \|\nabla p(h)*\xi\|_{L^2} \leq \|\xi\|_{L^2} + \|\nabla \xi\|_{L^2}$ which implies strong $L^2$ compactness for any uniformly bounded family of function in $H^1(M, \mu)$. The claim is thus valid for the regular function family, after identifying the limit by and integration by parts. By a separability argument, and $\sup_n\|\chi_n - \chi\| < \infty$, the claim is also true for $\chi_n - \chi$. The strong convergence of $p(h)*(\chi_n - \chi)$ coupled with the weak convergence of $\chi_n - \chi$ turns to $d_h(\chi_n, \chi) = \int_M (\chi_n - \chi) \cdot p(h)*(\chi_n - \chi) \, d\mu \rightarrow 0$ for $n \rightarrow \infty$. 

  For the other direction assume that $d_h(\chi_n, \chi) \rightarrow 0$ for $n \rightarrow 0$: As $\sup_n \|\chi_n \|_{L^2} < \infty$ there exists a subsequence (again denoted by $n$) that converges to some limit $\tilde{\chi}$. But from the previous step this implies $\lim_{n \rightarrow \infty} d_h(\chi_n, \tilde{\chi}) = 0$. Together with the assumption $d_h(\chi_n, \chi) \rightarrow 0$ and the properties of $d_h$ being a distance, if follows first $d_h(\chi, \tilde{\chi}) = 0$ and then $\chi = \tilde{\chi}$. Now the limit is fixed, so the whole sequence has to converge weakly to $\chi$.
  
\end{proof}
\begin{proof}[Proof of Lemma \ref{lem:dedi}]
  The statement is a classical result of gradient flow theory, see for example \cite[Theorem 3.1.4, Lemma 3.1.3]{greenbook}.
  The additional inequality \eqref{eq:dedi} follows immediately by the telescoping property of \eqref{eq:fdedi}. Furthermore,
  estimate \eqref{eq:bound_dist} is implied by \eqref{eq:dedi} and \eqref{eq:compare_dist}.
\end{proof}
\begin{proof}[Proof of Lemma \ref{lem:Compactness}]
  \emph{Step 1} Modulus of continuity: For any in time piecewise constant $\chi \in \{0,1\}^P$ with $\sum_{i = 1}^P \chi_i = 1$
  define 
  \begin{align*}
    \mathcal{I}(s) &:= \sum_{i = 0}^P \int_{(s,T) \times M} |\chi_i(t) - \chi_i(t-s)| \, d\mu \, dt\\
    C_0 &:= \int_{(h,T)} \frac{1}{2h^2} d_h^2(\chi(t), \chi(t-h)) \,dt + 4 \int_0^T E_h(\chi(t))\,dt.
  \end{align*}
  Then it holds 
  \begin{align}
    \mathcal{I}(s) \leq C_0 \left\{\begin{array}{ll}
      \frac{s}{\sqrt{h}} & \text{for } s \leq h\\
      2 \sqrt{h} & \text{for } h\leq s \leq \sqrt{h}\\
      4s & \text{for } \sqrt{h} \leq s
    \end{array} \right\} \leq 4 C_0 \sqrt{s}.
  \end{align}
  The proof of this \emph{Step 1} is exactly as in \cite{MR4385030} - it is restated for the 
  convenience of the reader.

  Beginning with the first case $s \leq h$ we notice that by the piecewise interpolation it 
  holds $\mathcal{I}(s) \leq \frac{s}{h} \mathcal{I}(h)$. Also by the $\L^1$- estimate 
  \eqref{eq:l1_estimate} it holds 
  \begin{equation*}
    \mathcal{I}(s) \leq \int_s^T \frac{1}{2\sqrt{h}} d_h^2(\chi, \chi(\cdot - s)) \,dt +
    4 \sqrt{h} \int_0^T E_h(\chi) \, dt
  \end{equation*}
  and hence $\mathcal{I}(h) \leq C_0 \sqrt{h}$.

  For the case $h \leq s \leq \sqrt{h}$ one first assumes $s = N h$ with $N \in \N$. An easy 
  computation shows that the triangle inequality 
  \begin{equation*}
    d_h^2{\chi, \chi(\cdot - s)} \leq N \sum_{n = 1}^N d_h^2(\chi( \cdot - (n-1)h) ,\chi(\cdot
    -nh)) \, dt
  \end{equation*}
  holds which immediately implies by the first case the inequality 
  \begin{equation*}
    \mathcal{I}(s) \leq \left( \frac{s}{h} \right)^2 \int_h^T \frac{1}{2\sqrt{h}}
    d_h^2(\chi, \chi(\cdot - h)) \, dt + 4 \sqrt{h} \int_0^T E_h(\chi) \,dt \leq C_0 \max \left\{
      \frac{s^2}{\sqrt{h}}, \sqrt{h} \right\} = C_0 \sqrt{h}.
  \end{equation*}
  For an arbitrary $h \leq s \leq \sqrt{h}$ one uses the decomposition $s = Nh + s'$ with $s'
  \in (0,h)$ and uses $\mathcal{I}(s) \leq \mathcal{I}(Nh) + \mathcal{I}(s')$.

  Similarly, for the last case $\sqrt{h} \leq s$ one uses $s = N \sqrt{h} + s'$ with $N \in \N, 
  s' \leq \sqrt{h}$ together with $\mathcal{I}(s) \leq N \mathcal{I}(\sqrt{h}) + \mathcal{I}(s')$.

  \emph{Step 2:} Pre-compactness of $\{\chi_h\}_{h \downarrow 0}$.

  For every $h_0 \geq h > 0$ the estimates \eqref{eq:kernel_diff} and \eqref{eq:approx_mon} 
  tell us that $\chi_h$ is uniformly close to $p(h_0) * \chi_h$ in $L^{\infty}((0,T), L^2(M,\mu))$
  (and thus in $L^1((0,T)\times M, dt\times \mu))$ as
  \begin{align*}
    \int_M |\chi_h - p(h_0)*\chi|^2 \, d\mu \leq \int_M |\chi_h - p(h_0)*\chi| \, d\mu \leq 
    2\sqrt{h_0} E_{h_0}(\chi_h) \lesssim \sqrt{h_0}(E_h(\chi_h) + 1).
  \end{align*}
  Hence, it is enough to prove for fixed $h_0 >0$ that $\{p(h_0)*\chi_h\}_{h\downarrow 0}$ is 
  compact in $L^1((0,T)\times M, dt \times \mu)$.

  To that aim, one utilizes \cite[Theorem 4.7.28]{MR2267655} which characterizes $L^1$-compactness
  on arbitrary measure spaces. Namely, for our finite measure $dt \llcorner_{(0,T)} \times \mu$ 
  and bounded sequence $\{p(h_0)*\chi_h\}_{h\downarrow 0}$ it is enough to prove that for any 
  $\eps > 0$ there exists a finite partition $\pi$ of $(0,T) \times M$ such that for every $h > 0$
  it holds 
  \begin{equation}
    \| p(h_0)*\chi_h - \mathbb{E}^\pi \left(p(h_0)*\chi_h\right)\|_{L^1((0,T)\times M)} \leq \eps
  \end{equation}
  with 
  \begin{equation*}
    \mathbb{E}^\pi f(t,x) := \frac{1}{(dt \times \mu) (E)}\int_E f \, dt \, d\mu \quad \text{if } (t,x) \in E \in \pi.
  \end{equation*}
  To that aim, let $\eps > 0$ and partition $(0,T)\times M$ in products of the form $E_{\ell,k} = [(\ell-1)s, \ell s) \times U_r(x_k)$ with small diameters, i.e.,
  \begin{align*}
    \dot{\cup}_{k = 1}^K U_r(x_k) &= M,\\
    \dot{\cup}_{\ell = 1}^L [(\ell-1)s, \ell s) &= [0,T),\\
    \diam \left(U_r(x_k)\right) &\leq r \quad \text{for all } k = 1,\dots, K.
  \end{align*}
  Then, simply plugging in the definitions and using the triangle inequality once yields
  \begin{align*}
    &\| f- \mathbb{E}^\pi \left(f\right)\|_{L^1((0,T)\times M)} \\
    &= \int_{(0,T)\times M } \left|f(x,t) - \sum_{\ell, k} \frac{1}{\mu(U_r(x_k)) s}
    \int_{E_{\ell,k}} f(y,\tau) \dmy \,d\tau \mathbbold{1}_{E_{\ell,k}}(t,x) \right|\\
    &= \sum_{\ell, k}  \frac{1}{\mu(U_r(x_k)) s}\int_{E_{\ell,k}} \left|\int_{E_{\ell,k}}f(x,t) - f(y,\tau) \dmy \, d\tau \right| 
    \dmx\,dt\\ 
    &\leq \sum_{\ell, k}  \frac{1}{\mu(U_r(x_k)) s}\int_{E_{\ell,k}}\int_{E_{\ell,k}}\left|f(x,t) - f(y,t) \right| + \left|f(y, \tau) - f(y,t) \right| \dmy \, d\tau  \dmx \, dt.
  \end{align*}
  The idea for $f = p(h_0) * \chi_h$ is to use the smoothness for the space difference while using \emph{Step 2} for the time difference.

  By the compactness of $(0,T) \times M$, the differentiability of $p(h_0)$ and the assumption on 
  $U_r(x_k)$ one gets for $x,y \in U_r(x_k)$ that
  \begin{equation*}
    \left|p(h_0) * \chi_h(x,t) - p(h_0) * \chi_h(y,t) \right| \leq \|\nabla p(h_0)\|_{C^0((0,T)\times M)} \dist(x,y)
    \leq C(h_0) r.
  \end{equation*}
  Using $r = \frac{\eps}{2C(h_0)T}$ and the fact that $E_{\ell,k}$ is a partition of $(0,T) \times M$ one therefore infers  
  \begin{align*}
    \sum_{\ell, k}  \frac{1}{\mu(U_r(x_k)) s}\int_{E_{\ell,k}}\int_{E_{\ell,k}}\left|f(x,t) - f(y,t) \right| \dmy \, d\tau  \dmx \, dt\\
    \leq  C(h_0) r \sum_{\ell,k} (dt \times \mu )(E_{\ell,k}) = C(h_0) r T = \frac{\eps}{2}.
  \end{align*}

  For the second summand, it is enough to prove that for $s = \left(\frac{\eps}{16C_0}\right)^2$ and $f = p(h_0) * \chi_h$ 
  it holds 
  \begin{equation*}
    \sum_{\ell = 1}^{L} \frac{1}{s} \int_{(\ell - 1)s}^{\ell s} \int_{(\ell - 1)s}^{\ell s} 
    \int_M |f(t) - f(\tau)|\, d\mu\, dt \, d\tau \leq \frac{\eps}{2}.
  \end{equation*}
  In order to apply the modulus of continuity of \emph{Step 1}, the variable $s' = t- \tau$ is 
  introduced, i.e.
  \begin{equation*}
    \sum_{\ell = 1}^{L} \frac{1}{s} \int_{(\ell - 1)s}^{\ell s} \int_{(\ell - 1)s}^{\ell s} 
    |f(t) - f(\tau)| \, dt \, d\tau = \sum_{\ell = 1}^{L} \frac{1}{s} \int_{(\ell - 1)s}^{\ell s} \int_{(\ell - 1)s-\tau}^{\ell s - \tau} 
    |f(s' + \tau) - f(\tau)| \, ds' \, d\tau.
  \end{equation*}
  Now, with the help of Fubini's theorem and the previous transformation one deduces
  \begin{align*}
    \sum_{\ell = 1}^{L} &\frac{1}{s} \int_{(\ell - 1)s}^{\ell s} \int_{(\ell - 1)s}^{\ell s} 
    |f(t) - f(\tau)| \, dt \, d\tau \\
    &=  \sum_{\ell = 1}^{L} \frac{1}{s} \int_{-s}^{s}
     \int_{(\ell - 1)s + (-s' \vee 0)}^{\ell s + (-s' \wedge 0)} |f(s' + \tau) - f(\tau)|\, d\tau \, ds'\\
     &=  \sum_{\ell = 1}^{L} \frac{1}{s} \int_{0}^{s}
    \int_{(\ell - 1)s}^{\ell s -s'} |f(s' + \tau) - f(\tau)|\, d\tau \, ds'
    + \sum_{\ell = 1}^{L} \frac{1}{s} \int_{0}^{s}
    \int_{(\ell - 1)s+ s'}^{\ell s } |f(\tau) - f(\tau - s')|\, d\tau \, ds'\\
    &\leq 2 \frac{1}{s} \int_{0}^{s}\int_{s'}^{T} |f(\tau) - f(\tau -s')|\, d\tau \, ds'.
  \end{align*}
  By \emph{Step 2} together with the Jensen inequality it holds 
  $\int_{s'}^{T} \int_M |f(\tau) - f(\tau -s')|\, d\mu\, d\tau \leq 4 C_0 \sqrt{s'}$ and hence one concludes with $s' \leq s$
  \begin{equation*}
    \sum_{\ell = 1}^{L} \frac{1}{s} \int_{(\ell - 1)s}^{\ell s} \int_{(\ell - 1)s}^{\ell s} \int_M
    |f(t) - f(\tau)|\,d\mu \, dt \, d\tau \leq 8 C_0 \sqrt{s} = \frac{\eps}{2}.
  \end{equation*}

  The lower bound \eqref{eq:lowersemicontinuity_energy} was proven in the $\Gamma$-convergence of the thresholding energy by Laux and Lelmi in \cite{laux2021large}. Later in the proof of Lemma \ref{lem:conv_densities}, we give a simpler, localized proof for the bound \eqref{eq:lowersemicontinuity_energy}.
\end{proof}
\begin{proof}[Proof of Lemma \ref{lem:conv_densities}]
  Spelling out the pushforward measures it boils down to prove for any function $f\in C^1({TM^2}\times \R), (x,y, z_x, z_y,t)\in M \times M \times T_x M \times T_y M \times \R
  \mapsto f(x,y,z_x, z_y,t)$
  with polynomial growth in $z_x$ and $z_y$ the limits
  \begin{align}
    &\lim_{h \downarrow 0} \int_0^T \frac{1}{\sqrt{h}}\int_M \int_{B_{\inj}(x)} f \Big(x,y, \frac{\exp^{-1}_x(y)}{\sqrt{h}},  \frac{\exp^{-1}_y(x)}{\sqrt{h}},t\Big) p(h,x,y) u_{h,i}(x,t) 
    (1-u_{h,i})(y,t) \nonumber\\
    &\hspace{370pt} \times \dmy \dmx \,dt \nonumber\\
    &= \int_0^T \int_M \int_{T_x M}f(x,x,z,-z,t)  G_1(z) \left( \langle \mathbf{n_i}(x), z \rangle_x \right)_+
    \,dz\, \rho(x) |\nabla \chi_i|(x,t) \,dt \label{eq:energy_conv_1}
  \end{align}
  for any test function in $C^1(TM^2 \times (0,T))$.

  \emph{Step 1:} By the ``Cut and Replace''-Lemma \ref{lem:cut_and_replace}\ref{part:cut_and_replace}) one can replace $p(h,x,y)$ by $G_h(x,y)v_0(x,y)$ such that the limit on the left-hand side of 
  equation \eqref{eq:energy_conv_1} is equal to
  \begin{align}
    &\lim_{h \downarrow 0} \int_0^T \frac{1}{\sqrt{h}}\int_M \int_{B_{\inj}(x)} f \Big(x,y, \frac{\exp^{-1}_x(y)}{\sqrt{h}},  \frac{\exp^{-1}_y(x)}{\sqrt{h}},t\Big)
    G_h(x,y) v_0(x,y) u_h(x,t) (1-u_h)(y,t) \label{eq:energy_conv_gauss} \\
    & \hspace{370pt} \times\dmy \dmx \,dt. \nonumber
  \end{align}
  
  \emph{Step 2}: For showing \eqref{eq:energy_conv_gauss} it will be convenient to show 
  the following equi-integrability
  \begin{align*}
    \int_M \int_{B_{\frac{\inj}{\sqrt{h}}}(0)} e^{\frac{3|z|^2}{8}}e_h(z,t,x) \,dz \dmx \leq C E_h(u_h(t)) < \infty \quad \text{for all } t \in [0, \infty) 
  \end{align*}
  with a constant $C > 0$ and using the notation introduced in Chapter \ref{sec:notation} for
  \begin{align*}
    e_h(z,t,x) &:= G_h(z) \frac{1}{\sqrt{h}} (1-u_h)(x,t)u_h(x_{\sqrt{h}z},t) \eta(x_{\sqrt{h}z})\\
    \eta(y) &:= \sqrt{\det (g)(y)}\rho(y)v_0(x,y).
  \end{align*}
  Indeed, we first observe that $G_4(z) = \frac{1}{2^d}e^{3|z|^2/8}G_1(z)$ which 
  yields together with the Gaussian bound \eqref{eq:gaussian_one}, in form of 
  $G_{4h}(z) \leq \frac{1}{C_1} p(C_2 4h, x,x_z)$, that
  \begin{align*}
    \int_M &\int_{B_{\frac{\inj}{\sqrt{h}}}(0)} e^{\frac{3|z|^2}{8}}e_h(z,t,x) \,dz \dmx\\
   &\leq \frac{2^{d+1}}{C_1}\int_M \int_{B_{\inj}(0)}  p(C_2 4h, x,
    x_z) \frac{1}{\sqrt{h}} (1-u_h)(x,t)u_h(x_{z},t) \eta(x_{z}) \,dz \dmx.
  \end{align*}
By changing the polar coordinates around $x$ back to coordinates on $M$, the 
right-hand side is bounded by $ \frac{2^{d+1}}{C_1} E_{ C_2 4h}(u_h(t))$. The claim follows
now with the monotonicity \eqref{eq:mon_const} and $C = \frac{2^{d+2}\sqrt{C_2}}{C_1}$.

\emph{Step 3}: It is sufficient for the limit \eqref{eq:energy_conv_1} to prove the local lower bound
\begin{align}\label{eq:local_bound}
  &\liminf_{h \downarrow 0} \int_0^T \int_M \int_{B_{\inj}(x)} f \Big(x,y, \frac{\exp^{-1}_x(y)}{\sqrt{h}},  \frac{\exp^{-1}_y(x)}{\sqrt{h}},t\Big)
  \frac{G_h(x,y) v_0(x,y) u_{h,i}(x,t) (1-u_{h,i})(y,t)}{\sqrt{h}} \\
  &\hspace{370pt}\times \dmy \dmx \,dt \nonumber \\
 &\geq \int_0^T \int_M \int_{T_x M}f(x,x,z,-z,t)\rho(x) G_1(z)(\langle \mathbf{n}_i(x), z\rangle_x)_+ \,dz \,|\nabla \chi_i|(x,t)\, dt. \nonumber
\end{align}

Indeed, the energy convergence assumption~\eqref{ass:energy_conv}, 
with $\chi_h$ replaced by $u_h$, together with applying Lemma \ref{lem:cut_and_replace} to the far field term and using (see \eqref{eq:derivation_c0} for more details)
\begin{equation*}
\int_{T_x M}G_1(z) (\langle \mathbf{n}_i(x), z\rangle_x)_+ \,dz = c_0
\end{equation*}
 for all $x \in M, i = 1, \dots, P$, yields the global upper bound
\begin{align}\label{eq:global_bound}
  \limsup_{h \downarrow 0} \frac{1}{\sqrt{h}} \sum_{i = 1}^P \int_0^T \int_M \int_{B_{\inj}(x)} G_h(x,y)v_0(x,y) u_{h,i}(x,t) (1-u_{h,i})(y,t) \dmy \dmx \,dt \nonumber\\
  \leq \int_0^T \int_{T_x M} \rho(x) G_1(z)(\langle \mathbf{n}_i(x), z\rangle_x)_+ \,|\nabla \chi_i|(x,t) \,dz \, dt.
\end{align}
Combining the global upper bound \eqref{eq:global_bound} together with the local lower bound \eqref{eq:local_bound}
(for every $i = 1, \dots, P$) one concludes that for every index $i$, it has to hold equality globally
\begin{align}\label{eq:global_bound_i}
  \lim_{h \downarrow 0} \frac{1}{\sqrt{h}} \int_0^T \int_M \int_{B_{\inj}(x)} G_h(x,y)v_0(x,y) u_{h,i}(x,t) (1-u_{h,i})(y,t) \dmy \dmx \,dt \nonumber\\
  = \int_0^T \int_{T_x M} \rho(x) G_1(z)(\langle \mathbf{n}_i(x), z\rangle_x)_+ \,|\nabla \chi_i|(x,t) \,dz \, dt.
\end{align}
Now, by splitting the test function $f = 1 - (1 - f)$ and applying \eqref{eq:global_bound_i} to 
$1$ and \eqref{eq:local_bound} to $(1-f)$ one infers that in \eqref{eq:local_bound} also equality has to hold.

After changing to normal coordinates on the left-hand side of \eqref{eq:energy_conv_gauss}, it 
follows by \emph{Step 2} that it is enough to assume $f$ is bounded. Furthermore, 
by rescaling and splitting into positive and negative part one assumes $f$ is $[0,1]$-valued.
For the proof of the local lower bound the argument will be provided pointwise in time. Hence, fix a $t \in [0,T]$ and drop the $t$-variable in the notation for better readability. 

\emph{Step 4:} The left-hand side of the local lower bound \eqref{eq:local_bound} can be bounded from below for any $\xi \leq f$ and $\Theta(a,b) := \xi(a,b, PT_{x \rightarrow a}z, -PT_{x \rightarrow b}z)$ by 
\begin{align}
  \liminf_{h \downarrow 0}  \int_M \int_{B_{\inj}(x)} f(x,y, \exp^{-1}_x(y), \exp^{-1}_y(x)) 
  \frac{G_h(x,y) u_{h,i}(x) (1-u_{h,i})(y)}{\sqrt{h}} \dmy \dmx& \nonumber\\
  \geq - \int_M  \int_{T_x M} D(\Theta v_0 \rho^2)(x)[z] G_1(z)\,dz \,\chi_i(x)\,d\Vol(x) &. \label{eq:lower_bound_differential}
\end{align}
 
The integral on the left-hand side of \eqref{eq:lower_bound_differential} is bounded from below for any $\xi \leq f$ by 
\begin{equation*}
\frac{1}{\sqrt{h}}\int_M \int_{B_{\inj}(x)} \xi\Big(x,y, \frac{\exp^{-1}_x(y)}{\sqrt{h}},  \frac{\exp^{-1}_y(x)}{\sqrt{h}}\Big)
    G_h(x,y) v_0(x,y) u_h(x) (1-u_h)(y) \dmy \dmx
  \end{equation*}
which is, using $u(1-u') \geq u - u'$ for $u,u' \in [0,1]$, bounded from below by 
\begin{align}
  &\int_M \int_{B_{\inj}(x)}  \xi\Big(x,y, \frac{\exp^{-1}_x(y)}{\sqrt{h}},  \frac{\exp^{-1}_y(x)}{\sqrt{h}}\Big)
  G_h(x,y)v_0(x,y)\frac{(u_{h,i}(x) - u_{h,i}(y))}{\sqrt{h}} \dmy \dmx \nonumber\\
  &=\int_M \int_{B_{\inj}(x)}  u_{h,i}(x) 
  G_h(x,y)v_0(x,y)\frac{\xi \big(x,y, \frac{\exp^{-1}_x(y)}{\sqrt{h}},  \frac{\exp^{-1}_y(x)}{\sqrt{h}}\big) - \xi\big(y,x, \frac{\exp^{-1}_y(x)}{\sqrt{h}},  \frac{\exp^{-1}_x(y)}{\sqrt{h}}\big)}{\sqrt{h}}\nonumber\\
  &\hspace{380pt} \times \dmy \dmx \nonumber\\
    &=\int_M \int_{B_{\frac{\inj}{\sqrt{h}}}(0) \subset T_x M} D_h'\xi(x,z) 
  G_1(z)v_0(x,x_{\sqrt{h}z})\eta(x_{\sqrt{h}z})  \,dz\, u_{h,i}(x) \dmx    \label{eq:wrong_difference}
\end{align}
Here one first swapped the names of the $x,y$  variable (using the symmetry of $G_h$ and $v_0$ in $(x,y)$) to put the difference quotient onto the test function 
$\xi$. Second, one applied normal coordinates around $x$ and introduced
  \begin{align*}
  D_h'\xi(x,z) &:= \frac{\xi(x,x_{\sqrt{h}z},z,-PT_{x\rightarrow x_{\sqrt{h}z}}z) 
  - \xi(x_{\sqrt{h}z},x ,-PT_{x\rightarrow x_{\sqrt{h}z}}z, z)}{\sqrt{h}},\\
  \eta(y) &:= \sqrt{\det (g)(y)}\rho(y)v_0(x,y).
\end{align*}
Notice that if the sign in the $z$-variables of $f$ agreed, this would be a difference quotient  in the $x$ variables. 
Hence, use the symmetry of $G_1$ and $B_{\frac{\inj}{\sqrt{h}}}(0)$ to change the coordinates from $z$ to $-z$
in the subtrahend. By that the lower bound \eqref{eq:wrong_difference} is equal to
\begin{align*}
    \int_M \int_{B_{\frac{\inj}{\sqrt{h}}}(0) \subset T_x M} D_h(\xi \eta)(x,z) 
    G_1(z) \,dz\, u_h(x) \dmx    
\end{align*}
with 
\begin{equation*}
  D_h(\xi \eta)(x,z) := \frac{\xi(x,x_{\sqrt{h}z},z,-PT_{x\rightarrow x_{\sqrt{h}z}}z) \eta(x_{\sqrt{h}z})
  - \xi(x_{-\sqrt{h}z},x ,PT_{x\rightarrow x_{-\sqrt{h}z}}z, -z) \eta(x_{-\sqrt{h}z})}{\sqrt{h}}.
\end{equation*}
We will show now that 
\begin{equation}\label{eq:conv_difference_quotient}
  |D_h(\xi \eta)(x,z) + D (\Theta v_0 p^2)(x)[z]| \lesssim \sqrt{h} |z|
\end{equation}
such that $D_h(\xi \eta)(x,z) G_1(z)$ converges uniformly to $-D (\Theta v_0 p^2)(x)[z] G_1(z)$ and $u_h \mathbbold{1}_{B_{\frac{\inj}{\sqrt{h}}}(0)}$ to $\chi$ in $L^1(M)$ which yields the claimed \eqref{eq:lower_bound_differential}.
To pass to the limit one splits similarly to the product rule
\begin{align*}
  D_h(\xi \eta)(x,z) = D_h^a (\xi)(x,z) \eta(x_{\sqrt{h}z}) &+ D_h^b (\xi)(x,z) \eta(x_{\sqrt{h}z}) \\
  &+ D_h (\eta)(x,z) \xi(x_{-\sqrt{h}z},x ,PT_{x\rightarrow x_{-\sqrt{h}z}}z, -z)
\end{align*}
 with the simpler quotients
\begin{align*} 
  D_h^a\xi(x,z) &:= \frac{\xi(x,x,z,-z) 
  - \xi(x_{-\sqrt{h}z},x ,PT_{x\rightarrow x_{-\sqrt{h}z}}z, -z)}{\sqrt{h}},\\
  D_h^b \xi(x,z) &:= \frac{\xi(x,x_{\sqrt{h}z},z,-PT_{x\rightarrow x_{\sqrt{h}z}}z) 
  - \xi(x,x,z,-z)}{\sqrt{h}},\\
  D_h \eta(x,z) &:= 
  \frac{\eta(x_{\sqrt{h}z}) - \eta(x_{-\sqrt{h}z})}{\sqrt{h}}.
\end{align*}
By smoothness of the involved terms the uniform limits
\begin{align*}
\xi(x_{-\sqrt{h}z},x ,PT_{x\rightarrow x_{-\sqrt{h}z}}z, -z)
&\xrightarrow{h \rightarrow 0} \xi(x,x,z,-z),\\
\eta(x_{\sqrt{h}z}) &\xrightarrow{h \rightarrow 0} \eta(x) = 1
\end{align*}
hold. Here, the definition of $\eta$, together with $v_0(x,x) = \frac{1}{\rho(x)}$ and $\det(g)(x) = \det(Id) = 1$ was used.
 The identity $\det(g)(x) = \det(Id)$ is valid, as the previous calculation were operated in normal coordinates.
Next, for fixed $x \in M, z \in T_x M$, the auxiliary function $\Theta(a,b) := \xi(a,b, PT_{x \rightarrow a}z, -PT_{x \rightarrow b}z)$
is introduced to identify the uniform limits
\begin{align*}
  \lim_{h \rightarrow 0} D_h^a \xi(x,z) &= \frac{d}{dt}\Big|_{t = 0} \Theta(-x_{tz}, x)
  = -D_a \Theta(x,x)\left[\frac{d}{dt}\Big|_{t=0} x_{tz}\right]= -D_a \Theta(x,x)[z],\\
  \lim_{h \rightarrow 0} D_h^b \xi(x,z) &= -\frac{d}{dt}\Big|_{t = 0} \Theta(x,x_{tz})
  = -D_b \Theta(x,x)[z].
\end{align*}
Here $D_a$ and $D_b$ denote the differential with respect to the $a$ and $b$ variable, respectively, that is $D_a \Theta (x,x) = D \Theta(\cdot, x) (x) $.
Similarly, the limit of the difference on $\eta$ is given by
\begin{align}
  \lim_{h \rightarrow 0} D_h \eta(x,z) &= -2\frac{d}{dt}\Big|_{t = 0} \eta(x_{tz})
  =-2D \eta(x)[z] = -2 \langle \nabla \eta (x), z \rangle_x.
\end{align}
By the symmetry $v_0(x,y)=v_0(y,x)$ it holds $\nabla v_0(\cdot,x) (x) = \nabla v_0(x, \cdot)(x)$
and hence $2 \nabla v_0(x, \cdot) (x) = \nabla v_0(\cdot, \cdot)(x)$. Together with the definition
of $\eta$ and $\nabla \sqrt{\det(g)}(x) = 0$ this leads to
\begin{equation*}
  2\nabla \eta(x) = 2v_0(x,x) \nabla \rho(x)  + \rho(x) \nabla v_0(\cdot,\cdot)(x) = \frac{\nabla (\rho^2 v_0)(x)}{\rho(x)}.
\end{equation*} 
Putting all limits together yields exactly \eqref{eq:conv_difference_quotient} where the convergence rate is given by the rate at which a difference quotient converges to the respective derivative.

\emph{Step 5:} Conclusion.

The proof will be finished after showing 
\begin{equation} \label{eq:identify_grad_div}
  \int_{T_x M}  D\zeta(x)[z]G_1(z)\,dz = \div \left( \int_{T_{\cdot} M} \zeta(\cdot) z G_1(z)\,dz\right)(x)
\end{equation}
for fixed $x \in M$ and $\zeta \in C^1(M)$. Indeed, using identity
\eqref{eq:identify_grad_div} for $\zeta = \Theta v_0 \rho^2$ in the lower bound \eqref{eq:lower_bound_differential} together with
$v_0(x,x) = \frac{1}{\rho(x)}$ and integration by parts one obtains
\begin{align*}
  -\int_M  &\int_{T_x M} D(\Theta v_0 \rho^2)(x)[z] G_1(z)\,dz \,\chi_i(x)\,d\Vol \\
  &=  \int_M  \int_{T_x M} (\Theta v_0 \rho^2)(x) z G_1(z)\,dz \, \cdot \nabla \chi_i(x) \\
  &= \int_M  \int_{T_x M} \xi(x,x,z,-z) \rho(x) G_1(z) \langle \mathbf{n}_i(x), z \rangle_x \,dz \, |\nabla \chi_i|(x).
\end{align*}
Taking the supremum over all $\xi \leq f$ yields \eqref{eq:local_bound} by Fatou's lemma.

So, it remains to prove \eqref{eq:identify_grad_div}. By using $\div = \tr \nabla$ and writing the Levi--Civita connection as difference quotient by
\cite[Corollary 4.35]{lee2018introduction} one obtains for any $\phi \in \Gamma(TM)$ the divergence 
in normal coordinates as
\begin{align*}
  \div (\phi)(x) &= \sum_{j = 1}^d \langle \nabla_{\partial_j} \phi (x), \partial_j \rangle_x\\
  &= \lim_{t \rightarrow 0} \sum_{j = 1}^d \frac{\langle PT_{x_{t\partial_j} \rightarrow x}\phi(x_{t\partial_j}), \partial_j \rangle_x
  - \langle \phi(x), \partial_j \rangle_x}{t}.
\end{align*}
For $\phi(x) := \int_{T_{x} M} \zeta(x) z G_1(z)\,dz$, one uses that parallel 
transport is a linear isometry and $G_1$ is invariant under linear isometries such that
\begin{align*}
  PT_{x_{t\partial_j} \rightarrow x}\phi(x_{t\partial_j}) 
  &= \int_{T_{x_{t\partial_j}} M} \zeta(x_{t\partial_j}) PT_{x_{t\partial_j} \rightarrow x}z G_1(PT_{x_{t\partial_j} \rightarrow x}z)\,dz\\
  &= \int_{T_{x} M} \zeta(x_{t\partial_j}) z G_1(z)\,dz.
\end{align*}
Combining the two previous identities one concludes
\begin{align*}
  \div \left( \int_{T_{\cdot} M} \zeta(\cdot) z G_1(z)\,dz\right)(x) 
  &= \lim_{t \rightarrow 0} \sum_{j = 1}^d \int_{T_{x} M} \frac{\zeta(x_{t\partial_j})- \zeta(x)}{t} 
  \langle z, \partial_j \rangle_x G_1(z) \, dz\\
  &=\sum_{j = 1}^d \int_{T_{x} M} D \zeta (x)[\partial_j]
  \langle z, \partial_j \rangle_x G_1(z) \, dz\\  
  &=\int_{T_{x} M} D \zeta (x)[z] G_1(z) \, dz
\end{align*}
where we used in the last equality the linearity of the differential together with $\sum_{j = 1}^d \partial_j \langle z, \partial_j \rangle = z$.
\end{proof}

\begin{proof}[Proof of Proposition \ref{prop:metricslope}]
  \emph{Step 1}: For any $\xi \in \Gamma(TM), u: M \rightarrow \R^P$ it holds 
  \begin{equation}
    \frac{1}{2}|\partial E|^2(u) \geq \delta E(u, \xi)- h \lim_{s\rightarrow 0} 
    \frac{\frac{1}{2h}d^2(u,u_s)}{s^2}
  \end{equation}
  where $u_s$ is the variation of $u$ along $\xi$ defined by 
  $\partial_s u_s + \langle \nabla u_s, \xi \rangle = 0$ and $\delta E(u,\xi) := \frac{d}{ds} \big|_{s = 0} E(u_s)$ is the first variation of the energy.

  To prove \emph{Step 1} one utilizes $\frac{1}{2} (\frac{a}{b})^2 \geq a- \frac{1}{2}b^2$ together
  with the definition of the metric slope $\partial E$
  \begin{equation*}
    \frac{1}{2}|\partial E|^2(u)\geq \lim_{s \rightarrow 0}\left(
      \frac{1}{s} (E(u) - E(u_s))_+ 
    \right) - \frac{\frac{2h}{2}\frac{1}{2h}d^2(u, u_s)}{s^2}
    \geq -\frac{d}{ds}\Big|_{s = 0} E(u_s) - \lim_{s \rightarrow 0}
    h\frac{\frac{1}{2h}d^2(u, u_s)}{s^2}.
  \end{equation*}
   In the end \emph{Step 1} will be used integrated in time for test functions $\xi \in C_0^\infty((0,T), \Gamma(TM))$. In regard of proving \eqref{eq:prop_metric} for $h$ small enough the remainder integral from $\lfloor \frac{T}{h} \rfloor h$ to $T$ will always be zero due to the compact support of $\xi$. Hence, we will drop this remainder term w.l.o.g.\ in the rest of the proof.

  \emph{Step 2}: There exists a $\Lambda \in \R^P$ such that for any $\xi \in C^\infty_0((0,T),\Gamma(TM))$ it holds
  \begin{align}
    \sum_{n = 1}^{\lfloor \frac{T}{h} \rfloor}\int_{(n-1)h}^{nh} &\delta E_h^n(u_h, \xi) \,dt  \\ \xrightarrow{h \downarrow 0} 
    \int_0^T &\delta E(\chi, \xi, \Lambda) \,dt
    := c_0 \sum_{i = 1}^P \int_0^T \int_M \div(\xi) \rho + \langle \xi, \nabla \rho \rangle-\langle \mathbf{n}_i, \nabla_{\mathbf{n}_i} \xi \rangle \rho  -  \rho \langle\xi,\mathbf{n}_i
    \rangle \Lambda_i\,  |\nabla \chi_i| \,dt .\nonumber
  \end{align}

  For simpler notation one introduces 
  \begin{align*}
    E_{h,i}(u) &:= \frac{1}{\sqrt{h}}\int_M (1- u_{i}) p(h)*u_{i} \, d\mu,\\
    E_i(\chi)     &:= c_0 \int_M \rho |\nabla \chi_i|
  \end{align*}
  such that the energies can be written as 
  \begin{align*}
    E^n_h(u) &= \sum_{i = 1}^P E_{h,i}(u) - \Lambda^n \cdot \int_M u \,d\mu,\\
    E(\chi,\Lambda) := c_0\sum_{i = 1}^P \int_M \rho |\nabla \chi_i| - c_0 \Lambda \cdot \int_M \chi \, d\mu 
    &= \sum_{i = 1}^P E_i(\chi) - c_0\Lambda \cdot \int_M \chi \, d\mu.
  \end{align*}
  Hence, it is sufficient to prove that for all $t \in (0,T)$
  \begin{equation}\label{eq:conv_first_var}
    \sum_{i = 1}^P\delta E_{h,i}(u_h, \xi) \rightarrow \sum_{i = 1}^P\delta E_i(\chi, \xi) = c_0 \sum_{i = 1}^P\int_M \div(\xi) \rho + \langle \xi, \nabla \rho \rangle-\langle \mathbf{n}_i, \nabla_{\mathbf{n}_i} \xi \rangle \rho\, |\nabla \chi_i|
  \end{equation}
  and there exists $\Lambda \in  L^2((0,T),\R^P)$ such that 
  \begin{align}\label{eq:conv_lagrange}
    \delta\left(\sum_{n = 1}^{\lfloor\frac{T}{h}\rfloor}\int_{(n-1)h}^{nh} \Lambda^n \cdot \int_M \bullet \, d\mu \, dt \right)(u_h, \xi) \xrightarrow{h  \downarrow 0} c_0 \sum_{i = 1}^P \int_0^T \Lambda \int_M \langle \xi, \mathbf{n}_i \rangle \, |\nabla \chi_i|\,dt.
  \end{align}

  \emph{Step 2.1}: Proof of convergence \eqref{eq:conv_lagrange}.

  First, introduce $\Lambda_h$ as the piecewise constant interpolation of $\Lambda^n$
   analogous to the interpolation $\chi_h$ in \eqref{eq:piecewise_interpolation}. Then the first variation
   is calculated to be
   \begin{align*}
    \delta\Bigg(\sum_{n = 1}^{\lfloor\frac{T}{h}\rfloor}\int_{(n-1)h}^{nh} \Lambda^n \cdot & \int_M \bullet \, d\mu \, dt \Bigg)(u_h, \xi)
    \\
    &= \frac{d}{ds}\Big|_{s = 0} \int_0^T \Lambda_h \cdot \int_M u_{h,s} \, d\mu \, dt= \int_0^T \Lambda_h \cdot \int_M \div(\xi) u_{h} \, d\mu \, dt.
   \end{align*}
   As already mentioned before in \eqref{eq:l2_lagrange_multiplier}, by \cite[Proposition 2]{kramer2024efficient} the time $L^2$-norm of $\Lambda_h$ is uniformly bounded.
   Hence, there exists a subsequence, again denoted by $h$, and a $\Lambda \in   L^2((0,T),\R^P)$ such 
   that 
   \begin{equation*}
    \Lambda_h \xrightharpoonup{ L^2((0,T),\R^P)} c_0 \Lambda.
   \end{equation*}
   The convergence of $\langle \Lambda_h, u_h\rangle_{ L^2((0,T),\R^P)} \rightarrow \langle c_0 \Lambda, \chi \rangle_{ L^2((0,T),\R^P)}$ for $h \downarrow 0$ follows by the strong convergence $u_h \xrightarrow{L^2(M\times (0,T), \R^d)} \chi$. In conclusion 
   the limit of the first variation is given by the definition of $\nabla \chi_i$ by
   \begin{align*}
    \delta\left(\sum_{n = 1}^{\frac{T}{h}}\int_{(n-1)h}^{nh} \Lambda^n \cdot \int_M \bullet \, d\mu \, dt\right)(u_h, \xi)
    \rightarrow &c_0\int_0^T \Lambda \cdot \int_M \div(\xi) \chi \, d\mu \, dt\\
    & = c_0 \sum_{i = 1}^P\int_0^T \Lambda_i  \int_M \langle \xi, \mathbf{n}_i \rangle  \, |\nabla \chi_i| \, dt
   \end{align*}

   \emph{Step 2.2}: Calculation and simplification of $\delta E_{h,i}(u_h, \xi)$ to
   \begin{align}
    &\lim_{h \downarrow 0} \delta E_{h,i}(u_h, \xi) = \lim_{h \downarrow 0} I + II + III, \nonumber\\
    &I := \frac{1}{\sqrt{h}} \int_M \int_{B_{\inj}(x)}  \left(\langle \xi(x) , \nabla_x G_h(x,y) \rangle_x
    + \langle \xi(y) , \nabla_y G_h(x,y)\rangle_y \right)v_0(x,y)\\
    &\hspace{120pt}\cdot  (1-u_{h,i})(x)u_{h,i}(y) \dmx\dmy \nonumber\\
    &II := \frac{1}{\sqrt{h}} \int_M \int_{B_{\inj}(x)}  \left(\langle \xi(x) , \nabla_x v_0(x,y) \rangle_x
    + \langle \xi(y) , \nabla_y v_0(x,y)\rangle_y \right)G_h(x,y)\\
    & \hspace{120pt}\cdot (1-u_{h,i})(x)u_{h,i}(y) \dmx \dmy\nonumber\\
    &III := \frac{1}{\sqrt{h}} \int_M \int_{B_{\inj}(x)} \left(\divrho(\xi)(y) + \divrho(\xi)(x)\right)G_h(x,y)v_0(x,y)\\
    &\hspace{120pt}\cdot (1-u_{h,i}(x))u_{h,i}(y)\dmx\dmy.\nonumber
   \end{align}

   By definition of the first variation together with $- \langle \xi, \nabla u_h \rangle = 
   - \divrho(u_h \xi) + \divrho(\xi) u_h$ it follows
   \begin{align*}
    &\delta E_{h,i}(u_h, \xi) \\
    =& \frac{1}{\sqrt{h}} \int_M \langle -\xi, \nabla (1- u_{h,i}) \rangle p(h)*
    u_{h,i}+ (1-u_{h,i}) p(h)*\langle -\xi, \nabla u_{h,i} \rangle\\
    =&\frac{1}{\sqrt{h}} \int_M \int_M  \Big(\langle \xi(x) , \nabla_x p(h,x,y) \rangle_x + \langle \xi(y) , \nabla_y p(h,x,y) \rangle_y 
    \\
     &\hspace{80pt}+ \divrho(\xi)(x)  p(h,x,y) + \divrho(\xi)(y)  p(h,x,y) \Big)(1-u_{h,i})(y)u_{h,i}(x)\dmx\dmy
   \end{align*}

   The next step is to use Lemma \ref{lem:cut_and_replace}\ref{part:cut_and_replace}) to cut off the far field term $\{(x,y) \in M^2: \dist(x,y) \geq \inj\}$ and 
   replace $p(h)$ by $G_h v_0$ such that
   \begin{align*}
    &\lim_{h \downarrow 0} \delta E_{h,i}(u_h, \xi) \\
    &=\lim_{h \downarrow 0}\frac{1}{\sqrt{h}} \int_M \int_{B_{\inj}(x)} \Big( \langle \xi(x) , \nabla_x \big(G_h(x,y)v_0(x,y)\big) \rangle_x
    + \langle \xi(y) , \nabla_y \big(G_h(x,y)v_0(x,y)\big)\rangle_y \\
     & \hspace{120pt}+ \left(\divrho(\xi)(y) + \divrho(\xi)(x)\right)G_h(x,y)v_0(x,y)\Big)\\
     & \hspace{120pt}\cdot(1-u_{h,i}(y))u_{h,i}(x)\dmx\dmy,
   \end{align*}
   which is the claimed identity after applying the product rule. 

   \emph{Step 2.3} Calculation of 
   \begin{equation*}
    \lim_{h \downarrow 0} III = c_0 \int_M 2(\div(\xi) \rho + \langle \xi, \nabla \rho \rangle) |\nabla \chi_i|.
   \end{equation*}

   As $\divrho (\xi) \rho = \div (\rho\xi) = \div(\xi) \rho + \langle \xi, \nabla \rho \rangle$ it is sufficient to prove
   \begin{equation*}
     \lim_{h \downarrow 0} III = c_0 \int_M 2\divrho (\xi) \rho|\nabla \chi_i|.
   \end{equation*}
   After applying Lemma \ref{lem:conv_densities} to $III$ with the test function $f(x,y,z_x, z_y) : = \divrho(\xi)(x) +\divrho(\xi)(y)$
   the limit reads
   \begin{equation*}
    \lim_{h \downarrow 0} III = \int_M 2\divrho(\xi)(x) \rho(x) 
    \int_{T_x M} G_1(z) (\langle \mathbf{n}_i(x), z \rangle_x )_+ \, dz \,|\nabla \chi_i|.
   \end{equation*}
   Now fix an orthonormal basis $\{\partial_i\}_{i = 1}^d$ of $T_x M$ and write for any tangent vector $z \in T_x M$, using Einstein summation convention, $z = z^i \partial_i$ and $\tilde{z} = \{z_i\}_{i = 1}^d \in \R^d$. Hence, $|z|_x = \langle z^i \partial_i ,z^j \partial_j \rangle_x = |\tilde z|_2$, which implies
   $G_h(z^i \partial_i)= G^{\R^d}_h(\tilde{z})$ with $G^{\R^d}_h$ denoting the Euclidean heat kernel. The claim is now finished using the rotational symmetry of $G^{\R^d}_1$ and $G_1^{\R^d}(z) = G_1^{\R^1}(z_1)G_1^{\R^{d-1}}(z')$ to compute for any $\mathbf{n} \in T_x M$ with $|\mathbf{n}|= 1$ (so in particular for $\mathbf{n} = \mathbf{n}_i(x)$) that 
   \begin{align}\label{eq:derivation_c0}
    \int_{T_x M} G_1(z) (\langle \mathbf{n}, z \rangle_x )_+ \, dz &= 
    - 2 \int_{0 < \tilde z \cdot \tilde{\mathbf{n}}} \nabla G^{\R^d}_1(\tilde{z}) \cdot  \tilde{\mathbf{n}} \,d\tilde{z}\\
    &= 2 \int_{ 0 = \tilde z \cdot \tilde{\mathbf{n}}} G^{\R^d}_1(\tilde{z}) \tilde{\mathbf{n}}\cdot  \tilde{\mathbf{n}}\,d\tilde{z}\nonumber\\
    &= 2 \int_{z_1 = 0} G^{\R^d}_1(\tilde{z})  \, d\tilde{z} \nonumber\\
    &= 2G_1^{\R^1}(0) \nonumber\\
    &= c_0. \nonumber
   \end{align}

   \emph{Step 2.4}: Calculation of 
   \begin{equation*}
    \lim_{h \downarrow 0} II = - c_0 \int_M \langle \xi, \nabla \rho \rangle |\nabla \chi_i|.
   \end{equation*}

   Again, apply Lemma \ref{lem:conv_densities}, this time with $f(x,y,z_x,z_y) : = \frac{ \langle \xi(x) , \nabla_x v_0(x,y) \rangle_x + \langle \xi(y) , \nabla_y v_0(x,y)\rangle_y}{v_0(x,y)}$, to obtain
   \begin{equation*}
    \lim_{h \downarrow 0} II = \int_M \langle \xi(x) , \nabla_x v_0(x,x) + \nabla_y v_0(x,x)  \rangle_x \frac{\rho(x)}{v_0(x,x)}
    \int_{T_x M} G_1(z) (\langle \mathbf{n}_i(x), z \rangle_x )_+ \, dz \,|\nabla \chi_i|.
   \end{equation*}
   The claim is now obtained by using $c_0 = \int_{T_x M} G_1(z) (\langle \mathbf{n}(x), z \rangle_x )_+ \, dz$ as in \emph{Step 2.3} and $v_0(x,x) = \frac{1}{\rho(x)}$ to recognize
   \begin{align*}
    (\nabla_x v_0(x,x) + \nabla_y v_0(x,x)) \frac{\rho(x)}{v_0(x,x)} = (\nabla v_0(\cdot, \cdot))(x) \rho(x)^2 = \left(\nabla \frac{1}{\rho}\right)(x) \rho(x)^2 = -\nabla \rho(x).
   \end{align*}

   \emph{Step 2.5}: Calculation of 
   \begin{equation}\label{eq:limit_I}
    \lim_{h \downarrow 0} I = - c_0 \int_M (\langle \mathbf{n}_i, \nabla_{\mathbf{n}_i} \xi \rangle + \div \xi ) \rho|\nabla \chi_i|.
   \end{equation}

   The following Taylor expansion will be necessary: For any vector field $\xi \in \Gamma(TM)$
   there exists a remainder $R \in \Gamma(TM)$ with $|R| \in O(\dist^2(x,y))$ such that
   \begin{equation}\label{eq:taylor_vector_field}
    \xi(y) = PT_{x \rightarrow y } \xi(x) + PT_{x \rightarrow y } 
    \nabla_{\exp_x^{-1}(y)} \xi(x) +R
   \end{equation}
   for all $x,y$ with $\dist(x,y) \leq \inj$. To prove this formula let $\gamma_{x,y}$ be the unit speed geodesic from $x$ to $y$, i.e.~, 
   $\gamma_{x,y}(0) = x, \gamma_{x,y}(\dist(x,y)) = y, \dot{\gamma_{x,y}} = 1$ (the dot represents the time derivative). Let $\{\partial_i\}_{i = 1}^d$ be 
   an orthonormal basis of $T_x M$. Then $E_i(t) := PT_{x \rightarrow \gamma_{x,y}(t)} \partial_i$ is an orthonormal
   basis of $T_{\gamma_{x,y}(t)} M$ for all $t \in [0, \dist(x,y)]$, i.e.~an orthonormal frame. 
   It holds (with Einstein notation)
   \begin{equation*}
    \xi(\gamma_{x,y}(t)) = \alpha_i(t) E_i(t)
   \end{equation*}
   with $\alpha_i(t)$ coefficient functions that smoothly depend on $t$. Taylor expansion of the 
   coefficient function yields 
   \begin{equation*}
    \xi(y) = \alpha_i(0)E_i(\dist(x,y)) + \alpha_i'(0)E_i(\dist(x,y)) \dist(x,y) + 
    \frac{\alpha_i''(\eta)}{2}E_i(\dist(x,y)) \dist^2(x,y)
   \end{equation*}
   with some $\eta \in [0, \dist(x,y)]$. One defines $R:=\frac{\alpha_i''(\eta)}{2}E_i(\dist(x,y)) \dist^2(x,y)$
   such that by the smoothness of $\alpha_i''$ the claimed $|R| \lesssim \dist^2(x,y)$ is verified.
   It holds by definition of parallel transport that $D_t E_i(t) = 0$ and 
   by definition of the covariant derivative along a curve $c$ that $D_t \xi(c(t)) = \nabla_{\dot{c}(t)} \xi$. So one concludes 
   with $\dot{\gamma_{x,y}}(0) = \frac{1}{\dist(x,y)}\exp_x^{-1}(y)$ and the linearity of the connection the Taylor formula by
   \begin{equation*}
    \alpha_i'(0)E_i(\dist(x,y))  = PT_{x \rightarrow y } \alpha_i'(0) E_i(0) 
    = PT_{x \rightarrow y } D_t\Big|_{t = 0} \alpha_i(t) E_i(t)
    = \frac{ PT_{x \rightarrow y } \nabla_{\exp_x^{-1}(y)} \xi(x) }{\dist(x,y)}.
   \end{equation*}
   Turning towards the limit \eqref{eq:limit_I}, one notes that by definition $\nabla_x G_h(x,y) = -\frac{\nabla_x \dist^2(x,y)}{4h} G_h(x,y)$.
   Using geodesic variations one can calculate 
   \begin{equation*}
    -\frac{1}{2}\nabla_x \dist^2(x,y) = \exp_x^{-1}(y) = PT_{y \rightarrow x}(-\exp_y^{-1}(x)) = 
    PT_{y \rightarrow x}(\frac{1}{2}\nabla_y \dist^2(x,y))
   \end{equation*}
   which implies, with $\langle PT_{x \rightarrow y} \xi(x) , \zeta(y) \rangle_y = \langle \xi(x) , PT_{y \rightarrow x} \zeta(y) \rangle_x $ for $\xi, \zeta \in \Gamma(TM)$ and the Taylor expansion \eqref{eq:taylor_vector_field}, that
   \begin{align}\label{eq:dif_heat_kernels}
    \langle \xi(x) , \nabla_x G_h(x,y) \rangle_x + \langle \xi(y) , \nabla_y G_h(x,y)\rangle_y
    &= \frac{G_h(x,y)}{2h}\langle \xi(y)- PT_{x \rightarrow y}^{\gamma_{x,y}}\xi(x) , \exp_y^{-1}(x)\rangle_y\\
    &= - \frac{G_h(x,y)}{2h}\langle
    \nabla_{\exp_x^{-1}(y)} \xi(x) , \exp_x^{-1}(y)\rangle_x\nonumber\\
    &\quad +\frac{G_h(x,y)}{2h}\langle R , \exp_y^{-1}(x)\rangle_y.\nonumber
   \end{align}
   As $\langle R, \exp_y^{-1} \rangle \in O(\dist^3(x,y))$ one obtains by Lemma \ref{lem:cut_and_replace}\ref{part:nonscaled_vanish}) that
   \begin{equation*}
    \lim_{h \downarrow 0}\frac{1}{\sqrt{h}} \int_M \int_{B_{\inj}(x)}  \frac{G_h(x,y)}{2h}\langle R , \exp_y^{-1}(x)\rangle_y v_0(x,y) (1-u_{h,i})(x)u_{h,i}(y) \dmx\dmy = 0.
   \end{equation*}
   Hence, after plugging \eqref{eq:dif_heat_kernels} into $I$ it only remains the summand incorporating the covariant derivative
   \begin{align*}
    \lim_{h \downarrow 0} I &= \lim_{h \downarrow 0} \frac{1}{\sqrt{h}} \int_M \int_{B_{\inj}(x)}  \left(\langle \xi(x) , \nabla_x G_h(x,y) \rangle_x
    + \langle \xi(y) , \nabla_y G_h(x,y)\rangle_y \right)v_0(x,y)\\
    &\hspace{120pt}\cdot  (1-u_{h,i})(x)u_{h,i}(y) \dmx\dmy \nonumber\\
    &= -\lim_{h \downarrow 0}\frac{1}{2h\sqrt{h}} \int_M \int_{B_{\inj}(x)}\langle 
    \nabla_{\exp_x^{-1}(y)} \xi(x) , \exp_x^{-1}(y)\rangle_x G_h(x,y)v_0(x,y) \\
    & \hspace{120pt}\cdot  (1-u_{h,i})(x)u_{h,i}(y) \dmx\dmy 
   \end{align*}
   By application of Lemma \ref{lem:conv_densities} to the right-hand side with $f(x,y,z_x,z_y) = \langle 
    \nabla_{z_x} \xi(x) , z_x\rangle_x \rho(x) v_0(x,y)$ one identifies the limit as
    \begin{equation*}
      -\frac{1}{2}\int_M \rho(x)\int_{T_x M}  \langle 
    \nabla_{z} \xi(x) , z\rangle_x G_1(z) (\langle \mathbf{n}_i(x), z \rangle_x )_+ \, dz \,|\nabla \chi_i|(x)
    \end{equation*}
    Now, \emph{Step 2.5} is completed after using normal coordinates as orthonormal basis $\{\partial_i\}_{i = 1}^d $ (s.t. $g_{ij} = \delta_i^j$), which implies $\langle \nabla_{\partial_i} \xi, \partial_i \rangle_x = tr(\nabla \xi) = \div \xi$, to compute for any $\mathbf{n} \in T_x M$ with $|\mathbf{n}|= 1$ (with notation as in \emph{Step 2.3}) 
    \begin{align*}
      &-\int_{T_x M}  \langle \nabla_{z} \xi(x) , z\rangle_x G_1(z) (\langle \mathbf{n}, z \rangle_x )_+ \, dz \\
      &= 2\int_{0 >\tilde{z}\cdot \tilde{\mathbf{n}}} \tilde{\mathbf{n}} \cdot \tilde{z} \nabla G_1(\tilde{z})^{\R^d} \cdot \{\langle \nabla_{\partial_i} \xi(x) , \partial_j \rangle \}_{ij} \cdot \tilde{z} \,d\tilde{z}\\
      &= - 2\int_{0 >\tilde{z}\cdot \tilde{\mathbf{n}}}\tilde{\mathbf{n}}\cdot\{\langle \nabla_{\partial_i} \xi(x) , \partial_j \rangle \}_{ij} \cdot \tilde{z} G_1^{\R^d}(\tilde{z}) 
      +  \tilde{\mathbf{n}} \cdot \tilde{z} G_1^{\R^d} \sum_{j = 1}^d \langle \nabla_{\partial_j} \xi(x), \partial_j \rangle_x \, d\tilde{z}\\
      &=- 4\int_{0 = \tilde{z}\cdot \tilde{\mathbf{n}}} \langle \mathbf{n}, \nabla_{\mathbf{n}} \xi(x) \rangle_x G_1^{\R^d}(z) + \langle \mathbf{n}, \mathbf{n}\rangle_x G_1^{\R^d}(\tilde{z}) \div \xi(x) \, d\tilde{z}\\
      &= - 2c_0 (\langle \mathbf{n}, \nabla_{\mathbf{n}} \xi(x) \rangle_x + \div \xi(x)).
    \end{align*}

    \emph{Step 2.6}: Conclusion of \emph{Step 2}, that is proof of limit \eqref{eq:conv_first_var}.
    By \emph{Step 2.3}, \emph{Step 2.4} and \emph{Step 2.5} it holds exactly the claimed quantity
    \begin{align*}
      \lim_{h \downarrow 0} \delta E_{h,i}(u_h, \xi) &=  c_0 \int_M 2(\div( \xi )\rho + \langle \xi, \nabla \rho \rangle)  -\langle \xi, \nabla \rho\rangle   -(\langle \mathbf{n}_i, \nabla_{\mathbf{n}_i} \xi \rangle + \div \xi ) \rho |\nabla \chi_i|\\
      &= c_0 \int_M \div(\xi) \rho + \langle \xi, \nabla \rho \rangle-\langle \mathbf{n}_i, \nabla_{\mathbf{n}_i} \xi \rangle \rho |\nabla \chi_i|.
    \end{align*}

    \emph{Step 3}: For any $\xi \in \Gamma(TM)$, $u_h$ the variational interpolation \eqref{eq:vi_prop} and $\partial_s u_{h,s} + \langle \nabla u_{h,s}, \xi \rangle = 0$ (distributionally) it holds
    \begin{equation}\label{eq:second_variation_distance}
       \lim_{h \downarrow 0} h \lim_{s \rightarrow 0} \frac{\frac{1}{2h}d_h^2(u_h, u_{h,s})} {s^2} = \frac{c_0}{2} \sum_{i = 1}^P \int_M \left( \langle   \xi ,\mathbf{n}_i\rangle \right)^2 \rho \, |\nabla \chi_i|.
    \end{equation}

    \emph{Step 3.1}: Simplification of the limit to 
    \begin{align}
      &\lim_{h \downarrow 0} h \lim_{s \rightarrow 0} \frac{\frac{1}{2h}d_h^2(u_h, u_{h,s})} {s^2}\label{eq:simplify_to_Hess}\\
       = &\lim_{h \downarrow 0} \sqrt{h} \int_M \int_{B_{\inj}(x)} \langle \nabla_y \langle \nabla_x G_h(x,y), \xi(x)  \rangle_x, \xi(y) \rangle_y v_0(x,y) u_{h,i}(x) (1-u_{h,i})(y) \dmx \dmy.\nonumber
    \end{align}

    By the definition \eqref{def:dist_h} of the metric $d_h$ and the definition of $u_{h,s}$ one concludes
    \begin{equation}\label{eq:second_variatio_distance}
      \lim_{h \downarrow 0} h \lim_{s \rightarrow 0} \frac{\frac{1}{2h}d_h^2(u_h, u_{h,s})} {s^2} = \lim_{h \downarrow 0} \sqrt{h} \int_M \int_M \langle \nabla u_{h,i}(x), \xi(x) \rangle p(h,x,y)\langle \nabla u_{h,i}(y), \xi(y) \rangle \dmx \dmy .
    \end{equation}
    Twice integrating by part the left-hand side of \eqref{eq:second_variatio_distance} together with $- \langle \xi, \nabla u \rangle = - \divrho(u \xi) + \divrho(\xi) u$ yields
    \begin{align*}
      \sqrt{h} \int_M \int_M \langle \nabla& u_{h,i}(x), \xi(x) \rangle p(h,x,y)\langle \nabla u_{h,i}(y), \xi(y) \rangle \dmx \dmy\\
      =\sqrt{h} \int_M \int_M &\Big(\langle \nabla_y \langle \nabla_x p(h,x,y), \xi(x)  \rangle_x, \xi(y) \rangle_y   \\
       &+\langle \nabla_x p(h,x,y), \xi(x)  \rangle_x \divrho(\xi)(y) \\
       &+\langle \nabla_y p(h,x,y), \xi(y)  \rangle_y \divrho(\xi)(x)\\
       &+ p(h,x,y) \divrho(\xi)(y) \divrho(\xi)(x) \Big)u_{h,i}(x) (1-u_{h,i})(y)\dmx \dmy
    \end{align*}
    The leading order term is the one containing $\nabla_y \nabla_x p$. To see that all the other terms vanish use, as before, Lemma \ref{lem:cut_and_replace} to cut off the far field term and reduce the kernel to $G_h v_0$ such that the limit is equal to the limit of
    \begin{align*}
      \sqrt{h} \int_M \int_{B_{\inj}(x)} &\Big(\langle \nabla_y \langle \nabla_x G_h(x,y), \xi(x)  \rangle_x, \xi(y) \rangle_y v_0(x,y)  \\
      &+ \langle \nabla_y \langle \nabla_x v_0(x,y), \xi(x)  \rangle_x, \xi(y) \rangle_y G_h(x,y) \\
      &+\langle \nabla_y v_0(x,y), \xi(y) \rangle_y \langle \nabla_x G_h(x,y), \xi(x)  \rangle_x  \\
      &+\langle \nabla_y G_h(x,y), \xi(y) \rangle_y \langle \nabla_x v_0(x,y), \xi(x)  \rangle_x \\
       &+\langle \nabla_x \big(G_h(x,y) v_0(x,y)\big), \xi(x)  \rangle_x  \divrho(\xi)(y)\\
       &+\langle \nabla_y \big(G_h(x,y) v_0(x,y)\big), \xi(y)  \rangle_y \divrho(\xi)(x) \\
       &+ G_h(x,y)v_0(x,y) \divrho(\xi)(y) \divrho(\xi)(x) \Big) u_{h,i}(x) (1-u_{h,i})(y)\dmx \dmy.
    \end{align*}
    Every summand, except the one containing $\nabla_x \nabla_y G_h$, is bounded for a given constant $C > 0$ by $C \left(1 + \frac{\dist(x,y)}{h}\right)G_h(x,y) $. Hence, by Lemma \ref{lem:cut_and_replace}\ref{part:nonscaled_vanish}) only the term 
    \begin{equation}
      \sqrt{h} \int_M \int_{B_{\inj}(x)} \langle \nabla_y \langle \nabla_x G_h(x,y), \xi(x)  \rangle_x, \xi(y) \rangle_y v_0(x,y) u_{h,i}(x) (1-u_{h,i})(y)\dmx \dmy
    \end{equation} 
    is non-vanishing in the limit which finishes \emph{Step 3.1}.

    \emph{Step 3.2}: Computation of the limit \eqref{eq:second_variation_distance}. It holds for the $u_{h,s}$, the variation along $\xi$, that
    \begin{equation*}
             \lim_{h \downarrow 0} h \lim_{s \rightarrow 0} \frac{\frac{1}{2h}d_h^2(u_h, u_{h,s})} {s^2} = \frac{c_0}{2} \sum_{i = 1}^P \int_M \left( \langle \mathbf{n}_i, \xi \rangle \right)^2 \rho \, |\nabla \chi_i|.
    \end{equation*}

    Using the definition of the gradient, the covariant derivative with respect to the Levi-Civita connection and the definition of $G_h$ one computes
    \begin{align}
      \langle \nabla_y \langle \nabla_x G_h(x,y), \xi(x)  \rangle_x, \xi(y) \rangle_y &= 
      \left(d_y \langle \nabla_x G_h(x,y), \xi(x)  \rangle_x \right)[\xi(y)] \\
      &= \frac{d}{dt}\Big|_{t=0} \langle \nabla_x G_h(x,c(t)), \xi(x)  \rangle_x \nonumber\\
      &= -\langle  D_t \Big|_{t=0} PT_{y \rightarrow x} \nabla_y G_h(x,c(t)), \xi(x)  \rangle_x \nonumber\\
      &= -\langle  PT_{y \rightarrow x} D_t \Big|_{t=0} \nabla_y G_h(x,c(t)), \xi(x)  \rangle_x \nonumber\\
      &= -\langle  \nabla_{\xi(y)} \big(\nabla G_h(x,\cdot)\big)(y), PT_{x \rightarrow y} \xi(x)  \rangle_y \nonumber\\
      &=\frac{1}{2h}G_h(x,y) \Big( \langle  \nabla_{\xi(y)} \nabla \frac{1}{2}\dist^2(x,\cdot), PT_{x \rightarrow y} \xi(x)  \rangle_y  \nonumber\\
      &\quad- \frac{1}{2}\Big\langle \frac{\exp_y^{-1}(x)}{\sqrt{h}}, \xi(y) \Big\rangle_y \Big\langle \frac{\exp_y^{-1}(x)}{\sqrt{h}} , PT_{x \rightarrow y}\xi(x) \Big\rangle_y \Big) \nonumber.
    \end{align}
    Plugging this into \eqref{eq:simplify_to_Hess} from \emph{Step 3.1} the limit \eqref{eq:second_variation_distance} is equal to
    \begin{align*}
      &\lim_{h \downarrow 0} \frac{1}{2\sqrt{h}} \int_M \int_{B_{\inj}(x)} \Big\langle  \nabla_{\xi(y)} \nabla \frac{\dist^2(x,\cdot)}{2}, PT_{x \rightarrow y} \xi(x)  \Big\rangle_y G_h(x,y) v_0(x,y)\\
      &\hspace{120pt} \cdot u_h(x) (1-u_h)(y) \dmx \dmy\\
      -      & \lim_{h \downarrow 0} \frac{1}{4\sqrt{h}} \int_M \int_{B_{\inj}(x)} \Big\langle \frac{\exp_y^{-1}(x)}{\sqrt{h}}, \xi(y) \Big\rangle_y \Big\langle \frac{\exp_y^{-1}(x)}{\sqrt{h}} , PT_{x \rightarrow y}\xi(x) \Big\rangle_y G_h(x,y) v_0(x,y)\\
      &\hspace{120pt} \cdot u_h(x) (1-u_h)(y) \dmx \dmy.
    \end{align*}
    With the same proof as in Lemma \ref{lem:cut_and_replace}\ref{part:cut_and_replace}) we can replace the domain of integration $B_{\inj}(x)$ by $B_r(x)$ with $r = \inj \wedge  \sqrt[4]{\frac{1}{2 d^3 C_R^2 }e^{-1 - d C_R \diam(M)^2}}$. Hence, due to Lemma \ref{lem:hessian_dist_equal_identity} and Lemma \ref{lem:cut_and_replace}\ref{part:nonscaled_vanish}) one can replace the Hessian of the squared distance by the identity
    \begin{align*}
      &\lim_{h \downarrow 0} \frac{1}{\sqrt{h}} \int_M \int_{B_{r}(x)} \Big\langle \xi(y) -  \nabla_{\xi(y)} \nabla \frac{\dist^2(x,\cdot)}{2}, PT_{x \rightarrow y} \xi(x)  \Big\rangle_y G_h(x,y) v_0(x,y) \\
      &\hspace{120pt} \cdot u_h(x) (1-u_h)(y) \dmx \dmy\\
      & = 0
    \end{align*}
    So one is now in the position to pass to the limit with $f(x,y,z_x, z_y) = \langle \xi(y), PT_{x \rightarrow y} \xi(x) \rangle_y - \langle z_y, \xi(y) \rangle_y \langle z_y , PT_{x \rightarrow y} \xi(x) \rangle_y$ in Lemma \ref{lem:conv_densities} to compute
    \begin{align*}
      \lim_{h \downarrow 0} h \lim_{s \rightarrow 0} \frac{\frac{1}{2h}d_h^2(u_h, u_{h,s})} {s^2} = \frac{1}{4}\sum_{i = 1}^P \int_M  \rho(x) \int_{T_x M}\!(2|\xi(x)|^2- \langle \xi(x), z \rangle_x^2 ) G_1(z) (\langle \mathbf{n}_i (x), z \rangle_x)_+ \, dz \,|\nabla \chi_i|.
    \end{align*}
    It remains to explicitly compute the $z$ integral for any $\mathbf{n} = \mathbf{n}_i(x) \in TM$ (in notation of \emph{Step 2.3})
    \begin{align*}
      &\int_{T_x M}(2|\xi(x)|^2- \langle \xi(x), z \rangle_x^2 ) G_1(z) (\langle \mathbf{n}, z \rangle_x)_+ \, dz \\
      &= 2c_0|\xi(x)|^2 + 2\int_{\tilde{z} \cdot \tilde{\mathbf{n}} > 0} \tilde{\xi}(x) \cdot \tilde{z} \ \tilde{\xi}(x)\cdot \nabla_{\tilde{z}} G_1^{\R^d}(\tilde{z})\ \tilde{\mathbf{n}} \cdot \tilde{z} \, d\tilde{z}\\
      &= 2c_0|\xi(x)|^2 - 2\int_{\tilde{z} \cdot \tilde{\mathbf{n}} > 0} |\tilde{\xi}(x)|^2 G_1^{\R^d}(\tilde{z}) \tilde{\mathbf{n}}\cdot\tilde{z} +  \tilde{\xi}(x)\cdot \tilde{z} G_1^{\R^d}(\tilde{z}) \tilde{\mathbf{n}} \cdot \tilde{\xi}(x)\, d\tilde{z}\\
      &= 2c_0\langle \xi(x) , \mathbf{n} \rangle_x^2.
    \end{align*}

    \emph{Step 4}: Conclusion.
    Utilizing exactly the previously shown \emph{Step 1} (applied to $(E,d,u) = (E_h^n, d_h, u_h(t))$), \emph{Step 2} and \emph{Step 3} (integrated in time) it follows
    \begin{align}\label{eq:befor_riesz}
      &\liminf_{h \downarrow 0} \frac{1}{2}\sum_{n = 1}^{\lfloor\frac{T}{h}\rfloor}\int_{(n-1)h}^{nh} |\partial E_h^n|^2(u_h(t)) \,dt  \\
       &\geq - \lim_{h \downarrow 0} \sum_{n = 1}^{\lfloor\frac{T}{h}\rfloor}\int_{(n-1)h}^{nh} \delta E_h^n(u_h, \xi) \,dt -\int_0^T h \lim_{s \rightarrow 0} \frac{\frac{1}{2h}d_h^2(u_h, u_{h,s})} {s^2} \,dt \nonumber\\
      &= -c_0 \sum_{i = 1}^P \int_0^T \int_M  \div(\xi) \rho + \langle \xi, \nabla \rho \rangle-\langle \mathbf{n}_i, \nabla_{\mathbf{n}_i} \xi \rangle \rho  -  \rho \langle\xi,\mathbf{n}_i
    \rangle \Lambda_i + \rho\langle \xi ,\mathbf{n}_i \rangle^2  \, |\nabla \chi_i| \,dt.\nonumber
    \end{align}
    As $\rho > 0$ and left-hand side of \eqref{eq:befor_riesz} is bounded it follows by Riesz' representation theorem in $L^2(|\nabla \chi_i| dt)$ the existence of the mean curvature vectors $\mathbf{n}_i H_i$ given by \eqref{def:mean_curvature}. Now, the definition \eqref{def:mean_curvature} of the mean curvature vectors applied to $\rho \xi$ yields 
    \begin{equation*}
      \sum_{i = 1}^P\int_M H_i \rho \langle \xi, \mathbf{n}_i\rangle |\nabla \chi_i| = \sum_{i = 1}^P\int \rho(\div (\xi) - \langle \mathbf{n}_i, \nabla_{\mathbf{n}_i} \xi \rangle) + \langle \nabla \rho, \xi \rangle - \langle \mathbf{n}_i, \nabla \rho \rangle \langle \xi, \mathbf{n}_i\rangle |\nabla \chi_i|
    \end{equation*}
    such that \eqref{eq:befor_riesz} turns into
    \begin{align*}
      &\liminf_{h \downarrow 0} \frac{1}{2}\sum_{n = 1}^{\lfloor\frac{T}{h}\rfloor}\int_{(n-1)h}^{nh} |\partial E_h^n|^2(u_h(t)) \,dt  \\
      &\geq -c_0 \sum_{i = 1}^P \int_0^T \int_M \rho \langle \xi, \mathbf{n}_i\rangle ( H_i - \Lambda_i + \langle \mathbf{n}_i,  \nabla \log \rho\rangle) + \rho\langle \xi ,\mathbf{n}_i \rangle^2  \, |\nabla \chi_i| \,dt.
    \end{align*}
    Taking the supremum over all smooth vector fields $\xi \in C^\infty_0((0,T), \Gamma(TM))$ implies 
    \begin{align}
       &\liminf_{h \downarrow 0} \frac{1}{2}\sum_{n = 1}^{\lfloor\frac{T}{h}\rfloor}\int_{(n-1)h}^{nh} |\partial E_h^n|^2(u_h(t)) \,dt \geq  \frac{c_0}{2} \sum_{i = 1}^P\int_0^T \int_M   \left( H_i - \Lambda_i + \langle \mathbf{n}_i,  \nabla \log \rho\rangle\right)^2 \rho \,|\nabla \chi_i|\,dt. \label{eq:ambiguous_lambda}
    \end{align}
    Note that \[\overline{H_i + \langle \mathbf{n}_i, \nabla \log \rho \rangle} := \frac{ \int_M ( H + \langle \mathbf{n}_i, \nabla \log \rho \rangle ) \rho |\nabla \chi_i| }{\int_M \rho |\nabla \chi_i|}\] is the minimizer among all $\Lambda_i$ of the $L^2$-distance appearing on the right-hand side term of \eqref{eq:ambiguous_lambda} and hence
    \begin{align*}
       &\liminf_{h \downarrow 0} \frac{1}{2}\sum_{n = 1}^{\lfloor\frac{T}{h}\rfloor}\int_{(n-1)h}^{nh} |\partial E_h^n|^2(u_h(t)) \,dt  \\&\geq  \frac{c_0}{2} \sum_{i = 1}^P \int_0^T  \int_M \left( H_i  + \langle \mathbf{n}_i,  \nabla \log \rho\rangle - \overline{H_i + \langle \mathbf{n}_i, \nabla \log \rho \rangle}\right)^2 \rho \,|\nabla \chi_i|\,dt.
    \end{align*}
    The result holds true for the whole sequence as the limit is now independent of the subsequence chosen in the compactness of the $\lambda_h$.
\end{proof}

\begin{proof}[Proof of Proposition \ref{prop:distance}]
Due to the definition \eqref{def:dist_h} of $d_h$, the semigroup property $p(h + h') = p(h) * p(h')$  and the symmetry of the heat kernel it holds
\begin{align*}
  \int_h^T \frac{1}{2h^2} d_h^2(\chi_h(t), \chi_h(t-h)) \, dt = \sum_{i = 1}^P\int_h^T \int_M \frac{1}{h \sqrt{h}}\Big|p\Big(\frac{h}{2}\Big)*(\chi_{h,i} - \chi_{h,i}(\cdot - h))\Big|^2 \, d\mu \, dt.
\end{align*} 
Here and in the following it will be denoted $\chi_{i,h}( \cdot - t)(x,s) = \chi_{i,h}(x, s-t)$ for any $t \in \R$.
This motivates the introduction of the non-negative (bounded) ``dissipation measures'' on $(0,T) \times M$
\begin{equation}
  \nu_i := \lim_{h \downarrow 0} \frac{1}{h \sqrt{h}} \Big|p\Big(\frac{h}{2}\Big)*(\chi_{h,i} - \chi_{h,i}(\cdot - h))\Big|^2 d\mu \, dt,
\end{equation}
where $p*$ defined in Chapter \ref{sec:notation} is the convolution in space but not in time.
Hence, to prove \eqref{eq:prop_distance} it will be sufficient to prove the localized version
\begin{equation}\label{eq:localized_dissipation_inequality}
  \frac{c_0}{2} |V_i|^2 \rho |\nabla \chi_i |dt \leq \nu_i
\end{equation}
Indeed, inequality \eqref{eq:localized_dissipation_inequality} implies that the limit of every converging subsequence of $\frac{1}{h \sqrt{h}} |p(\frac{h}{2})*(\chi_{h,i} - \chi_{h,i}(\cdot - h))|^2 d\mu \, dt$ is bounded from below by $ \frac{c_0}{2} |V_i|^2 \rho |\nabla \chi_i |dt$ which shows that the $\liminf$ in \eqref{eq:prop_distance} is bounded from below by the right quantity. 

To that aim assume from now on that we fixed an $i \in \{0,\dots, P\}$ and drop the $i$ in the notation for easier readability. Also, the measures $\mu \times dt$ will be left out in the notation; whenever a distributional equation or limit of a function $f$ is discussed, it is understood to refer to the measure $f \mu \,dt$.  Define the intermediate timescale
\begin{equation}
  \tau := \alpha \sqrt{h}, 
\end{equation}
which is a fraction of the characteristic spatial timescale for some fixed $\alpha \in (0, \infty)$. In the very end the limit $\alpha \downarrow 0$ will be taken so think about $\alpha$ small. The difference between two time steps is described by the variables
\begin{equation}\label{def:deltas}
  \delta \chi := \chi_{h} - \chi_{h}(\cdot - \tau) \text{ and } \delta \chi_{\pm}:= \max\{0,\pm\delta \chi \}.
\end{equation}

\emph{Step 1}: The mixed term vanishes distributionally 
\begin{align}\label{eq:vel_step1}
  \lim_{h \downarrow 0} \frac{1}{\sqrt{h}} \left( \delta \chi_+\, p(h) * \delta \chi_- + \delta \chi_-\, p(h) * \delta \chi_+ \right) = 0.
\end{align}
  Spelling out the convolution it holds in a distributional sense
  \begin{equation}\label{eq:mixedterm}
    \delta \chi_+ \,p(h)*\delta \chi_- + \delta \chi_-\, p(h) * \delta \chi_+ = \int_M \delta \chi_+\, p(h, \cdot,y) \delta \chi_-(y) + \delta \chi_+( y)\, p(h, \cdot,y) \delta \chi_-  \dmy.
  \end{equation}
  By definition of \eqref{def:deltas} and the fact that $\chi_h \in \{0,1\}$ one can rephrase $\delta \chi_+ = \chi_{h} (1-\chi_{h})(\cdot - \tau)$ and $\delta \chi_- = \chi_{h}(\cdot - \tau)(1 - \chi_{h})$. Using this one checks 
  \begin{align}\label{eq:deltapdelatm}
  &\delta \chi_+(x) \, \delta \chi_-(y) \leq (1- \chi_{h})(x, \cdot - \tau) \chi_h(y, \cdot - \tau)\nonumber,\\
    &\delta \chi_+(y) \,\delta \chi_-(x) \leq (1- \chi_{h})(x) \chi_h(y).
  \end{align}
  Thus, together with Lemma \ref{lem:cut_and_replace}\ref{part:cut_and_replace}) it is ok to pass to 
  \begin{align*}
        &\lim_{h \downarrow 0} \frac{1}{\sqrt{h}} \delta \chi_+ p(h)*\delta \chi_- + \delta \chi_- p(h) * \delta \chi_+\\
        &= \lim_{h \downarrow 0} \frac{1}{\sqrt{h}}\int_{B_{\inj}(\cdot)}\delta \chi_+( \cdot) p(h, \cdot,y) \delta \chi_-(y) + \delta \chi_+( y) p(h, \cdot,y) \delta \chi_-(\cdot)  \dmy.
  \end{align*}
  Now split the domain of integration for some fixed tangent vector $\bar{\mathbf{n}}$ into
  \begin{align}\label{eq:splitted}
 \int_{\{y:\langle \bar{\mathbf{n}}, \exp^{-1}(y) \rangle \geq0 \} \cap B_{\inj}} \delta \chi_+(\cdot) p(h,\cdot,y)  \delta \chi_-(y) + \delta \chi_+( y) p(h, \cdot,y) \delta \chi_-(\cdot) \dmy  
\end{align}
and the analogous counterpart including $\{y:\langle \bar{\mathbf{n}}, \exp^{-1}(y) \rangle \leq 0 \}\cap B_{\inj}$. The estimates \eqref{eq:deltapdelatm} yield now that \eqref{eq:splitted} is dominated by the distributional limit of 
\begin{align*}
\int_{\{y:\langle \bar{\mathbf{n}}, \exp^{-1}(y) \rangle \geq 0 \} \cap B_{\inj}}p(h,\cdot,y) \big( (1- \chi_{h})(\cdot, \cdot - \tau) \chi_h(y, \cdot - \tau) + (1- \chi_{h})(\cdot) \chi_h(y) \big) \dmy 
\end{align*}
After putting the constant time shift $- \tau$ onto the test function this is again equivalent to the weak limit of 
\begin{equation*}
  2\int_{\{y:\langle \bar{\mathbf{n}}, \exp^{-1}(y) \rangle \geq0 \} \cap B_{\inj}}p(h,\cdot,y)(1- \chi_{h})(\cdot) \chi_h(y)  \dmy
\end{equation*}
if it exists. Lemma \ref{lem:conv_densities} applied to the characteristic function of the closed set $\{y:\langle \bar{\mathbf{n}}, \exp^{-1}(y) \rangle \geq 0 \} $ yields that the weak limit is dominated by the measure 
\begin{align*}
  2\int_{\langle \bar{\mathbf{n}}, z \rangle \geq0  } G_1(z) \langle \mathbf{n}, z \rangle_- \,dz \,\rho|\nabla \chi| \, dt.
\end{align*}
Note that Lemma \ref{lem:conv_densities} yields here $\langle \mathbf{n},z \rangle_-$ due to the swapped roles of $x$ and $y$ which yield the opposite sign of $z$ as $z = \frac{\exp_x^{-1}(y)}{\sqrt{h}} = -PT_{y \rightarrow x}\frac{\exp_y^{-1}(x)}{\sqrt{h}} $.
The other integration domain $\{y:\langle \bar{\mathbf{n}}, \exp^{-1}(y) \rangle \leq 0 \}\cap B_{\inj}$ is taken care of similarly by swapping the roles of $\chi$ and $1 - \chi$ such that this part is dominated by $  2\rho\int_{\langle \bar{\mathbf{n}}, z \rangle \leq 0  } G_1(z) \langle \mathbf{n}_i, z \rangle_+ |\nabla \chi| \, dt\,dz$. 

To summarize, the weak limit $\lambda \geq 0$ of \eqref{eq:mixedterm} is bounded by 
\begin{align}\label{eq:bound_weak_limit}
  \lambda \leq 2 \left(\int_{\langle \bar{\mathbf{n}}, z \rangle \geq 0} G_1(z) \langle \mathbf{n}, z \rangle_- \, dz+ 
  \int_{\langle \bar{\mathbf{n}}, z \rangle \leq 0  } G_1(z) \langle \mathbf{n}, z \rangle_+ \, dz\right) \rho\,|\nabla \chi| \,dt.
\end{align}
Due to (see the computation in \emph{Step 2.3} in Proposition \ref{prop:metricslope})
\begin{equation*}
  c_0 =
\int_{\langle \bar{\mathbf{n}}, z \rangle \geq 0  } G_1(z) \langle \mathbf{n}, z \rangle_- \, dz = \int_{\langle \bar{\mathbf{n}}, z \rangle \leq 0} G_1(z) \langle \mathbf{n}, z \rangle_+ \, dz
\end{equation*}
it holds even $\lambda \leq 2 c_0 \|\rho\|_\infty |\nabla \chi|dt$, such that there exists a Radon-Nikodym derivative $\theta \in L^1(|\nabla \chi| dt)$ with $\lambda = \theta |\nabla \chi| dt$. Hence, the estimate \eqref{eq:bound_weak_limit} can be rewritten as 
\begin{align*}
   \theta \leq 2 \rho \int_{\langle \bar{\mathbf{n}}, z \rangle \geq 0 } G_1(z) \langle \mathbf{n}, z \rangle_- \, dz+ 
  2\rho \int_{\langle \bar{\mathbf{n}}, z \rangle \leq 0} G_1(z) \langle \mathbf{n}, z \rangle_+ \, dz \\ |\nabla \chi| dt  \text{--a.e. and for all } \bar{\mathbf{n}} .
\end{align*}
By an elementary separability argument one can change the order between the ``for all'' quantors, such that one is allowed to choose $\bar{\mathbf{n}} = \mathbf{n}$, resulting to $\theta \leq 0$ $|\nabla \chi|dt$ - a.e. This finishes \emph{Step 1}.

\emph{Step 2}: It holds distributionally 
\begin{equation}\label{eq:vel_step2}
  \limsup_{h \downarrow 0} \frac{1}{\sqrt{h}} \delta \chi \, p(h)*\delta \chi d\mu dt\leq \alpha^2 \nu.
\end{equation}
Again the $\mu \times dt$ is dropped in the notation. It is sufficient to prove the following two statements
\begin{align}
  \lim_{h \downarrow 0} \frac{1}{\sqrt{h}} \big(\delta \chi p(h)* \delta\chi - |p(\frac{h}{2}) * \delta \chi|^2 \big) = 0 \label{eq:vanish_commutator},\\
  \limsup_{h \downarrow 0} \frac{1}{\sqrt{h}} |p(\frac{h}{2}) * \delta \chi |^2- \alpha^2 \frac{1}{\sqrt{h}h} |p(\frac{h}{2})*(\chi_{h} - \chi_{h}(\cdot - h))|^2 \leq 0 \label{eq:boundalpha2}.
\end{align}
Start with \eqref{eq:boundalpha2} by assuming w.l.o.g.~that $\tau = Nh$ for $N \in \N$, i.e., $\alpha = N \sqrt{h}$, to apply the Cauchy-Schwarz inequality after telescoping
\begin{align*}
  \frac{1}{\sqrt{h}} |p(\frac{h}{2})*(\chi_{h} - \chi_{h}(\cdot - h))|^2 &\leq N \sum_{n = 0}^{N-1} \frac{1}{\sqrt{h}} |p(\frac{h}{2})*(\chi_{h}(\cdot - nh) - \chi_{h}(\cdot - (n+1)h))|^2\\
  &=  \alpha^2 \frac{1}{N} \sum_{n = 0}^{N-1} \frac{1}{h\sqrt{h}} |p(\frac{h}{2})*(\chi_{h}(\cdot - nh) - \chi_{h}(\cdot - (n+1)h))|^2.
\end{align*}
The right-hand side is an average of time shifts of $ \alpha^2 \frac{1}{\sqrt{h}h} |p(\frac{h}{2})*(\chi_{h} - \chi_{h}(\cdot - h))|^2$ and these time shifts are small as $Nh = O(\sqrt{h}) = o(1)$. Hence, the expression has the same weak limit as $\alpha^2 \frac{1}{\sqrt{h}h} |p(\frac{h}{2})*(\chi_{h} - \chi_{h}(\cdot - h))|^2$ which yields \eqref{eq:boundalpha2}.

For \eqref{eq:vanish_commutator} one takes a smooth test function $\zeta \in C_0^\infty(M \times (0,T))$ and uses the semigroup property of the heat kernel to compute 
\begin{equation*}
  \int_0^T \int_M \zeta (\delta \chi \, p(h)* \delta \chi - |p(\frac{h}{2})*\delta \chi|^2) \,d\mu \,dt= - \int_0^T \int_M \left([\zeta, p(\frac{h}{2})*]\delta \chi\right)  p(\frac{h}{2}) * \delta \chi \,d\mu \,dt,
\end{equation*}
with $[\zeta, p(\frac{h}{2})*]$ denoting the commutator of multiplying with $\zeta$ and convoluting with $p(\frac{h}{2})$. Due to \eqref{eq:boundalpha2} the sequence of measures $\frac{1}{\sqrt{h}} |p(\frac{h}{2})*\delta \chi|^2$ is bounded; hence, it is enough to show 
\begin{equation}\label{eq:comm_calc} 
  \lim_{h \downarrow 0} \frac{1}{\sqrt{h}} \int_0^T\int_M |[\zeta, p(\frac{h}{2})*]\delta \chi|^2 \, d\mu \,dt= 0.
\end{equation}
  Spelling out the commutator and using $|\delta \chi| \leq 1$ yields
\begin{align*}
 |( [\zeta, p(\frac{h}{2})*]\delta \chi )(x,t)| &= \left|\int_M p(\frac{h}{2}, x,y) (\zeta(x,t) - \zeta(y,t)) \delta \chi(y,t) \dmy\right|\\
 &\leq \|\nabla \zeta\|_\infty \int_M  p(\frac{h}{2}, x,y) \dist(x,y)\dmy.
\end{align*}
By the Gaussian bound \eqref{eq:gaussian_one}, the scaling of balls \eqref{eq:doubling_prop}, and $\sup_x x e^{-x} < \infty$ it holds 
\begin{equation*}
  \frac{\dist(x,y)}{\sqrt{h}}p(\frac{h}{2},x,y) \lesssim\frac{\dist(x,y)}{\sqrt{h}} G_{C_4\frac{h}{2 }}(x,y)\lesssim G_{C_4h}(x,y)
\end{equation*}
such that together with the Gaussian bound \eqref{eq:gaussian_one} which implies $\int_M G_{C_4h} \lesssim \int_M p({\frac{C_4}{C_2}h}) = 1$ it holds
\begin{align*}
  \frac{1}{\sqrt{h}} \int_0^T \int_M |[\zeta, p(\frac{h}{2})*]\delta \chi|^2 \, d\mu \,dt &\lesssim \sqrt{h}\|\nabla \zeta\|^2_\infty \int_0^T \int_M\left( \int_M  G_{C_4h}(x,y) \dmy\right)^2 \dmx \,dt\\
  &\lesssim \sqrt{h}\|\nabla \zeta\|^2_\infty T.
\end{align*}
 Equation \eqref{eq:vanish_commutator} is verified, as the right-hand side is vanishing.

 \emph{Step 3}: For any given unit vector field $\bar{\mathbf{n}} \in \Gamma (TM), |\bar{\mathbf{n}}| = 1 \ \mu$-a.e.\ and $\bar{V} \in (0, \infty)$ it holds in a distributional sense
 \begin{align}\label{eq:vel_step3}
  \limsup_{h \downarrow 0} \left(\frac{1}{\sqrt{h}} (\delta \chi_+ p(h)*(1- \delta \chi_+) + \delta \chi_- p(h)* (1- \delta \chi_-) - 2 \int_{\langle z, \bar{\mathbf{n}} \rangle > \alpha \bar{V}} G_1(z)\,dz|\delta \chi| )\right)d\mu \,dt \\
  \leq 2 \int_{0 \leq \langle z, \bar{\mathbf{n}} \rangle < \alpha \bar{V}} G_1(z) |\langle z, \mathbf{n} \rangle | \,dz \,\rho |\nabla \chi| \,dt.\nonumber
 \end{align}
 Again $d\mu \times dt$ is dropped in notation. The proof of this step can be split in three parts:
 \begin{enumerate}[1)]
  \item  The left-hand side is substituted by
  \begin{align}\label{eq:com_vanish}
    \lim_{h \downarrow 0} \frac{1}{\sqrt{h}} (\delta \chi_{\pm}) p(h)* (1- \delta \chi_{\pm}) = \lim_{h \downarrow 0} \frac{1}{2\sqrt{h}}\int_{B_{\inj}(\cdot)} p(h,\cdot, y)|\delta \chi_\pm - \delta \chi_\pm (y)| \dmy .
  \end{align}

  \item The integrand can then be estimated in two ways 
  \begin{align}\label{eq:time_space}
    |\delta \chi_+(x) - \delta \chi_+(y)| + |\delta \chi_- - \delta \chi_-(y) |\leq \begin{cases}
      |\chi_{h}(x) - \chi_{h}(y) | + |\chi_{h}(x, \cdot - \tau ) - \chi_{h}(y, \cdot - \tau)| \\
      |\delta \chi(x)| + |\delta \chi(y)|.
    \end{cases}
  \end{align}
Denote for any $x,y \in M, z_x \in T_x M, z_y \in T_y M$ the in $x$ and $y$ antisymmetric function
 \begin{align*}
 f(x,y, z_x, z_y) = \frac{1}{2}(\langle z_x, \bar{\mathbf{n}}(x) \rangle_x - \langle z_y, \bar{\mathbf{n}}(y) \rangle_y ).
 \end{align*} 
 Then the first ``space-like'' inequality is used on the set 
 \begin{equation*}
   I := \{y:|f(\cdot, y, \exp_{\boldsymbol{\cdot}}^{-1}(y), \exp_y^{-1}(\cdot)) | \leq \tau \bar{V} \cap B_{\inj}(\cdot)\}
 \end{equation*}
while the second ``time-like'' inequality is used on the complementary set 
\begin{equation*}
O := \{y: |f(\cdot, y, \exp_{\boldsymbol{\cdot}}^{-1}(y), \exp_y^{-1}(\cdot)) | > \tau \bar{V} \cap B_{\inj}(\cdot)\}.
\end{equation*}
  \item Finally, the following statements are applied to the respective quantities
  \begin{align}
    \limsup_{h \downarrow 0} &\frac{1}{2\sqrt{h}} \int_{I} p(h,\cdot, y) ( |\chi_{h}(\cdot) - \chi_{h}(y) |  + |\chi_{h}( \cdot - \tau ) - \chi_{h}(y, \cdot - \tau)| \dmy \label{eq:bound_space_like} \\
    &\leq 2 \int_{0 \leq \langle z, \bar{\mathbf{n}} \rangle < \alpha \bar{V}}G_1(z) |\langle z, \mathbf{n} \rangle| \rho |\nabla \chi| \,dz\nonumber\\
    \lim_{h \downarrow 0}& \frac{1}{2\sqrt{h}}\Big( \int_{O} p(h,\cdot,y) (|\delta \chi| + |\delta \chi(y)|) \dmy 
    - 2\int_{ \langle z, \bar{\mathbf{n}} \rangle  \geq \alpha \bar{V}} G_1(z) \, dz\, |\delta \chi|  \Big) = 0. \label{eq:lim_time_like}
  \end{align}
 \end{enumerate}

 Let's start with the proof of \eqref{eq:com_vanish} for which one first realizes 
 \begin{equation*}
  \lim_{h \downarrow 0} \frac{1}{\sqrt{h}} (\delta \chi_{\pm}) p(h)* (1- \delta \chi_{\pm}) = \lim_{h \downarrow 0} \frac{1}{\sqrt{h}} \int_{B_{\inj}}(\delta \chi_{\pm}) p(h, \cdot, y) (1- \delta \chi_{\pm})(y) \dmy
 \end{equation*}
 by Lemma \ref{lem:cut_and_replace}\ref{part:cut_and_replace}) after using the bound $\delta \chi_{\pm}(x)(1-\delta \chi_\pm)(y) \leq \chi_{h}(x)(1-\chi_{h})(y) + \chi_{h}(y, \cdot - \tau)(1-\chi_{h})(x, \cdot - \tau)$. Furthermore, with a similar reasoning as for \eqref{eq:comm_calc} the commutator vanishes such that
 \begin{equation*}
  \lim_{h \downarrow 0} \frac{1}{\sqrt{h}} \int_{B_{\inj}}\!\!(\delta \chi_{\pm}) p(h, \cdot, y) (1- \delta \chi_{\pm})(y) \dmy = \lim_{h \downarrow 0} \frac{1}{\sqrt{h}} \int_{B_{\inj}}\!\!(1- \delta \chi_{\pm})p(h, \cdot, y)(\delta \chi_{\pm})(y)   \dmy.
 \end{equation*}
 To summarize, the two previous identities together with $|\chi - \tilde{\chi}| = \chi (1 - \tilde{\chi}) + \tilde{\chi}(1 - \chi)$ for any $\chi, \tilde{\chi} \in \{0,1\}$ yield 
 \begin{align*}
  &\lim_{h \downarrow 0} \frac{1}{\sqrt{h}} (\delta \chi_{\pm}) p(h)* (1- \delta \chi_{\pm}) \\
  &= \frac{1}{2} \frac{1}{\sqrt{h}} \int_{B_{\inj}}(\delta \chi_{\pm}) p(h, \cdot, y) (1- \delta \chi_{\pm})(y) + (1- \delta \chi_{\pm})p(h, \cdot, y)(\delta \chi_{\pm})(y) \dmy 
  \\
  &= \frac{1}{2} \frac{1}{\sqrt{h}} \int_{B_{\inj}}(\delta \chi_{\pm}) p(h, \cdot, y) |\delta \chi_{\pm} - \delta \chi_{\pm}(y)|\dmy 
 \end{align*}
 which is precisely \eqref{eq:com_vanish}.

 The proof of \eqref{eq:time_space} is equivalent as in \cite[Proposition 1]{MR4385030}. The second inequality follows by the triangle inequality $|\delta \chi_\pm(x) - \delta \chi_\pm(y)| \leq \delta \chi_\pm(x) + \delta \chi_\pm(y)$ together with $\delta \chi_+ + \delta \chi_- = |\delta \chi|$. The first inequality is a consequence of 
 \begin{equation}
  |(a-a')_+ - (b-b')_+| + |(a-a')_- - (b-b')_-| \leq |a - b| +|a' - b'|,
 \end{equation}
 which can be verified by checking the two cases $(a-a')(b-b') \geq 0$ and $(a-a')(b-b') < 0$.

 Turning towards \eqref{eq:bound_space_like} one notes that the inequality follows from
 \begin{align}
  \limsup_{h \downarrow 0} \frac{1}{2\sqrt{h}} \int_{I} p(h,\cdot, y)  \big(|\chi_{h}(\cdot) - \chi_{h}(y) |  + |\chi_{h}( \cdot - \tau ) - \chi_{h}(y, \cdot - \tau)|\big) \dmy \nonumber\\
   \leq \int_{| \langle z, \bar{\mathbf{n}} \rangle| \leq \alpha \bar{V}}G_1(z) |\langle z, \mathbf{n} \rangle| \rho |\nabla \chi| \,dz \label{eq:bound_space_like_easy}
 \end{align}
 when using that $G_h$ is even such that
 \begin{equation*}
  \int_{0 \leq \langle z, \bar{\mathbf{n}} \rangle \leq \alpha \bar{V}}G_1(z) |\langle z, \mathbf{n} \rangle| \,dz = \int_{0 \geq \langle z, \bar{\mathbf{n}} \rangle \geq \alpha \bar{V}}G_1(z) |\langle z, \mathbf{n} \rangle| \,dz.
 \end{equation*}
But the bound \eqref{eq:bound_space_like_easy} is a consequence of Lemma \ref{lem:conv_densities}, tested with the characteristic function of the closed set $\{ f(x,y,\exp^{-1}_{\boldsymbol{\cdot}}(y), \exp^{-1}_y(\cdot)) \leq \tau \bar{V} \} \cap B_{\inj}$, after putting the time shift $\cdot - \tau$ on the test function and using $|\chi_{h, i}(x) - \chi_{h}(y)| = \chi_{h}(x)(1-\chi_{h})(y) + \chi_{h}(y) (1- \chi_{h})(x)$ as well as $f(x,x, z, -z) = \langle z, \bar{\mathbf{n}}(x) \rangle_x$.

The limit \eqref{eq:lim_time_like} follows from 
\begin{equation}\label{eq:diff_deltas}
  \lim_{h \downarrow 0} \frac{1}{2\sqrt{h}} \int_{|f(\cdot, y, \exp^{-1}(y), \exp_y^{-1}(\cdot))| > \tau \bar{V} \cap B_{\inj}(x)} p(h,\cdot, y) (|\delta \chi(y)| - |\delta \chi|) \dmy = 0
\end{equation}
and 
\begin{equation}\label{eq:change_to_normal_co}
   \lim_{h \downarrow 0} \frac{1}{\sqrt{h}} \int_{|f(\cdot, y, \exp_x^{-1}(y), \exp_y^{-1}(\cdot))| > \tau \bar{V} \cap B_{\inj}(x)} p(h,\cdot, y) |\delta \chi| \dmy - \frac{2}{\sqrt{h}} \int_{\langle z, \bar{\mathbf{n}} \rangle \geq \alpha \bar{V}} G_1(z) \,dz |\delta \chi|.
\end{equation}
To prove equation \eqref{eq:diff_deltas} one uses a test function $\xi \in C^{\infty}(M \times (0, T))$ and the symmetry of $f$ and $p(h)$ to estimate
\begin{align*}
  &\left|\frac{1}{2\sqrt{h}} \int_0^T \int_M \xi(x) \int_{|f(x, y, \exp^{-1}(y), \exp_y^{-1}(x))| > \tau \bar{V} \cap B_{\inj}(x)} p(h,x, y) (|\delta \chi(x)| - |\delta \chi(y)|) \, d\mu(y,x) \,dt \right|\\
  &=\frac{1}{2\sqrt{h}} \left|\int_0^T \int_M |\delta \chi|(x) \int_{|f(x, y, \exp^{-1}(y), \exp_y^{-1}(x))| > \tau \bar{V} \cap B_{\inj}(x)} p(h,x, y) (\xi(x) - \xi(y)) \, d\mu(y,x) \,dt\right|\\
  &\leq \int_0^T \frac{\|\xi\|_\infty}{2} \int_M |\delta \chi|(x) \int_{M} \frac{\dist(x,y)}{\sqrt{h}}p(h,x, y) \dmy \dmx \,dt.
\end{align*}
The right-hand side vanishes as the integral $\int_M\frac{\dist(x,y)}{\sqrt{h}}p(h,\cdot, y) \dmy$ is bounded due to the Gaussian bounds \eqref{eq:gaussian_one} (compare with the proof of Lemma \ref{lem:cut_and_replace}\ref{part:nonscaled_vanish} for more details) and $\int_M |\delta_\chi| \, d\mu$ vanishes by the compactness of Lemma \ref{lem:Compactness}. 

Turning towards \eqref{eq:change_to_normal_co} one first plugs in the asymptotic expansion \eqref{eq:asymptotic_exp} to identify the limits
\begin{align*}
     &\lim_{h \downarrow 0} \frac{1}{2\sqrt{h}} \int_{|f(\cdot, y, \exp^{-1}(y), \exp_y^{-1}(\cdot))| > \tau \bar{V} \cap B_{\inj}(\cdot)} p(h,\cdot, y) |\delta \chi| \dmy \\
     = &\lim_{h \downarrow 0} \frac{1}{2\sqrt{h}} \int_{|f(\cdot, y, \exp^{-1}(y), \exp_y^{-1}(\cdot))| > \tau \bar{V} \cap B_{\inj}(\cdot)} G_h(\cdot, y) v_0(\cdot, y)|\delta \chi| \dmy.
\end{align*}
The right-hand side limit is equivalent to 
\begin{align*}
  &\lim_{h \downarrow 0} \frac{1}{\sqrt{h}} \int_{|f(\cdot, y, \exp^{-1}(y), \exp_y^{-1}(\cdot))| > \tau \bar{V} \cap B_{\inj}(\cdot)} G_h(\cdot, y) v_0(\cdot, y)|\delta \chi| \dmy\\
  &= \lim_{h \downarrow 0} \frac{1}{\sqrt{h}} \int_{|f(\cdot, x_{\sqrt{h}z}, z, - PT_{x \rightarrow x_{\sqrt{h}z}},z)| > \tau \bar{V} \cap B_{\inj}(0)} G_1(z) v_0(\cdot, x_{\sqrt{h}z} )\eta(x_{\sqrt{h}z})\, dz|\delta \chi| 
\end{align*}
after using the scaling of $G_h$, normal coordinates around $x$ and denoting $\eta(y) = \sqrt{\det(g)(y)}\rho(y)$. As $v_0$ and $\eta$ are smooth their product satisfies
\begin{equation*}
  |v_0(x, x_{\sqrt{h}z})\eta(x_{\sqrt{h}z}) - 1|  =  |v_0(x, x_{\sqrt{h}z})\eta(x_{\sqrt{h}z}) - v_0(x,x) \eta(0)| \lesssim \sqrt{h},
\end{equation*}
where $v_0(x,x) = \frac{1}{\rho(x)}$ and $\eta(x) = \rho(x)$ was used.
This yields together with the vanishing $\int_M |\delta_\chi| \, d \mu$ due to compactness of Lemma \ref{lem:Compactness} and the definition of $f$ that
\begin{align*}
&\lim_{h \downarrow 0} \frac{1}{\sqrt{h}} \int_{|f(\cdot, x_{\sqrt{h}z}, z, - PT_{x \rightarrow x_{\sqrt{h}z}}z)| > \tau \bar{V} \cap B_{\inj}(0)} G_1(z) v_0(\cdot, x_{\sqrt{h}z}) \eta(x_{\sqrt{h}z})\, dz|\delta \chi| \\
& = \lim_{h \downarrow 0} \frac{1}{\sqrt{h}} \int_{|\frac{1}{2}(\langle z, \bar{\mathbf{n}} \rangle + \langle z, PT_{x \rightarrow x_{\sqrt{h}z}} \bar{\mathbf{n}}(x_{\sqrt{h}z}) \rangle)| > \tau \bar{V}} G_1(z) \, dz|\delta \chi|. 
\end{align*}
As $G_1$ is even it holds
\begin{align*}
  \int_{\frac{1}{2}(\langle z, \bar{\mathbf{n}} \rangle + \langle z, PT_{x \rightarrow x_{\sqrt{h}z}} \bar{\mathbf{n}}(x_{\sqrt{h}z}) \rangle) > \tau \bar{V} } G_1(z) \, dz =\int_{\frac{1}{2}(\langle z, \bar{\mathbf{n}} \rangle + \langle z, PT_{x \rightarrow x_{\sqrt{h}z}} \bar{\mathbf{n}}(x_{\sqrt{h}z}) \rangle) < - \tau \bar{V} } G_1(z) \, dz
\end{align*}
and by the rotational symmetry of $G_1$ and the assumption $|\bar{\mathbf{n}}(y)| = 1$ a.e.~one furthermore checks
\begin{align*}
  \int_{\langle z, PT_{x \rightarrow x_{\sqrt{h}z}} \bar{\mathbf{n}}(x_{\sqrt{h}z}) \rangle > \tau \bar{V} } G_1(z) \, dz = \int_{\langle z,  \bar{\mathbf{n}}(x) \rangle> \tau \bar{V} } G_1(z) \, dz
\end{align*}
Putting everything together one concludes 
\begin{align*}
  &\lim_{h \downarrow 0} \frac{1}{\sqrt{h}} \int_{|f(\cdot, y, \exp^{-1}(y), \exp_y^{-1}(\cdot))| > \tau \bar{V} \cap B_{\inj}(\cdot)} p(h,\cdot, y) |\delta \chi| \dmy \\
  & = \lim_{h \downarrow 0} \frac{1}{\sqrt{h}} \int_{|\frac{1}{2}(\langle z, \bar{\mathbf{n}} \rangle + \langle z, PT_{x \rightarrow x_{\sqrt{h}z}} \bar{\mathbf{n}}(x_{\sqrt{h}z}) \rangle)| > \tau \bar{V}} G_1(z) \, dz|\delta \chi|\\
& = \lim_{h \downarrow 0} \frac{2}{\sqrt{h}} \int_{\langle z,  \bar{\mathbf{n}}(x) \rangle > \tau \bar{V} } G_1(z) \, dz |\delta \chi|.
\end{align*}

\emph{Step 4}: Conclusion. To get in position to apply the previous steps one uses the decomposition
\begin{align*}
  \frac{1}{\sqrt{h}}|\delta \chi| = \frac{1}{\sqrt{h}}(& \delta \chi p(h)*\delta \chi
   + \delta \chi_+ p(h)*\delta \chi_- 
   + \delta \chi_- p(h)*\delta \chi_+\\
   + &\delta \chi_+ p(h)*(1-\delta \chi_+)
   + \delta \chi_- p(h)*(1-\delta \chi_-)) 
\end{align*}
in the form, using $2 \int_{\langle z, \bar{\mathbf{n}} \rangle \geq 0} G_1(z) \,dz=1$,
\begin{align*}
  2 \int_{0 \leq \langle z, \bar{\mathbf{n}} \rangle \leq \alpha \bar{V}} G_1(z) \,dz\frac{1}{\sqrt{h}}|\delta \chi| = \frac{1}{\sqrt{h}}(& \delta \chi p(h)*\delta \chi
   + \delta \chi_+ p(h)*\delta \chi_- 
   + \delta \chi_- p(h)*\delta \chi_+\\
   + &\delta \chi_+ p(h)*(1-\delta \chi_+)
   + \delta \chi_- p(h)*(1-\delta \chi_-) \\
   -& 2 \int_{\langle z, \bar{\mathbf{n}} \rangle > \alpha \bar{V}} G_1(z) \,dz |\delta \chi|). 
\end{align*}
By an elementary lower-semi-continuity argument it holds in a distributionally sense $\alpha \rho |\partial_t \chi| \leq \liminf_{h \downarrow 0} \frac{1}{\sqrt{h}} |\delta \chi|$ provided the right-hand side is a finite measure. Note that the $\rho$-factor pops up, as the distribution $\frac{1}{\sqrt{h}} |\delta \chi|$ is as always tested with respect to the measure $\mu$. Therefore, the results of the previous Steps in form of \eqref{eq:vel_step1}, \eqref{eq:vel_step2} and \eqref{eq:vel_step3} yield the distributional estimate
\begin{equation}\label{eq:conclusion_limit_estimate}
  2 \alpha \int_{0 \leq \langle z, \bar{\mathbf{n}} \rangle \leq \alpha \bar{V}} G_1(z) \, dz \rho \, |\partial_t \chi |\leq \alpha^2 \nu + 2\int_{0 \leq \langle z, \bar{\mathbf{n}}\rangle \leq \alpha \bar{V}} G_1(z) |\langle z, \mathbf{n} \rangle| \, dz\, \rho |\nabla \chi| \,dt,
\end{equation} 
which in particular shows that $\partial_t \chi$ is a measure. The limit $\bar{V} \uparrow \infty$, $\frac{1}{C} \leq \rho \leq C$ and $4c_0=\int G_1(z)  |\langle z, \mathbf{n} \rangle| \, dz$ (for more details see \eqref{eq:derivation_c0}) additionally yield $\alpha |\partial_t \chi| \leq \alpha^2 \nu + 4c_0 C^2 |\nabla \chi| dt$ which is divided by $\alpha$:
\begin{equation*}
  |\partial_t \chi| \leq \alpha \nu + \frac{4C^2}{\alpha} c_0 |\nabla \chi| dt.
\end{equation*}
Fixing a nullset of $|\nabla \chi| \times dt$ and then sending $\alpha \downarrow 0$ implies that $\partial_t \chi$ is absolutely continuous with respect to $|\nabla \chi|$. Hence, there exists a $V \in L^1(|\nabla \chi| dt)$ such that \eqref{eq:V_dist_vel} holds distributionally with respect to the volume measure.

As $|\partial_t \chi| = |V| |\nabla \chi| dt$ is absolutely continuous with respect to $|\nabla \chi| dt$, the estimate \eqref{eq:conclusion_limit_estimate} also holds when replacing $\nu$ by $\nu'$ the absolutely continuous part of $\nu$. Denote $\nu' = \theta |\nabla \chi| dt$ to rewrite
\begin{align*}
  2 \alpha \int_{0 \leq \langle z, \bar{\mathbf{n}} \rangle \leq \alpha \bar{V}} G_1 \,dz \,\rho|V|  \leq \alpha^2 \theta + 2 \int_{0 \leq \langle z, \mathbf{n} \rangle \leq \alpha \bar{V}} G_1(z) |\langle z, \bar{\mathbf{n}} \rangle| \, dz \rho  \quad \quad  |\nabla \chi | dt \text{-a.e.}
\end{align*}
As in \emph{Step 1}, by a separability argument one can choose $\bar{\mathbf{n}} = \mathbf{n}$. This together with the radial symmetry of $G_1$ yields for an orthonormal basis $\{e_i\}_{i = 1}^d$ of $T_x M$
\begin{equation*}
    2 \alpha \int_{0 \leq \langle z, e_1 \rangle \leq \alpha \bar{V}} G_1 \,dz \,\rho|V|  \leq \alpha^2 \theta + 2 \int_{0 \leq \langle z, e_1 \rangle \leq \alpha \bar{V}} G_1(z) \langle z, e_1 \rangle \, dz \rho  \quad \quad  |\nabla \chi | dt \text{-a.e.}
\end{equation*}
Using $\alpha' = \alpha \bar{V}$ and diving by $\alpha^2$ this turns with the notation of \emph{Step 2.3} of the proof of Proposition \ref{prop:metricslope} to
\begin{equation*}
  2 \frac{\bar{V}}{\alpha'} \int_{0 \leq  \tilde{z}_1  \leq \alpha' } G_1^{\R^d}(\tilde{z}) \,d\tilde{z}\, \rho|V|  \leq  \theta + 2 \frac{\bar{V}^2}{\alpha'^2} \int_{0 \leq  \tilde{z}_1\leq \alpha'} G_1^{\R^d}(\tilde{z}) \tilde{z}_1 \, d\tilde{z} \ \rho  \quad \quad  |\nabla \chi | dt \text{-a.e.}
\end{equation*}
One concludes by using the limits (which follow by factorizing  $G_1^{\R^d}$ into the one and $(d-1)$ standard Gaussian)
\begin{align*}
  \lim_{\alpha' \downarrow 0} \frac{1}{\alpha'} \int_{0 \leq  \tilde{z}_1 \leq \alpha'} G_1^{\R^d}(\tilde{z}) \,d\tilde{z} = G_1^{\R^d}(0) = \frac{c_0}{2},\\
    \lim_{\alpha' \downarrow 0} \frac{1}{\alpha'^2} \int_{0 \leq  \tilde{z}_1  \leq \alpha' }\tilde{z}_1 G_1^{\R^d}(\tilde{z}) \,d\tilde{z} = \frac{d^2}{d^2 \tilde{z}_1} G^{\R^1}_1(0) = \frac{c_0}{4},
\end{align*}
to rewrite 
\begin{equation*}
   c_0\bar{V} |V| \rho \leq  \theta +  \frac{c_0}{2} \bar{V}^2\rho  \quad \quad  |\nabla \chi | dt \text{-a.e.}
\end{equation*}
By another separability argument one may use $\bar{V} = |V|$ such that the above yields \eqref{eq:localized_dissipation_inequality} in form of $\frac{c_0}{2} |V|^2 \rho \leq \theta$. This concludes the proof of Proposition \ref{prop:distance}.
\end{proof}

\section*{Acknowledgments}
This project is funded by the Deutsche Forschungsgemeinschaft (DFG, German Research Foundation) under Germany's Excellence Strategy EXC 2181/1 - 390900948 (the Heidelberg STRUCTURES Excellence Cluster) and the Research Training Group 2339 IntComSin -- Project-ID 321821685. I want to thank my supervisor Tim Laux who gave me so much insightful feedback in this fascinating project.

\frenchspacing
\bibliographystyle{abbrv}
\bibliography{volumeMBO}
  \end{document}